\documentclass[12pt]{amsart}

\usepackage{amssymb}
\usepackage{amsthm}
\usepackage{amsmath}
\usepackage{amscd}

\usepackage{verbatim}
\usepackage[all]{xy}
\usepackage{diagbox}
\usepackage{dsfont}
\usepackage{fancyhdr}

\numberwithin{equation}{section}

\theoremstyle{plain}
\newtheorem{theorem}{Theorem}[section]
\newtheorem{corollary}[theorem]{Corollary}
\newtheorem{lemma}[theorem]{Lemma}
\newtheorem{proposition}[theorem]{Proposition}

\newtheorem{conjecture}[theorem]{Conjecture}
\theoremstyle{definition}
\newtheorem{definition}[theorem]{Definition}
\newtheorem{remark}[theorem]{Remark}

\theoremstyle{remark}

\newcommand{\OO}{\mathcal O}
\newcommand{\A}{\mathbb{A}}
\newcommand{\R}{\mathbb{R}}

\newcommand{\Q}{\mathbb{Q}}
\newcommand{\Z}{\mathbb{Z}}
\newcommand{\N}{\mathbb{N}}
\newcommand{\C}{\mathbb{C}}

\renewcommand{\H}{\mathbb{H}}
\newcommand{\F}{\mathbb{F}}
\newcommand{\D}{\mathbb{D}}

\newcommand{\kzxz}[4]{\left(\begin{smallmatrix} #1 & #2 \\ #3 & #4\end{smallmatrix}\right) }

\newcommand{\vol}{\operatorname{vol}}
\newcommand{\tr}{\operatorname{tr}}

\newcommand{\sgn}{\operatorname{sgn}}

\newcommand{\Cl}{\operatorname{Cl}}

\newcommand{\GSpin}{\operatorname{GSpin}}

\newcommand{\Mp}{\operatorname{Mp}}

\newcommand{\Char}{\operatorname{char}}

\newcommand{\Spec}{\operatorname{Spec}}

\newcommand{\End}{\operatorname{End}}

\newcommand{\GL}{\operatorname{GL}}

\newcommand{\Div}{\operatorname{Div}}

\newcommand{\cha}{\operatorname{char}}

\newcommand{\Nor}{\operatorname{Nor}}

\newcommand{\ord}{\operatorname{ord}}

\newcommand{\Gspin}{\operatorname{GSpin}}

\newcommand{\ff}{\hbox{if }}
\newcommand{\SL}{\operatorname{SL}}
\newcommand{\Diff}{\operatorname{Diff}}
\newcommand{\B}{\mathbb B}

\newcommand{\CH}{\operatorname{CH}}
\newcommand{\GS}{\operatorname{GS}}

\newcommand{\MW}{\operatorname{MW}}

\begin{document}
\setcounter{section}{-1}

\title{On the arithmetic inner product formula and central derivatives of L-functions}

\date{\today}
\author[Tuoping Du]{Tuoping Du}
\address{School of Mathematics and Physics, North China  Electric Power University, Beijing, 102206, P.R. China}
\email{dtp1982@163.com}

\begin{abstract}

This research provides  a formal definition of the arithmetic theta lift for cusp forms of weight $3/2$ and establishes the arithmetic inner product formula, thereby completing the Kudla program on modular curves. This formula is demonstrated to be equivalent to the Gross-Zagier formula, for which we provide a new proof.

 Additionally, the authors introduce a new arithmetic representation for the central  derivatives of L-functions associated with cusp forms of higher weight. Although this representation differs from Zhang’s higher weight Gross-Zagier formula, it maintains a significant connection to it. This study also proposes a conjecture indicating that the vanishing of derivatives of L-functions is determined by the algebraicity of the coefficients of harmonic weak Maass forms.

A consistent approach is employed to study both parts of this work.

\end{abstract}

\dedicatory{}


\thanks{}

\maketitle

\tableofcontents

\section{Introduction}
There are primarily two approaches to study   the derivatives of L-functions in arithmetic geometry. The first approach involves representing the central derivative as the height of certain cohomological trivial cycles, such as  the Gross-Zagier formula \cite{GZ},  Zhang's higher weight Gross-Zagier formula \cite{Zhang}, and the higher derivatives of certain L-functions over function fields\cite{YunZhang}. The second approach  involves expressing the central derivative  as the self-intersection number of the arithmetic theta lift, a framework known as the Kudla program, which has recently been generalized to  the arithmetic inner product formula. The first case involving Shimura curves was introduced by Kudla, Rapoport, and Yang in their works \cite{KRYComp}\cite{KRYBook}. Recently, Li and Liu have made significant contributions \cite{LL1}\cite{LL2} to the study of higher-dimensional unitary Shimura varieties.

One  aim of this paper  is to formulate the arithmetic inner product formula and to provide a new proof of the Gross-Zagier formula on the modular curve $X_0(N)$, while also demonstrating their equivalence.
$$
\begin{tabular}{|c|*{2}{c|}c|}
\hline
\diagbox[dir=NW]{Type}{Curve} & Modular curve & Division Shimura curve \\   \hline
  Gross-Zagier& Gross-Zagier \cite{GZ} &  Yuan-Zhang-Zhang \cite{Yuanzhangzhang} \\   \hline
   AIPF&  \textbf{Theorem \ref{theoarithprod}}& Kudla-Rapoport-Yang\cite{KRYBook} \\
   \hline
 \end{tabular}
$$
Here we denote the arithmetic inner product formula as AIPF.

In addition, this study will focus on the arithmetic representation of the derivatives of L-functions  for higher weight, which can be represented  by the height of  Heegner cycles \cite{Zhang}.
 In this work, we will present a new representation. The approach taken in this research relies on the Borcherds lift \cite{Bo} and  CM values  \cite{BO} \cite{BY}.

Let $(V,Q)$ be a quadratic space over $\Q$ with signature $(n, 2)$,
and  let $L\subset V$  denote an even lattice. We denote the dual  of $L$  as $L^\sharp$ and define  $\Gamma'=\Mp_2(\Z)$
as the full inverse image of $\SL_2(\Z)$ within the two-fold metaplectic cover of $\SL_2(\R)$.
Consequently, there exists  a Weil representation $\rho_L$ acting on the finite-dimensional space $\C[L^\sharp /L]$,
with the standard basis represented as $\{e_\mu | \mu \in L^\sharp /L \}$.

We denote the space of  harmonic
weak Maass forms of weight $k$  with the representation $\rho_L$ as  $H_{k, \rho_L}$.
There exist  differential operators $$\xi_{k}: H_{k, \rho_L } \rightarrow S_{2-k, \overline{\rho}_L},$$ 
  Maass raising operator $R_k$ and lowering operator $L_k$.

Borcherds studied the theta lift of weakly modular forms \cite{Bo}, which is  generalized to the harmonic
weak Maass forms  $H_{k, \rho_L}$ by Bruinier \cite{Br}. In this paper, we will extend these  works to the general cases.


Let $\Delta$  be a fundamental discriminant and $r$  be an integer such that $\Delta\equiv r^2 (\mod 4N)$.

If $\Delta= r=1$,  we will omit  the indices $\Delta$ and $r$ in this paper.

For  $n=1$, we define
  \begin{equation}
V=\{w=\kzxz
    {w_1} { w_2 }
      {w_3}   {-w_1}  \in M_{2}(\Q) : \tr(w) =0 \},
\end{equation}
with the quadratic form $Q(w)=N \det{w}$,  and  we define
 \begin{equation}
L=\big\{w =\kzxz {b}{\frac{a}{N}}{c}{-b}
   \in M_{2}(\Z) :\,   a, b, c \in \Z \big\}\subset V,
\end{equation}
to be the lattice in $V$. 
For any $\mu \in L^{\sharp}/L$, and a positive rational number $m \in \sgn(\Delta)Q(\mu) +\Z$, the twisted Heegner divisor
$Z_{\Delta, r}(m, \mu) $  in modular curve $X_{0}(N)$ is defined by Bruinier and Ono in \cite{BO}, which is a generalization of the Heegner divisor introduced by Gross and Zagier in \cite{GZ}.

For any negative definite two-dimensional subspace $U \subseteq V$, there exists a CM cycle $Z(U)$ in the Shimura variety.
 Indeed, $Z(U)$ consists of two copies of the ideal class group of an imaginary quadratic field.
Then the genus character $\chi_\Delta$ can be used to define the twisted CM divisor $Z_{\Delta, r}(U)$.


The Gross-Zagier  formula   has been extended to
eigenforms $G$ of weight $2\kappa$, for any $\kappa \geq 1$.
Deligne conjectured and Zhang proved that the derivative
$L'(G, \chi, \kappa)$ corresponds to  to the heights
of Heegner cycles. This result is now referred to as the higher weight Gross-Zagier formula \cite{Zhang}.

 Let  $\mathcal{Y}=\mathcal{Y}_\kappa(N)$  be the  Kuga-Sato variety.
 For any CM point $x$ on $X_0(N)$, we denote $S_\kappa(x)$ as the  Heegner cycle over $x$ in  $\mathcal{Y}$.
 The class of this cycle in $H^{2\kappa}(Y(\C), \C)$ is zero.

We define the twisted Heegner cycles as follows
\begin{equation}\label{defheer}
Z_{\Delta, r, \kappa}(m, \mu)=m^{\frac{\kappa-1}{2}}\sum_{x \in Z_{\Delta, r}(m, \mu)}\chi_\Delta(x)S_\kappa(x).
\end{equation}
In a similar manner, we define the Heegner cycle $ Z_{\Delta, r, \kappa}(U)$ over $Z_{\Delta, r}(U)$.

For $f \in H_{3/2-\kappa, \bar{\rho}_L}$,  we denote
\begin{equation}
Z_{\Delta, r, \kappa}(f)=\sum_{m>0} c^+(-m, \mu)Z_{\Delta, r, \kappa}(m, \mu).
\end{equation}
where $c^+(m, \mu)$ are Fourier coefficients of its holomorphic part.

 The Gillet-Soul\'e  height $\langle~,~ \rangle_{GS}$ extends  the N\'eron-Tate height
 $\langle~,~  \rangle_{NT}$,
as    introduced by Gillet and Soul\'e \cite{GS}. Building on this work, Zhang developed  the global height pairing $\langle~, ~\rangle$ for  higher Heegner cycles, which can be expressed as a sum of local height pairings.

A primary aim of this study is to calculate the following intersection number
$$\langle Z_{\Delta, r, \kappa}(f), Z_{\Delta, r, \kappa}(U)\rangle.$$
We will demonstrate that the archimedean component of this intersection number is represented by the Borcherds lift.

For any $f \in H_{3/2-\kappa, \rho_L}$, when $\kappa=2j+1$ is odd, we define the higher twisted regularized theta lift as follows
\begin{equation}\label{dehigh}
\Phi_{\Delta, r}^j(z, h, f)
=   \frac{1}{(4\pi)^j}\int^{reg}_{\mathcal F}\langle R_{3/2-\kappa} ^jf(\tau), \Theta_{\Delta, r}(\tau, z, h) \rangle d\mu(\tau).
\end{equation}
In the case where $\kappa$ is even, we substitute the twisted Siegel theta function $\Theta_{\Delta, r}(\tau, z, h)$ with the  Millson theta function
$\Theta^{\mathcal M}_{\Delta, r}(\tau, z, h)$.
Indeed, these lifts are   higher Green functions.

When $\Delta=1$, Bruinier and Yang \cite{BY}, as well as  Bruinier, Ehlen and Yang \cite{BEY} investigated the CM value
$\Phi^j(Z(U), f)$. This value can be interpreted  as the archimedean component
 of the  global height pairing and  provides  the  derivatives of L-functions.

For any newform $G \in S_{2\kappa}^{new}(N):=S_{2\kappa}^{new}(\Gamma_0(N)) $, let $Z_{\Delta, r, \kappa}^G(m, \mu)$ denote the $G$--isotypical  component of the Heegner cycle. Now we denote  the Petersson norm  as $\parallel G\parallel:=\sqrt{\langle G,~ G \rangle_{Pet}}$.

The following theorem corresponds to Gross-Zagier formula \cite{GZ} when  $\kappa=1$, and to  Zhang's higher  Gross-Zagier formula \cite{Zhang} when $\kappa>1$.
\begin{theorem}[Gross-Zagier-Zhang formula]\label{thehigher} 
For any normalized newform $G \in S_{2\kappa}^{new}(N)$, one has
\begin{equation}
\langle Z_{\Delta, r, \kappa}^G(m, \mu), Z_{\Delta, r, \kappa}^G(m, \mu)\rangle=
\frac{(2\kappa-2)!\sqrt{|D|} }{2^{4\kappa-2}\pi^{2\kappa}\parallel G\parallel^2}m^{\kappa-1}
L'_K(G, \chi, \kappa),\nonumber
\end{equation}
where $\chi$ denotes  the genus character associated with the decomposition of the fundamental discriminant  as $$D=-4N m|\Delta|=\Delta D_0.$$
\end{theorem}
The term  $m^{\kappa-1}$ appears  because we multiply by $m^{\frac{\kappa-1}{2}}$ in the definition of $Z_{\Delta, r, \kappa}^G(m, \mu)$, as shown in  equation (\ref{defheer}).

It is straightforward to observe that for $\kappa=1$,
$$Z_{\Delta, r, \kappa}^G(m, \mu)=0 \Longleftrightarrow L'_K(G, \chi, \kappa)=0 .$$
However, for $\kappa>1$, this reason cannot be applied because there is no established non-degeneracy of  the pairings  $\langle~, ~\rangle$.

Furthermore,  determining whether a cycle is trivial  in the Chow group is quite hard. Consequently,  it is essential  to provide a new representation  of derivatives of L-functions for higher weight $\kappa$. We will study  cusp forms, extending our focus beyond merely newforms. The proposed approach is outlined in the  subsequent sections of this introduction.

\subsection{Main results}

We denote
\begin{equation}
  \tilde{\rho}_L=\left\{
                   \begin{array}{ll}
                     \rho_L, & \hbox{$\Delta >0$;} \\
                     \bar{\rho}_L, & \hbox{$\Delta <0$.}
                   \end{array}
                 \right.
 \end{equation}
For a normalized newform $G \in S_{2\kappa}^{new}(N)$, we denote by  $F_G$ the total real number field generated by the eigenvalues of $G$. There exists a  newform  $g \in S^{new}_{\frac{1}{2}+\kappa, \tilde{\rho}_L}$  that corresponds to $G$ under the Shimura
correspondence $Sh_{m_0,\mu_{0}}$ as given in \cite{GKZ}. We normalize $g$ such that all its coefficients are contained in $F_G$. Furthermore, there exists a function
$f \in H_{_{3/2-\kappa}, \bar{\tilde{\rho}}_L}(F_G)$ such that
$$\xi_{3/2-\kappa}(f)=\parallel g\parallel^{-2} g ,$$
where the coefficients of the principal part of $f$  are also contained in  $F_G$.
Roughly speaking, up to  normalization,  this is represented by the following map
\begin{equation}\label{mapone}
\xymatrix@C=1.5cm{
  H_{3/2-\kappa, \bar{\tilde{\rho}}_L} \ar[r]^{\xi_{3/2-\kappa}} & S_{\frac{1}{2}+\kappa, \rho_L} \ar[r]^{Sh_{m_0,\mu_{0}}} & S_{2\kappa}(N) }.
  \end{equation}

\begin{conjecture}\label{conmain}
Assume that $f \in H_{_{3/2-\kappa}, \bar{\tilde{\rho}}_L}(F_G)$ is given as above.
Then the following statements are equivalent:
  \begin{enumerate}
    \item $ L'(G, \chi_\Delta, \kappa)=0$.
    \item $Z_{\Delta, r, \kappa}(f)$ vanishes in the Chow group $\CH^\kappa(\mathcal{Y})$.
\item $c^+(|\Delta|, \mu_r) \in F_G.$
        \end{enumerate}
\end{conjecture}
\begin{remark}

\begin{enumerate}
  \item When $\kappa=1$, this conjecture has been  proved by Gross and Zagier \cite{GZ}, Borcherds  \cite{BoDuke}, Bruinier and Ono \cite{BO}.
  \item The implication $$(2)\Rightarrow (1)$$ follows from  Theorem \ref{maintheo}.
   \item
This conjecture implies that for \textbf{ almost all} newform $G$,
the order of the L-function 
$$\ord_{s=\kappa} L(G, \chi_\Delta, s)\leq 1.$$
The higher order ($\kappa >1$) relies on the algebraicity of   $c^+(|\Delta|, \mu_r)$.
\item Alfes, Bruinier and Schwagenscheidt  constructed specific modular forms \cite{ABS}, whose Fourier coefficients are included in  $i\pi\sqrt{\Delta}F_G$ if and only if $c^+(|\Delta|, \mu_r) \in F_G$. We  propose the conjecture that the vanishing of Heegner cycles is determined by the algebraicity of the Fourier coefficients.
\end{enumerate}
\end{remark}

In the work \cite{BY}, Bruinier and Yang  put forth a conjecture concerning  the derivatives of L-functions, proposing  that these derivatives can be  expressed in terms of the intersections of CM cycles. Following this, Brunier, Howard, and Yang made a significant contribution to unitary Shimura varieties \cite{BHY}.   In this paper, we propose a twisted version of  this conjecture, which is relevant to modular curves and Kuga-Sato varieties, as detailed below.

We have the following result in Theorem \ref{maintheo}
$$
\begin{tabular}{|c|*{2}{c|}c|}
\hline
\diagbox[dir=NW]{$\kappa$}{$\Delta$} & positive & negative \\   \hline
  odd& (\ref{equsec})  & (\ref{equfir}), (\ref{equsec}) \\   \hline
   even&  (\ref{equfir}) &  (\ref{equfir}) \\
   \hline
 \end{tabular}
$$

When $\Delta>0$ and $\kappa$ is odd, if  $\Delta$ has a prime factor $p$ such that $p \equiv 3(\mod 4)$, then the equation (\ref{equfir}) holds.
\begin{theorem}\label{maintheo}
 For any $f \in H_{3/2-\kappa, \bar{\tilde{\rho}}_L}$ and Heegner cycle  $ Z_{\Delta, r, \kappa}(m, \mu)$, there are infinitely  many  Heegner cycles $Z_{\Delta, r, \kappa}(U)$,  such that the  global heights are given by
\begin{equation}\label{equfir}
\langle  Z_{\Delta, r, \kappa}(f),  Z_{\Delta, r, \kappa}(U)\rangle
=\frac{2\sqrt{N|\Delta|}\Gamma(\kappa-\frac{1}{2})}{(4\pi)^{\kappa-1}\pi^{\frac{3}{2}}}L'(Sh_{m_0, \mu_0}(\xi_{3/2-\kappa}f), \chi_\Delta, \kappa).\tag{\uppercase\expandafter{\romannumeral1}}
\end{equation}
and 
\begin{equation}\label{equsec}
\langle  Z_{\Delta, r, \kappa}(m, \mu),  Z_{\Delta, r, \kappa}(U)\rangle
=\frac{\sqrt{N|\Delta|}\Gamma(\kappa-\frac{1}{2})}{(4\pi)^{\kappa-1}\pi^{\frac{3}{2}}}L'(Sh_{m_0, \mu_0}(\xi_{3/2-\kappa}F_{m, \mu}), \chi_\Delta, \kappa).\tag{\uppercase\expandafter{\romannumeral2}}
\end{equation}

\end{theorem}
\begin{remark}
$(1)$When $\Delta=1$, this theorem was proved  in  \cite{BY} and  \cite{BEY}.


$(2)$ Both  cycles are cohomologically trivial in $H^{2\kappa}(Y(\C), \C)$, so we can use the global 
height $\langle~, ~\rangle$ in the above theorem.
\end{remark}

This theorem offers  a novel representation of the derivatives of L-functions,  which is useful  for  the exploration of subsequent works.
\begin{enumerate}
  \item  When  $\kappa = 1$,  this formula allows for a direct proof of the arithmetic inner product formula-Theorem \ref{theoarithprod}, as well as   the  Gross-Zagier formula \cite{GZ}. As a result, we provide a new proof of this formula.
  \item   For $\kappa >1$, this formula gives  a new approach to the Gross-Zagier-Zhang formula.

\end{enumerate}

\begin{corollary}
If $Z_{\Delta, r, \kappa}(f)= 0$, then $L'(Sh_{m_0, \mu_0}(\xi_{3/2-\kappa}f), \chi_\Delta, \kappa) = 0$.
This indicates that the implication $(2)\Rightarrow (1)$   in Conjecture \ref{conmain}.
\end{corollary}

\begin{conjecture}[Modularity hypothesis]\label{condimone}
The following generating function is a cusp form,
\begin{equation}\label{conone}
\sum_{m>0, \mu}Z_{\Delta, r, \kappa}^G(m, \mu)q^me_\mu=g(\tau)\otimes Z_{\Delta, r, \kappa}(f).
\end{equation}
\end{conjecture}
When $\kappa = 1$,
this conjecture has been proved  in    \cite{GKZ}, \cite{BO} and \cite{BoDuke}.
In the case of $\kappa > 1$,  the image of Heegner cycles under the p-adic Abel-Jacobi map is at most one-dimensional  in \cite{Ne}.

According to Theorem \ref{maintheo}, we have the following result.
\begin{proposition}
$$\text{Modularity~ hypothesis} ~\ref{condimone}~\Rightarrow~ \text{Theorem}~ \ref{thehigher}.$$
\end{proposition}

%

 Kudla has proposed a research program that known as the Kudla program, which focuses on the arithmetic geometric properties of the derivatives of L-functions. This program is structured into two primary steps:
\begin{enumerate}
 \item To  establish the arithmetic Siegel-Weil formula and to construct  an arithmetic theta function.
 \item To prove  the arithmetic inner product formula.

\end{enumerate}

Kudla, Rapoport, and Yang investigated the Shimura curve, presenting the initial case of the Kudla program \cite{KRYComp}\cite{KRYBook}.
Recent progress within this program has yielded significant advancement, particularly in unitary cases; however, challenges continue to persist with the modular curve, the most fundamental case, which remained unresolved until the works of Yang and the author \cite{DY1} \cite{DY2}. They completed the first step of the  Kudla program and constructed the arithmetic theta functions,
which is defined  as follows
\begin{equation} \label{eq:GeneratingFunction1}
\widehat{\phi}_{\Delta, r}(\tau) =  \sum_{ n, \mu}
 \widehat{\mathcal Z}_{\Delta, r}(n, \mu, v) q_\tau^n e_\mu \in \C[L^\sharp/L]\otimes \widehat{\CH}^1_\R(\mathcal X_0(N)),
\end{equation}
where $\widehat{\mathcal Z}_{\Delta, r}(n, \mu, v) \in \widehat{\CH}^1_\R(\mathcal X_0(N)) $ are arithmetic Heegner divisors.
This function is a vector valued modular form for $\Gamma'$ of weight $\frac{3}2$.

We will complete the Kudla program on modular curves in this work.
\begin{definition}
The arithmetic theta lift is defined by
\begin{equation}
\widehat{ \theta}_{\Delta, r}:  S_{\frac{3}{2}, \tilde\rho_L}
\rightarrow \widehat{\CH}^1_\R(\mathcal X_0(N)),~~~g  \mapsto  \widehat{ \theta}_{\Delta, r}(g)=\langle\widehat{\phi}_{\Delta, r},g(\tau)\rangle_{Pet}.
\end{equation}
\end{definition}
The  arithmetic Chow group is decomposed  into  Mordell-Weil component and  orthogonal complementary component as \cite{KRYBook},
\begin{equation}
\widehat{\CH}^1_\R(\mathcal X_0(N)) = \widetilde{\MW}\oplus \widetilde{\MW}^{\perp},~~ \widetilde{\MW} \simeq J_0(N)(\Q)\otimes \R. \nonumber
\end{equation}
The modularity of
 $\widehat{\phi}_{\Delta, r}(\tau) $ has been proved for square-free levels  $N$, while  the Mordell-Weil component $\widehat{\phi}_{MW}$ is  a modular form and can be defined for all levels $N$. Another arithmetic theta lift is defined as  follows
\begin{equation}
\widehat{\phi}_{MW}(g)=\langle \widehat{\phi}_{MW}, g \rangle_{Pet}.
\end{equation}

Then we have the following result.
\begin{theorem}\label{thelift}
Let the notation be as above. Then
\begin{enumerate}
  \item 
$\widehat{ \theta}_{\Delta, r}(g)=\widehat{\phi}_{MW}(g).$
  \item 
$\widehat{ \theta}_{\Delta, r}(g) \in \widetilde{\MW}.$
  \item \text{For any} $\widehat{\mathcal{Z}} \in \widetilde{\MW}^{\perp}$, 
$
\langle \widehat{ \theta}_{\Delta, r}(g), \widehat{\mathcal{Z}}  \rangle_{\GS}=0.
$
\end{enumerate}
\end{theorem} 
Both lifts  $\widehat{ \theta}_{\Delta, r}(g)$ and $\widehat{\phi}_{MW}(g)$ belong to the Mordell-Weil part $\widetilde{\MW}$ of arithmetic Chow group.
As a result, we can employ the N\'eron-Tate height. Then we can prove the following result
\begin{theorem}[Arithmetic inner product formula]\label{theoarithprod}
Let the notation be as above. For any  $g \in S_{\frac{3}{2}, \tilde\rho_L}$, then
\begin{eqnarray}
\langle \widehat{\phi}_{MW}(g), \widehat{\phi}_{MW}(g) \rangle_{NT}=\frac{2\parallel g\parallel^{2}\sqrt{N|\Delta|}}{\pi}L'(G, \chi_\Delta, 1).
\end{eqnarray}
When $N$ is square free,
\begin{equation}\label{equarith}
\langle \widehat{\theta}_{\Delta, r}(g), \widehat{\theta}_{\Delta, r}(g) \rangle_{NT}=\frac{2\parallel g\parallel^{2}\sqrt{N|\Delta|}}{\pi}L'(G, \chi_\Delta, 1).
\end{equation}
\end{theorem}
We find the following relations
\begin{corollary}
$$
\text{Arithmetic inner product formula  } \iff  \text{Gross-Zagier formula.} \nonumber
$$
\end{corollary} 
So we can  drop Assumption A on $\Delta$ in the Theorem \ref{theoarithprod}.

The arithmetic inner product formula  (\ref{equarith}) may be generalized to arbitrary level $N$. Zhu constructed a scalar valued  arithmetic theta function  in \cite{Zhu}. We will prove analogy arithmetic inner formula for general level $N$ in future.

\begin{corollary}\label{cormod}
\begin{equation}
\widehat\theta_{\Delta, r}(g)\neq 0 \iff L'(G, \chi_\Delta, 1)\neq 0.
\end{equation}
\end{corollary}

%
%
%
%
%
%
%
%

For each $G \in S_{2}^{new}(N)$, there is a cuspidal automorphic representation $\pi(G)\simeq\bigotimes_{p\leq \infty}\pi_p$
of $PGL_2(\A)$. According to Waldspurger's work \cite{Wa}, there exists an irreducible genuine cuspidal representation $\sigma\simeq\bigotimes\sigma_p$, s.t., $Wald(\sigma, \psi)=\pi(G)$.

Now we denote  $\pi=Wald(\sigma, \psi_{-1})$.
Waldspurger proved an important result as follows.
\begin{theorem}\cite{Wa}
When the global root number  $\epsilon(\frac{1}{2}, \pi)=1$,
\begin{equation}
\text{the theta lift }\theta_{\psi}(\sigma, V^B)\neq 0 \iff L(\frac{1}{2}, \pi)\neq 0;
\end{equation}
when the global root number  $\epsilon(\frac{1}{2}, \pi)=-1$, $L(\frac{1}{2}, \pi)=0$.
\end{theorem}
 Kudla, Rapoport and Yang proved the following result \cite[Charpter 9]{KRYBook}.
\begin{theorem}
Assume that $\epsilon(\frac{1}{2}, \pi)=-1$. For some genuine cuspidal representations $\sigma$, there exist indefinite quaternion algebras  $B$ with $D(B)>1$, such that 
 $$\text{the arithmetic theta lift }\theta^{ar}_{\psi}(\sigma, V^B)\neq 0 \iff L'(\frac{1}{2},\pi )\neq 0.$$
\end{theorem}
We extend the above result  to the modular curve ($D(B)=1$) in Corollary \ref{cormod}.
%
%

\subsection{Plan of proof}

In Section \ref{Sec1}, we will present preliminary concepts related to modular forms and Green functions.

In Section \ref{Sec2}, we will introduce  the arithmetic theta function $\widehat{\phi}_{\Delta, r}$. According to 
the decomposition of arithmetic Chow group $\widehat{\CH}^1_\R(\mathcal X_0(N))$, we 
 write it as follows
$$
 \widehat{\phi}_{\Delta, r}= \widehat{\phi}_{\MW}+\deg(\widehat{\phi}_{\Delta, r})\widehat{\mathcal P}_\infty    + \widehat{\phi}_{\text{Vert}} + a(\phi_{SM}).$$
Each term in the summation is a modular form of weight $3/2$.

In Section \ref{Sec3},  we will define the arithmetic theta lift $\widehat{ \theta}_{\Delta, r}(g)$ and prove  Theorem  \ref{thelift},  $$\widehat{ \theta}_{\Delta, r}(g) \in \widetilde{\MW},~~for~~any~~g \in S_{\frac{3}{2}, \tilde\rho_L}.$$

In Section \ref{Sec4},  we will introduce several Green functions, including automorphic Green functions and the regularized Borcherds lift.

In Section \ref{Sec5},  we will define the translated twisted Heegner divisors for general level $N$. We will sutdy it  by adelic language.

In Section \ref{Sec6}, we will study the CM values of Green functions. The archimedean  intersection number
$\langle Z_{\Delta, r, \kappa}(f), Z_{\Delta, r, \kappa}(U)\rangle_\infty$ is given by this value.

In Section \ref{Sec7}, we will elucidate the relationship between $\Phi_{\Delta, r}^j(Z(U), f)$ and the derivative of L-functions  in Theorem \ref{theotwistedvalue}, which leads us to conclude that the archimedean intersection number can provide the derivative of L-functions.

In Section \ref{Sec8}, we will compute the intersection numbers at finite places. Combining this with  archimedean part -Theorem \ref{theotwistedvalue}, we will prove Theorem  \ref{maintheo}.

In Section \ref{Sec9},  we will study the self-intersection of the arithmetic theta lift.
 According to Theorem \ref{maintheo}, we will demonstrate the arithmetic inner product formula- Theorem  \ref{theoarithprod} and the Gross-Zagier formula. Furthermore, we will  clarify the  relationship  between  Theorem \ref{maintheo} and Theorem \ref{thehigher}.

The subsequent diagram illustrates these relationships.

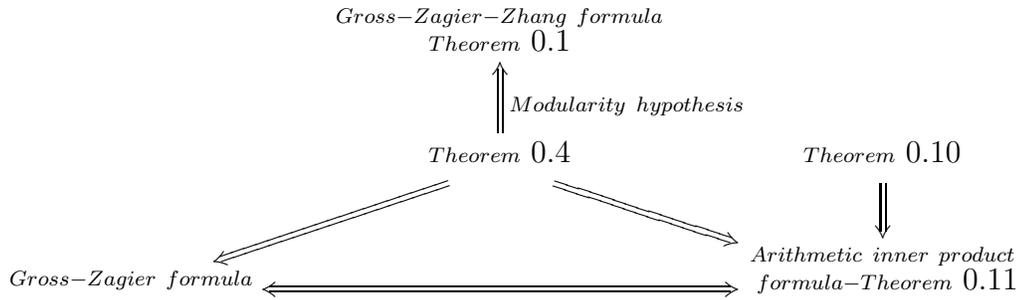
\begin{displaymath}
\xymatrix{
 &^{\substack{Gross-Zagier-Zhang ~formula\\Theorem ~\ref{thehigher} }} \ar@{<=}[d]^{Modularity~ hypothesis} \\
                 &^{Theorem ~\ref{maintheo} }\ar@{=>}[dr] &  ^{Theorem ~\ref{thelift} }    \\
^{\substack{Gross-Zagier ~formula}} \ar@{<=}[ur]  \ar@{<=>}[rr] & & ^{ \substack{Arithmetic ~inner ~product\\ ~formula-Theorem~ \ref{theoarithprod}}}\ar@{<=}[u] }
\end{displaymath}



\part{Arithmetic theta lift}

\section{Preliminaries}\label{Sec1}

Let $\widetilde{\SL}_2(\R)$ be the metaplectic double cover of $\SL_2(\R)$, which can be  viewed as pairs $(g, \phi(g, \tau))$, where $g=\kzxz{a}{b}{c}{d} \in \SL_2(\R)$ and $\phi(g,
\tau)$ is a holomorphic function of $\tau \in \H$ such that $\phi(g, \tau)^2 = j(g, \tau) = c\tau +d$. Let $\Gamma'=\Mp_2(\Z)$ be the preimage of
$\Gamma=\SL_2(\Z)$ in $\widetilde{\SL}_2(\R)$, then $\Gamma'$ is generated by
$$
S=\left( \kzxz {0} {-1} {1} {0}, \sqrt \tau \right)  \quad  T= \left( \kzxz {1} {1} {0} {1} , 1 \right).
$$

Let $(V, Q)$ be a
quadratic space of signature $(p, q)$ and let $L\subset V$ be an even lattice. We write $L^\sharp$
for its dual lattice. The quadratic form on $L$ induces a $\Q/\Z$-valued quadratic
form on the discriminant group $L^{\sharp}/L $.
The standard basis of $S_L=\C[L^\sharp/L]$ is denoted by $$\{ e_\mu=L_\mu \mid \mu \in L^\sharp/L\}.$$  Then the Weil representation $\rho_L$ of $\Gamma'$ on $\C[L^\sharp/L]$ (
\cite{Bo}) is given by
\begin{eqnarray} \label{eq:WeilRepresentation}
&\rho_{L} (T)e_{\mu}&=e(Q(\mu))e_{\mu} ,\\
&\rho_{L}(S) e_{\mu}&= \frac{e(-\frac{p-q}8)}{\sqrt{\vert L^{\sharp} / L \vert}}\sum \limits_{\mu^{\prime} \in L^{\sharp} / L }e(-(\mu,
\mu^{\prime}))e_{\mu^{\prime}}\nonumber .
\end{eqnarray}

For a non-zero integer $\Delta$,
 we consider  the quadratic form
$Q_\Delta:=\frac{Q}{|\Delta|}$.
The Weil representation of the quadratic lattice $L^\Delta=(\Delta L, Q_\Delta)$ is denoted by $\rho_{L^\Delta}$.

\subsection{Modular form}

A twice continuously differentiable function $f : \H \rightarrow \C[L^\sharp/L]$ is called a weak Maass form of weight $k \in \frac{1}{2}\Z$ with representation $\rho_L$ if
\begin{itemize}
\item[1).] $f|_{k, \rho_L}\gamma = f$ for all $\gamma \in\Gamma'$;
\item[2).] there exists a $\lambda \in \C$ that$\Delta_kf = \lambda f$;
\item[3).] there is a $C>0$, such that $f(\tau)=O(e^{Cv})$  as $v\rightarrow \infty$ uniformly for $u$.
\end{itemize}
 Here the slash operator is given by
$$f\mid_{k, \rho_L}\gamma(\tau) =\phi(\tau)^{-2k}\rho^{-1}_L(\gamma')f(\gamma \tau).$$
When $\lambda =0$, $f$ is a harmonic weak Maass form.
The space of harmonic weak Maass forms is denoted by $H_{k, \rho_L}$. The function $$P_f(\tau) = \sum_{\mu, n\leq 0}c(n, \mu)q^ne_{\mu}$$ is called the  principal part of $f$.
 For any $f \in H_{k, \rho_L}$, it has an unique decomposition $f=f^++f^-$ by Fourier expansion, where
 \begin{equation}
f^+=\sum_{\mu}\sum_{n \gg 0} c^+(n, \mu)e(n\tau)e_{\mu}
 \end{equation}
 and
\begin{equation}
f^-=\sum_{\mu}\sum_{n<0} c^-(n, \mu)\Gamma(1-k, 2\pi |n|v)e(n\tau)e_{\mu} .
 \end{equation}
 Here $\Gamma(a, x)=\int_{-x}^{\infty}e^{-t}t^{a-1}dt$ is the incomplete $\Gamma$ function. For any field $F$, we write
 $ H_{k, \rho_L}(F)$ denote the space of harmonic Maass forms with principal part defined over $F$.
There is a differential operator defined by
\begin{equation}
\xi_{k}(f): H_{k, \overline{\rho}_L}\rightarrow S_{k, \rho_L}
\end{equation}
where $\xi_{k}(f)=2iv^k\frac{\overline{\partial} f}{\partial\bar{\tau}}$.
The exact sequence is given by
 \begin{equation}
 0\rightarrow M_{k, \overline{\rho}_L}^{!}\rightarrow H_{k, \overline{\rho}_L}\xrightarrow{\xi_k}S_{2-k, \rho_L}\rightarrow 0.
 \end{equation}
Here $M_{k, \overline{\rho}_L}^{!}$ is the space of weakly modular form.
We also define the Maass  lowering and raising operators in weight $k$ by
\begin{equation}
  L_k=-2iv^2\frac{\partial}{\partial \bar\tau},~and~  R_k=2i\frac{\partial}{\partial\tau}+kv^{-1}.
\end{equation}
The standard scalar product on the space $\C[L^{\sharp}/L]$ is defined as
\begin{equation}
\langle \sum_{\mu}f_\mu e_{\mu}, \sum_{\mu}g_\mu e_\mu\rangle=\sum_\mu f_\mu g_\mu.
\end{equation}
For any  modular form $f, g \in M_{k, \rho_L} $, the Petersson scalar product is defined by
\begin{equation}
\langle f, g\rangle_{Pet}=\int_{\mathcal{F}}\langle f, \bar{g}\rangle v^k\mu(\tau),
\end{equation}
where $\mu(\tau)=\frac{dudv}{v^2}$ is the hyperbolic measure.

Given a sublattice $M \subseteq L$with finite index,
we have the inclusions $M \subseteq L \subseteq L^\sharp \subseteq M^\sharp$,
$$L/M \subseteq L^\sharp/M \subseteq M^\sharp/M $$ and natural quotient map $\pi : L^\sharp/M \rightarrow L^\sharp/L, h \rightarrow \bar{h}$.

The restriction map and  the  trace map are defined as follows: for any
$f \in A_{k,\rho_L}$ and  any $g\in A_{k,\rho_M}$,

$$(f_M)_h=\begin{cases}
f_{\bar{h}} &\ff h \in L^\sharp/M,\\
0 &\ff h \notin L^\sharp/M,
\end{cases}$$

$$(g^L)_\mu=\sum_{\alpha \in L/M}g_{\alpha+\mu},$$
where $h \in M^\sharp/M$ and $\mu \in  L^\sharp/L$.
Here $A_{k,\rho_L}$ is the space of modular forms with weight $k$ and representation $\rho_L$.
\begin{lemma}\cite[Lemma 3.1]{BY}\label{lemresext}
$$res_{L/M}: A_{k,\rho_L} \rightarrow A_{k,\rho_M}, ~~ f \mapsto f_M$$
and
$$tr_{L/M}: A_{k,\rho_M} \rightarrow A_{k,\rho_L}, ~~  g \mapsto g^L$$
such that for $f \in A_{k,\rho_L}$ and $g\in A_{k,\rho_M}$, one has
 $$\langle f, \bar{g}^L \rangle=\langle f_M, \bar{g} \rangle.$$

\end{lemma}

\subsection{The signature $(1, 2)$ }
Let
   \begin{equation}
V=\{w=\kzxz
    {w_1} { w_2 }
      {w_3}   {-w_1}  \in M_{2}(\Q) | \tr(w) =0 \},
\end{equation}
with the quadratic form $Q(w)=N \det{w}=-Nw_2w_3-Nw_1^{2}$, where $N$ is a positive integer.
Let
 \begin{equation}\label{equlattice}
L=\big\{w =\kzxz {b}{\frac{-a}{N}}{c}{-b}
   \in M_{2}(\Z) |\,   a, b, c \in \Z \big\},
\end{equation}
be the lattice in $V$.
  We will identify
$$
\Z/2N\Z \cong L^\sharp/L,  \quad r\mapsto \mu_r = \kzxz {\frac{r}{2N}} {0} {0} {-\frac{r}{2N}}.
$$

Let $H= \GSpin(V)\cong \GL_2$, which acts on $V$  by conjugation, i.e.,  $g.w=gwg^{-1}$.
 Notice that $\Gamma_0(N)$ preserves $L$ and acts on $L^\sharp/L$ trivially.

For each $\mu \in  L^\sharp/L$, denote  $L_{\mu}=L+\mu$, and
\begin{equation}
L_\mu[n] =\{ w \in  L_\mu|\,  Q(w) = n\}.
\end{equation}

Let $\Delta \in \Z$ be a fundamental  discriminant, s.t., $\Delta \equiv r^2(\mod 4N)$, and let $L^\Delta= (\Delta L, Q_\Delta)$. 
It is easy to that the its dual lattice is $L^\sharp$.

The  generalized genus character \cite[Section 4]{BO}
$$
\chi_\Delta: L^\sharp/L^\Delta  \rightarrow \{ \pm 1 \}
$$
 is defined by
\begin{equation}
\chi_{\Delta}\bigg(\kzxz{\frac{b}{2N} }{\frac{-a}{N}}{c}{-\frac{b}{2N}}\bigg) =\begin{cases}
  (\frac{\Delta}{n}),  &\ff \Delta \mid b^2-4Nac~and~ \frac{b^2-4Nac}{\Delta}~is~ a\\
  &  ~square ~modulo~4N~and~(a, b, c, \Delta)=1,
  \\
  0,  &otherwise,
 \end{cases}
\end{equation}
Here $n$ is any integer prime to $\Delta$ represented by one of the quadratic forms $[N_1a, b, N_2c]$ with $N_1N_2=N$ and $N_1$, $N_2>0$.
 The generalized genus character $\chi_\Delta(w) = \chi_\Delta ([a, b, Nc])$ is  defined in \cite[Section 1]{GKZ}.
 It is  invariant under the action of $\Gamma_{0}(N)$ and the action of all Atkin-Lehner involutions \cite{GKZ} , i,e.,
\begin{equation}
\chi_{\Delta}(\gamma w\gamma^{-1})=\chi_{\Delta}(w), ~\chi_{\Delta}(W_{M}wW_{M}^{-1})=\chi_{\Delta}(w),
\end{equation}
where $\gamma \in \Gamma_{0}(N)$ and $W_{M}$ is the Atkin-Lehner involution with $M \| N$.
%

Let $\D$ be the
Hermitian domain of oriented negative 2-dimensional subspace of $V(\R)$. Then $\D$ can be identified with $\H\cup \bar \H$ via
\begin{equation}
z=x+iy\mapsto \R \mathfrak R \kzxz{z}{-z^2}{1}{-z}+\R \mathfrak I \kzxz{z}{-z^2}{1}{-z}.
\end{equation}
For any $w=\kzxz{\frac{b}{2N}}{-\frac{a}{N}}{c}{-\frac{b}{2N}} \in L^\sharp$, we denote the CM point by
\begin{equation}
z(w)=\frac{b}{2Nc}+\frac{\sqrt{b^2-4Nac}}{2N|c|} \in \H.
\end{equation}
For any $\mu\in L^{\sharp}/L$ and a positive rational number $n \in \sgn(\Delta)Q(\mu) +\Z$,  the twisted Heegner divisor  is defined by
\begin{equation}\label{TwistedHeegner}
Z_{\Delta, r}(n, \mu):=\sum_{w \in \Gamma_0(N) \setminus L_{r\mu}[n \mid\Delta \mid]}\chi_{\Delta}(w)z(w) \in \Div(X_{0}(N))_{\Q},
\end{equation}
which is defined over $\Q(\sqrt{\Delta})$. We count each  point $z(w)$  with  multiplicity $\frac{2}{|{\Gamma}_w|}$ in the orbifold $X_0(N)$, where $\Gamma_w$ is the stabilizer of $w$ in $\Gamma_0(N)$. This definition is the same as that in  \cite[Section 5]{AE} and \cite[Section 5]{BO}.

Following \cite[Section 3.1]{AE}, we let
\begin{equation}
 \psi_{\Delta, r}(e_\mu)=
 \sum_{\substack{\delta\in L^\sharp/ L^\Delta \\ \pi(\delta)=r\mu\\ Q_{\Delta}(\delta)\equiv sgn(\Delta)Q(\mu)(\Z) }}\chi_{\Delta}(\delta)e_{\delta},
\end{equation}
where $\pi$ denotes the quotient map $\pi : L^\sharp/L^\Delta \rightarrow  L^\sharp/ L$.

For  $f \in A_{k, \rho_{L^\Delta}}$ and $g \in A_{k, \tilde\rho_L}$, we define two operators by
\begin{equation}\psi(f)=\sum_{\mu \in L^\sharp /L}\langle \psi_{\Delta, r}(e_\mu), f \rangle e_\mu\end{equation}
and
\begin{equation}\phi(g)=\sum_{\substack{\delta, \pi(\delta)=r \mu\\ Q_{\Delta}(\delta)\equiv sgn(\Delta)Q(\mu)(\Z) }}\chi_{\Delta}(\delta)g_\mu e_{\delta}.\end{equation}

For any  $f \in A_{k, \rho_{L^\Delta}}$,  $\psi(f)\in A_{k, \tilde\rho_L }$, one can see\cite[Section 3.1]{AE}.
Then we obtain the following result
\begin{proposition}\label{proadjoint}
Let $k \in \frac{1}{2}\Z$. For any $f \in A_{k, \rho_{L^\Delta}}$ and $g \in A_{k, \tilde\rho_L}$,  we have 
 $$\psi(f)\in A_{k, \tilde\rho_L },~~ \phi(g) \in A_{k, \rho_{L^\Delta}}.$$ Moreover, we have
$$
\langle \psi(f), \bar{g} \rangle=\langle f, \overline{\phi(g)} \rangle.
$$
\end{proposition}

\subsection{Twisted theta functions}\label{sectheta}

For any $z \in \D$, we let
\begin{equation}
w(z) = \frac{1}{\sqrt{N} y} \left(
  \begin{array}{cc}
  -x  & z\overline{z}\\
    -1&x\\
  \end{array}
\right)\in V(\R).
\end{equation}

Then one has the following decomposition
$$
V(\R) =z \oplus  z^\perp.
$$
For each $w \in V(\R)$, we  can write it as $w=  w_z + w_{z^\perp}$.

We let
\begin{equation}\lambda(z)=\frac{\sqrt{|\Delta|}}{\sqrt{2}}w(z)=\frac{\sqrt{|\Delta|}}{\sqrt{2N}y}\kzxz{-x}{x^2+y^2}{-1}{x}
\end{equation}
be the
normalized vector.
For any $w=\kzxz{\frac{b}{2N}}{-\frac{a}{N}}{c}{-\frac{b}{2N}} \in L^\sharp$, we define
\begin{equation}
p_z(w)=-\frac{1}{\sqrt{2}}(w, \lambda(z))_{\Delta}=-\frac{1}{2\sqrt{N|\Delta|}y}(Nc|z|^2-bx+a).
\end{equation}


Then it follows that
\begin{lemma}
\begin{equation}\label{equdist}
\sqrt{Q_\Delta(w_{z^\bot})}=|p_z(w)|.
\end{equation}
\end{lemma}
\begin{proof}
For  $w =\kzxz
    {w_{1}}{w_{2}}
     {w_{3}}{-w_{1}}\in V(\R)$,  we have
\begin{equation} \label{formula1}
(w, w(z))_{\Delta}=-\frac{\sqrt{N}}{y |\Delta|}(w_3z\overline{z}-w_1(z+\overline{z})-w_2).
\end{equation}
Since 
\begin{equation}
(w(z), w(z))_{\Delta}=\frac{2}{|\Delta|},
\end{equation}
we have 
$$w_{z^\bot}=-\frac{\sqrt{N}}{2y }\big(w_3z\overline{z}-w_1(z+\overline{z})-w_2 \big)w(z).$$
Then we obtain the result.
\end{proof}

We define \begin{equation}
R(w, z)_{\Delta} =-(w_z, w_z)_{\Delta},
\end{equation}
and it can be written as \begin{align} \label{eq3.5}
R(w, z)_{\Delta}&=\frac{1}{2}(w, w(z))_{\Delta}^{2}-(w, w)_{\Delta}.
\end{align}

We consider the associated  majorant
\begin{equation}
(w, w)_{ z}= (w_{z^\perp}, w_{z^\perp})_{\Delta} + R(w, z)_{\Delta},
\end{equation}
which is a positive definite quadratic form on the space  $V(\R)$.

The following  Gaussian $\varphi_\infty$ belongs to  $S(V(\R))$,
\begin{equation}
\varphi_\infty(w, z )=e^{-\pi (w, w)_{ z}}.
\end{equation}
We denote
\begin{eqnarray} \label{eq3.5}
\varphi_{\Delta}^S( w, \tau, z)&=&ve(Q_{\Delta}(w)\tau)e^{-2\pi R(\sqrt{v}w, z)_{\Delta}}\nonumber\\
&=&ve(Q(w_{z^\perp})\tau+Q(w_z)\bar\tau).
\end{eqnarray}

By the Weil representation $\omega$, we have 
\begin{equation}
v^{\frac{1}{4}}\omega(g_\tau)\varphi_\infty(w, z )=\varphi_{\Delta}^S( w, \tau, z),
\end{equation}
where $g_\tau=\kzxz{1}{u}{}{1}\kzxz{v^{\frac{1}{2}}}{}{}{v^{-\frac{1}{2}}}$ and $\tau =u+iv \in \H$.

Let
$$\label{eq3.5}
\varphi^{0}_{\Delta}(w, z)=\bigg((w, w(z))_{\Delta}^{2}- \frac{1}{2\pi}\bigg)e^{-2\pi R(w, z)_{\Delta}}\mu(z) \notag
$$
and
\begin{equation}
 \varphi_{\Delta}(w, \tau, z) = e(Q_{\Delta}(w)\tau) \varphi^0_{\Delta}(\sqrt v w, z) ,
\end{equation}
be the  differential forms on $V(\R)$, where $\mu(z) =\frac{dx \, dy}{y^2}$.

The Siegel theta function, Millson theta function and  the Kudla-Millson theta function are defined as follows:
\begin{eqnarray}
\Theta(\tau, z)&=&\sum_{\delta \in L^{\sharp}/L^{\Delta}}\sum_{w \in  L^{\Delta}_\delta} \varphi_{\Delta}^S( w, \tau, z)e_\delta\nonumber\\
&=&v\sum_{\delta \in L^{\sharp}/L^{\Delta}}\sum_{w \in  L^{\Delta}_\delta} e(Q(w_{z^\perp})\tau+Q(w_z)\bar\tau),
\end{eqnarray}
\begin{equation}\label{thetamill}
\Theta^{\mathcal M}(\tau, z)=\sum_{w \in  L^{\Delta}_\delta}p_z(w)\varphi_{\Delta}^S( w, \tau, z)e_\delta,
\end{equation}
and
\begin{align}
\Theta^{KM}(\tau, z) &=\sum_{\delta \in  L^\sharp/L^\Delta} \sum_{w\in  L^\Delta_\delta} \varphi_{\Delta}( w, \tau, z)e_\delta ,
\end{align}
where $L^{\Delta}_\delta= \Delta L+\delta$. These theta functions can be defined by adelic language, and we will introduce them in the later sections.

The twisted  Siegel theta function is defined as
\begin{eqnarray} \label{siegeltheta}
\Theta_{\Delta, r}(\tau, z)= \psi(\Theta(\tau, z)),
\end{eqnarray}
which is a modular form  of weight $-1/2$  associated with the representation $(\Gamma',  \tilde{\rho}_L)$ with respect to  $\tau$. Furthermore, it  is $\Gamma_0(N)$-invariant as a function of $z$.

Similarly, the  weight $1/2$ twisted Millson theta function and weight $3/2$ twisted Kudla-Millson theta function
are  defined  in  \cite[Section 4]{AE} as,
\begin{equation}
\Theta^{\mathcal M}_{\Delta, r}(\tau, z, h)=\psi (\Theta^{\mathcal M}(\tau, z) ),
\end{equation}
and
\begin{equation} \label{eq:TwistedTheta}
\Theta_{\Delta, r}^{KM}(\tau, z)
=\psi(\Theta_{L^\Delta}^{KM}(\tau, z) ).
\end{equation}
These two theta functions take values in $ \Omega^{1, 1}(X_{\Gamma}) $.

It is important to note that when $\Delta=1$, we often omit the indices $\Delta$ and $r$.

\subsection{Green functions}

 We consider $X$ as a compact Riemann surface, with $\Omega^1$ representing the sheaf
of holomorphic $1$-forms. The global sections on $X$  are denoted by $\Gamma(X, \Omega^1)$, which
possesses a scalar product
$$\langle \omega_1, \omega_2 \rangle=\frac{i}{2}\int_{X}\omega_1 \wedge \overline{\omega_2}.$$
We denote the orthogonal basis by $\{f_1dz,...,f_gdz\}$ with the genus $g$,
and the Arakelov canonical $(1, 1)$-form
$$\nu_{can}=\frac{\sqrt{-1}}{2g}\sum_{i=1}^{i=g}|f_i|^2dz\wedge d\bar{z}.$$

A form $\nu$ is called volume form, if it is smooth, positive, real $(1, 1)$-form with $\int_X\nu=1$.

A $\nu$-admissible Green function $g(z_1, z_2)$ is a real valued function on $X\times X$ and smooth  out the diagonal $\Delta_X$ of $X \times X$ and satisfies
\begin{itemize}
\item[1)]Near the diagonal it has the following expansion $$g(z_1, z_2)=-\log |z_1-z_2|^2+smooth.$$
\item[2)] $g(z_1, z_2)=g(z_2, z_1)$.
\item[3)]It has the current equation
$$d_{z_1}d_{z_1}^c[g(z_1, z_2)]+\delta_{z_2}=[\nu(z_1)].$$
\end{itemize}
It is normalized if $$\int_{X}g(z_1, z_2)\nu(z_1)=0.$$
The generalized Green function $g(z_1, z_2)$ if the item $3)$ generalized to
$$d_{z_1}d_{z_1}^c[g(z_1, z_2)]+\delta_{z_2}=[\mu(z_1)]$$ for some   $(1, 1)$-form $\mu$.

From now on, we denote $g_{z_1}(z_2)=g(z_1, z_2)$. For any given divisor $D=\sum n_i P_i$, we will denote $$g_D=\sum n_i g_{P_i}.$$
Suppose that $g_{z_1}(z_2)$ is $\nu$-admissible.   In this case, the following equality holds
$$d_{z_1}d_{z_1}^c[g_D( z_2)]+\delta_{D}=\deg(D)[\nu(z_1)].$$
The positive elliptic Laplacian differential operator $\Delta_z$ is defined as
\begin{equation}
dd^cf=\frac{\sqrt{-1}}{2\pi}\partial\bar{\partial}f=\frac{1}{2}\Delta_z f\nu.
\end{equation}
Here $d^c= \frac{i}{4\pi}(\bar{\partial}-\partial)$ and $d=\bar{\partial}+\partial$.

\section{Arithmetic theta function}\label{Sec2}
The arithmetic theta function $\widehat{\phi}_{\Delta, r}(\tau)$ is defined  in \cite{DY1} and \cite{DY2}. In this
section, we  write it as 
$$ \widehat{\phi}_{\Delta, r}= \widehat{\phi}_{\MW}+\deg(\widehat{\phi}_{\Delta, r})\widehat{\mathcal P}_\infty    + \widehat{\phi}_{\text{Vert}} + a(\phi_{SM}).$$  
Each component of this summation is a vector valued modular form  of weight $3/2$.  
\subsection{Arithmetic intersection on $\mathcal X_0(N)$}

Recall the definition in \cite{KM}, let $\mathcal{Y}_{0}(N)$ $ ( \mathcal{X}_{0}(N) )$ be the moduli stack over $\Z$ of cyclic isogenies of degree $N$ of elliptic
curves (generalized elliptic curves) $\pi : E\rightarrow E^{\prime}$, such that $\ker \pi$ meets every irreducible component of each geometric fiber.
The stack $\mathcal{X}_{0}(N)$ is regular, flat over $\Z$ and is smooth over $\Z[\frac{1}N]$. Notice that $\mathcal{X}_{0}(N)(\C)=X_{0}(N)$.

Let $D=-4Nm $ be a discriminant and the order $\OO_{D}=\Z[\frac{D+\sqrt D}2]$ of discriminant $D$. Assume that $D\equiv r_\mu^2 \mod 4N$ and $\mu=\kzxz{\frac{r_\mu}{2N}}{}{}{-\frac{r_\mu}{2N}}$. Then  $\mathfrak {n}=[N, \frac{r_\mu+\sqrt D}2]$ is an ideal of $\OO_{D}$ with norm $N$.

   Let $\mathcal Z(m, \mu)$ be the moduli stack over $\Z$ of the pairs $(x, \iota)$ that defined in \cite[Section 7]{BY}, where
\begin{enumerate}
  \item   $x=(\pi: E \rightarrow E') \in \mathcal Y_0(N)$,
  \item $
\iota: \OO_{D} \hookrightarrow \End(x)=\{ \alpha \in  \End(E) : \pi \alpha \pi^{-1} \in \End(E^{\prime}) \}
$
is a CM action of $\OO_D$ on $x$ satisfying $\iota( \mathfrak {n}) \ker \pi =0$.
\end{enumerate}
It actually descends to a DM stack over $\Z$.
The forgetful map $$\mathcal{Z}(m , \mu) \rightarrow \mathcal{X}_{0}(N)$$  $$  (\pi : E\rightarrow E^{\prime}, \iota)\rightarrow (\pi : E\rightarrow
E^{\prime}) $$
is a finite   and \'etale map. Then $\mathcal{Z}(m , \mu)$ is a  Cartier divisor on  $\mathcal X_0(N)$.
\begin{lemma}\cite[Lemma 6.10]{BEY}
Let  $c$ be the conductor of the order $\OO_D$ and $N'=(N, c)$.
Let $\bar{Z}(m, \mu)$ be the Zariski closure of $Z(m, \mu)$ in $ \mathcal X_0(N)$. Then there exists an isomorphism
$$\mathcal{Z}(m , \mu)\cong \bar Z(m, \mu)$$
as stacks over $\Z[\frac{1}{N'}]$.
\end{lemma}
 
 We assume that $N$ is square free for easier.
When $p|N$, the special fiber  $\mathcal X_0(N) \pmod p$ has two irreducible components $\mathcal X_p^\infty$ and $\mathcal X_p^0$. Here we denote the component which contain the cusp $\mathcal P_\infty \pmod p$ and $\mathcal P_0 \pmod p$ by $\mathcal X_p^\infty$ and $\mathcal X_p^0$,
and $\mathcal P_\infty$ and $\mathcal P_0$ are Zariski closure of cusp infinity and zero.

Following the Gillet-Soul\'e intersection theory  \cite{GS},
a height  pairing has been  established for
the arithmetic Chow group $\widehat{\CH}_\R^1(\mathcal X_0(N))$.
 However, this height  pairing is insufficient.
It  has been extended to arithmetic divisors with log-log singularities  (\cite{BKK}, \cite{Kuhn2}),
 and arithmetic divisor with $L_1^2$-Green functions \cite{Bost}.
Similarly as in the work \cite{DY1},   we will employ  K\"uhn's method here.

Let S=\{cusp\} and let  $\widehat{\CH}_\R^1(\mathcal X_0(N), S)$  denote the quotient of the $\R$-linear combinations of the arithmetic divisors of $\mathcal  X_0(N)$ that exhibit log-log growth along $S$ divided by $\R$-linear combinations of the principal arithmetic divisors with log-log growth along $S$. For an arithmetic divisor
 $\widehat{\mathcal Z}=(\mathcal Z, g) \in \widehat{\CH}_\R^1(\mathcal X_0(N), S)$ with  log-log-singularity along $S$,  the function $g$ is  smooth  on $X_{0}(N)\setminus \{\mathcal Z(\C)\cup S \}$,
and satisfies the following conditions:
 \begin{align}
 dd^c [g] + \delta_Z =[\omega]&,
 \\
~~near  ~ S_j, ~~
 g(t_j) = -2 \alpha_j \log\big( -\log(|t_j|^2)\big)&-2 \beta_j \log |t_j| -2 \psi_j(t_j)  
 \end{align}
 where $\psi_j$ is a  smooth function, $\omega$ is a $(1, 1)$-form which is smooth away from $S$, and $t_j$ is a local parameter at cusp $S_j$.

There exists an extension  height paring \cite[Proposition 1.4]{Kuhn1}
\begin{equation}\label{equinter}
\widehat{\CH}_\R^1(\mathcal X_0(N), S) \times \widehat{\CH}_\R^1(\mathcal X_0(N), S) \rightarrow \R,
\end{equation}
such that if $\mathcal Z_1$ and $\mathcal Z_2$ are   divisors  intersect properly, then
$$
\langle (\mathcal Z_1, g_1),  (\mathcal Z_2, g_2)  \rangle_{GS}
=(\mathcal Z_1.\mathcal Z_2)_{fin} + \frac{1}2 g_1*g_2.
$$
The star product is defined as
\begin{align}
g_1*g_2&= g_1( Z_2 -\sum_{j} \ord_{S_j} (Z_2) S_j)
 + 2 \sum_{j} \ord_{S_j} Z_2 \left( \alpha_{1, j} -\psi_{1, j}(0)\right)\nonumber
 \\
&\quad  -\lim_{\epsilon \rightarrow 0} \left( 2 \sum_j (\ord_{S_j}Z_2) \alpha_{1, j} \log(-2 \log \epsilon)-\int_{X_\epsilon} g_2 \omega_1\right).
\end{align}
 For $\epsilon >0$, let $B_\epsilon(S_j)$ be the open disc of radius $\epsilon$  centered at $S_j$, and $X_\epsilon= X_0(N)- \bigcup_{j} B_\epsilon(S_j)$.


%

%
We  view  the metrized line bundle as an arithmetic divisor. 
Let $\omega_N$ be the Hodge bundle and $ \mathcal M_k(\Gamma_0(N))$  be the line bundle of weight $k$ modular form on $\mathcal X_0(N)$.  The normalized Petersson  metric for modular forms  gives a metrized line bundle  \begin{equation}\widehat{\omega}_N^k\cong \widehat{\mathcal M_k}(\Gamma_0(N)).
\end{equation}

 For a modular form $f$ of weight $k$, the  normalized Petersson norm is defined by
\begin{equation} \label{eq1.11}
\| f(z)\|_{Pet} = |f(z) (4 \pi e^{-C} y)^{\frac{k}2}|
\end{equation}
where $C=\frac{\log 4\pi +\gamma}{2}$ and $\gamma$ is Euler constant.
The modular  form $\Delta_N(z) $ is constructed  in \cite[(1.6)]{DY1},
\begin{equation} \label{eq:DeltaN}
\Delta_N(z) =\prod_{t|N} \Delta(t z) ^{a(t)}
\end{equation}
with
$$
a(t) = \sum_{r|t} \mu(\frac{t}r) \mu(\frac{N}r) \frac{\varphi(N)}{\varphi(\frac{N}r)},
$$
where $\mu(n)$ is the the M\"obius function, $\varphi(N)$ is the Euler function and weight $k=12\varphi(N)$.
Then  we can identify
\begin{equation}\widehat{\omega}_N=\frac{1}{k}(\Div(\Delta_N), -\log \| \Delta_N\|_{Pet}^2),\end{equation}
where
\begin{equation}
\Div\Delta_{N}=\frac{tk}{12}\mathcal{P}_{\infty}-k \sum_{p|N}\frac{p}{p-1}\mathcal{X}_{p}^{0},~ ~~t=N\prod_{p| N}(1+p^{-1}).
\end{equation}

\subsection{Arithmetic theta function}

For $r>0$ and $s \in \R$, let
\begin{equation}\label{belta}
\beta_s(r) =\int_1^\infty e^{-rt } t^{-s} dt ,
\end{equation}
and
\begin{equation}
\xi_{\Delta}(w, z)=\beta_1(2\pi R(w, z)_ \Delta).
\end{equation}
For $n \in \sgn(\Delta)Q(\mu)+\Z$,  the twisted Kudla's Green functions is defined in\cite[Section 3]{DY2} as

\begin{equation} \label{eq:TWGreen}
\Xi_{\Delta, r}(n , \mu, v)(z)=\sum_{0\ne w\in L_{r\mu}[n|\Delta| ]}\chi_{\Delta}(w)\xi_{\Delta}(\sqrt{v}w, z),
\end{equation}
which is a $\Gamma_0(N)$-invariant function as $z$.

\begin{theorem}\label{Kudlagreen}
\
 1) \cite{Kucentral}When $ n>0$, $\Xi_{\Delta, r}(n , \mu, v)(z)$  is a Green function for $Z_{\Delta, r}(n, \mu)$ on $Y_{0}(N)=\Gamma_0(N) \backslash \H$, and satisfies the following current equation,
\begin{equation}d d^c  [\Xi_{\Delta, r}(n , \mu, v)(z)] +\delta_{Z_{\Delta, r}(n, \mu)}=[\omega_{\Delta, r}(n, \mu, v)],\nonumber
\end{equation}
where  $\omega_{\Delta, r}(n, \mu,v)$ is the differential form
\begin{equation}
\omega_{\Delta, r}(n, \mu,v)=\sum_{w \in L_{r\mu}[n |\Delta|]}\chi_{\Delta}(w)\varphi^{0}_{\Delta}(w, z).\nonumber
\end{equation}
Moreover,  when $ n\le 0$, $\Xi_{\Delta, r}(n , \mu, v)(z)$ is smooth on $Y_0(N)$.\\
2)\cite{DY1}\cite{DY2} Around cusps, 
$$\Xi_{\Delta, r}(n , \mu, v)(z)~has~\begin{cases}
\log~singularities &\ff \Delta=1, -4Nn=\square>0,\\
\log(-\log)~singularities &\ff \Delta=1, n=0,\\
 no~singularities &\ff others.
\end{cases}$$

\end{theorem}

When $n\neq 0$, the 
arithmetic Heegner divisors   are defined in \cite{DY1} and \cite{DY2} as follows,
\begin{equation}
\widehat{\mathcal Z}_{\Delta, r}(n, \mu, v)
 =\begin{cases}
  (\mathcal Z_{\Delta, r}(n, \mu), \Xi_{\Delta, r}(n, \mu, v)) &\ff n >0,
  \\
  (g(n, \mu, v) \sum_{P \hbox{ cusps}} \mathcal P, \Xi_{\Delta, r}(n, \mu, v)) &\ff  -4Nn=\square>0,
 \\
  (0, \Xi_{\Delta, r}(n, \mu, v)) &\ff  others,
 \end{cases}
\end{equation}
where $\mathcal{Z}_{\Delta, r}(n, \mu)$ and $\mathcal P$ are the Zariski closures of $Z_{\Delta, r}(n, \mu)$ and  cusp $P$ in $\mathcal X_0(N)$.  When $\Delta=1$,
 $\Xi_{\Delta, r}(n, \mu, v)$ has log singularities at cusps, and the multiplicity $g(n, \mu, v)$ is given in  \cite{DY1} as follows
$$
g(n, \mu, v) =\begin{cases}
 \frac{\sqrt N}{4 \pi  \sqrt{v}}\beta_{3/2}(-4 nv \pi)   &\ff n \ne 0, \mu \notin  \frac{1}2 L/L,
 \\
  \frac{\sqrt N}{2 \pi  \sqrt{v}}\beta_{3/2}(-4 nv \pi)   &\ff n \ne 0, \mu \in  \frac{1}2 L/L,
 \\
   \frac{\sqrt N}{2 \pi \sqrt{v}}  &\ff n=0,  \mu =0,
 \end{cases}
$$
 when $\Delta\neq 1$, the multiplicity $g(n, \mu, v)=0$.

  When $\Delta=1$, we define
 \begin{equation}
 \widehat{\mathcal Z}_{\Delta, r}(0, 0, v)=
  (g_{\Delta, r}(0, 0, v) \sum_{P \hbox{ cusps}} \mathcal P, \Xi_{\Delta, r}(0, 0, v)) + \widehat{\omega} -(0, \log(\frac{v}{N})),
  \end{equation}
where 
\begin{equation}
\widehat{\omega}:=-\widehat{\omega}_N-W_N^*\widehat{\omega}_N=-2\widehat{\omega}_N
  - \sum_{p|N} \mathcal X_p^0 ;
\end{equation}
when $\Delta\neq 1$, define 
\begin{equation}
\widehat{\mathcal Z}_{\Delta, r}(0, 0, v)= (0, \Xi_{\Delta, r}(0, 0, v)).
\end{equation}

It is known that   all  these arithmetic divisors belong to  $ \widehat{\CH}^1_\R(\mathcal X_0(N))$, which is the arithmetic Chow group with real coefficients in the sense of Gillet-Soul\'e.  

\begin{theorem}\cite[Theorem 1.1]{DY1}\cite[Theorem 1.6]{DY2}
 The generating function
\begin{equation} \label{eq:GeneratingFunction1}
\widehat{\phi}_{\Delta, r}(\tau) =  \sum_{\substack{ n \equiv\sgn(\Delta) Q(\mu) \pmod \Z\\
                          \mu \in L^\sharp/L
                          }}
 \widehat{\mathcal Z}_{\Delta, r}(n, \mu, v) q_\tau^n e_\mu,
\end{equation}
is a vector valued modular form for $\Gamma'$ of weight $\frac{3}2$, valued in $\C[L^\sharp/L]\otimes \widehat{\CH}^1_\R(\mathcal X_0(N))$. Here $\Gamma'$ acts on $\C[L^\sharp/L]$ via the Weil representation
$\tilde{\rho}_L$ and  $q_\tau=e(\tau)$.
\end{theorem}
Then $\widehat{\phi}_{\Delta, r}(\tau)$ is called an arithmetic theta function.

\subsection{Decomposition of $\widehat{\CH}_\R^1(\mathcal X_0(N))$}
%
%

Let $\nu$ be a smooth, positive $(1, 1)$-form  on $X_0(N)$, and let $g(w, z)$ be the $\nu$-admissible Green function.
We define
\begin{equation}g_{\infty}(z)=\ g(P_{\infty}, z),\end{equation}
which is the Green function  associated with the infinity cusp $P_{\infty}$.
We denote the arithmetic divisor by $\widehat{\mathcal P}_{\infty} =(\mathcal{P}_{\infty}, g_{\infty})$.

Let $A(X_0(N))$  to be the space of smooth  functions $f$ on $X$ that are invariant under  conjugation  ($Frob_\infty$-invariant), and   $A^0(X_0(N))$ to be the subspace of functions $f \in A(X_0(N))$ with
$$
\int_{X} f(z) \nu(z)=0.
$$
Then we denote the associated arithmetic divisor by $a(f)=(0, f)$.

The vertical component is defined by 
\begin{equation}
 \text{Vert}=\sum_{p|N} \R \mathcal X^0+\R \mathcal X^\infty.
\end{equation}
We identify the vertical divisor $\mathcal X$ with arithmetic divisor $\widehat{\mathcal X}=(\mathcal X, 0)$.
Then 
we have
\begin{equation}
\mathcal X^{\infty}+\mathcal X^0=2(0, \log p)=2\mathds{1}\log p.
\end{equation}
So $\mathds{1}=(0, 1)$ can be viewed as a vertical arithmetic divisor.
It follows that
\begin{equation}
\text{Vert}=\sum_{p|N} \R \mathcal Y_p^\vee+\R \mathds{1}.
\end{equation}
Here
$\mathcal Y_p^\vee=\frac{1}{\langle\mathcal Y_p, \mathcal Y_p \rangle} \mathcal Y_p$, 
$\mathcal Y_p=\mathcal X^0_p-p\mathcal X^{\infty}_p$ with $\langle \mathcal Y_p, \hat\omega_N\rangle_{GS}=0$.

\begin{proposition} \label{prodecom}
(\cite[Propositions 4.1.2, 4.1.4]{KRYBook})
\begin{equation}
\widehat{\CH}_\R^1(\mathcal  X_0(N)) = \widetilde{\MW}\oplus (\R \widehat{\mathcal P}_{\infty} \oplus  \text{Vert}   \oplus a(A^0(X)).
\end{equation}
\end{proposition}
For every $\widehat{\mathcal Z} =(\mathcal Z, g_Z)$, it decomposes into
$$
\widehat{\mathcal Z}=\widehat{\mathcal Z}_{MW} + \deg (\widehat{\mathcal Z})\widehat{\mathcal P}_{\infty}   + \sum_{p|N} \langle \widehat{\mathcal Z}, \mathcal Y_p\rangle_{GS} \mathcal Y_p^\vee + 2 \kappa(\widehat{\mathcal Z}) \mathds{1} + a(f_{\widehat {\mathcal Z}})
$$
for some $f_{\widehat{\mathcal Z}} \in A^0(X)$, where
$$
 \kappa(\widehat{\mathcal Z}) =\langle \widehat{\mathcal Z}-\deg (\widehat{\mathcal Z}) \widehat{\mathcal P}_{\infty}, \widehat{\mathcal P}_{\infty}  \rangle_{GS}.
$$

According to this decomposition, we have
\begin{proposition}\label{Prothedeco}
\begin{equation}
 \widehat{\phi}_{\Delta, r}= \widehat{\phi}_{\MW}+\deg(\widehat{\phi}_{\Delta, r})\widehat{\mathcal P}_\infty    + \widehat{\phi}_{\text{Vert}} + a(\phi_{SM}).
\end{equation}
\end{proposition}
More precisely, the vertical component is equal to 
\begin{equation}
\widehat{\phi}_{\text{Vert}}=\Sigma_{p \mid N}\phi_p(\tau)\mathcal{Y}_p^{\vee}+ 2\phi_1(\tau) \mathds{1}, 
\end{equation}
where
\begin{equation}
 \phi_p=\langle  \widehat{\phi}_{\Delta, r},  \mathcal{Y}_p \rangle_{GS}
\end{equation}
and
\begin{equation}
\phi_1(\tau)=\langle \widehat{\phi}_{\Delta, r}(\tau)-\deg(\widehat{\phi}_{\Delta, r}(\tau))\widehat{\mathcal P}_{\infty}, \widehat{\mathcal P}_{\infty}\rangle_{GS}.
\end{equation}

The archimedean part of $ \widehat{\phi}_{\Delta, r}(\tau)$ is given as follows.
\begin{lemma}
 \begin{equation}\label{Greenfundecom}
 \Xi_{\Delta, r}(\tau, z)= g_{MW}(\tau, z)  +\deg(\widehat{\phi}_{\Delta, r}) g_{\infty}(z) +2\phi_1(\tau)+\phi_{SM}(\tau,z),
 \end{equation}
where \begin{equation}
   g_{MW}(\tau, z)=\sum_{n, \mu}g_{MW}(n, \mu, v)(z)q^ne_\mu,
      \end{equation}
and
 \begin{equation}
 \phi_{SM}(\tau, z)=\sum \phi_{SM}(n, \mu,v;z) q^n e_\mu.
 \end{equation}
Each term in the summation   is a modular form of weight $3/2$.
Here
 $g_{MW}(n, \mu, v)$ is the $\nu$-admissible Green function for the divisor $$y_{\Delta, r}(n, \mu):=Z_{\Delta, r}(n, \mu)-\deg(Z_{\Delta, r}(n, \mu))P_{\infty},$$
and $\phi_{SM}(n, \mu,v; z) \in A^0(X)$  are smooth functions.

\end{lemma}

\section{Arithmetic theta lift}\label{Sec3}

In this section, we will study the arithmetic theta lift as follows.
\begin{definition}
For any cusp form $g \in S_{\frac{3}{2}, \tilde\rho_L}$, define the arithmetic theta lift by
\begin{equation}
\widehat{ \theta}_{\Delta, r}(g)=\langle\widehat{\phi}_{\Delta, r},g\rangle_{Pet} \in \widehat{\CH}_\R^1(\mathcal X_0(N)).
\end{equation}
\end{definition}

We will prove the following result (Theorem \ref{thelift}) in this section.

\begin{theorem}\label{Theolift}
Let the notation be as above. Then 
\begin{enumerate}
  \item 
$\widehat{ \theta}_{\Delta, r}(g)=\widehat{\phi}_{MW}(g).$
  \item 
$\widehat{ \theta}_{\Delta, r}(g) \in \widetilde{\MW}.$
  \item For any $\widehat{\mathcal{Z}} \in \widetilde{\MW}^{\perp}$,  we have
$
\langle \widehat{ \theta}_{\Delta, r}(g), \widehat{\mathcal{Z}}  \rangle_{\GS}=0.
$
\end{enumerate}
\end{theorem}

\subsection{$\widehat{\mathcal P}_\infty$ component}
We will prove that the following result in this subsection
\begin{equation}\langle \widehat{\theta}_{\Delta, r}(g), \widehat{\mathcal P}_{\infty}\rangle_{GS}=0.
\end{equation}
\begin{lemma}\label{lemdeltaone}
\begin{equation}
\langle \widehat{\phi}_{\Delta, r}(\tau), \widehat{\mathcal P}_\infty \rangle_{GS}
=\frac{1}{2}\int_{X_0(N)}g_{\infty}(z)\Theta_{\Delta, r}^{KM}(\tau, z).
\end{equation}
\end{lemma}
\begin{proof}
According to \cite[Theorem 6.9]{DY1}, we known that
\begin{eqnarray}\label{eqsquarefour}
&&\langle \widehat{\mathcal Z}_{\Delta, r}(n, \mu, v), \widehat{\Delta}_N\rangle_{GS}\\
&=&-\frac{1}{2}\begin{cases}\int_{X_0(N)}\log \| \Delta_N\|_{Pet}^2\omega_{\Delta, r}(n, \mu, v)(z) &\ff n \neq 0 ;\\
\int_{X_0(N)}\log \| \Delta_N\|_{Pet}^2(\omega_{\Delta, r}(0, 0, v)(z)-\frac{dxdy}{2\pi y^2}) &\ff n=0, \mu=0,\\
0 &\ff n=0, \mu\neq 0.
\end{cases}\nonumber
\end{eqnarray}
where 
\begin{equation}
\widehat{\Delta}_N=(\frac{tk}{12}\mathcal{P}_{\infty}, -\log \| \Delta_N\|_{Pet}^2).\nonumber
\end{equation}
By the same argument, replace  $\widehat{\Delta}_N$ by $\widehat{\mathcal P}_{\infty} =(\mathcal{P}_{\infty}, g_{\infty})$, we have 
\begin{eqnarray}\label{eqsquarefour}
&&\langle \widehat{\mathcal Z}_{\Delta, r}(n, \mu, v), \widehat{\mathcal P}_{\infty}\rangle_{GS}\\
&=&\begin{cases}\frac{1}{2}\int_{X_0(N)}g_{\infty}(z)\omega_{\Delta, r}(n, \mu, v)(z) &\ff n \neq 0 ;\\
\frac{1}{2}\int_{X_0(N)}g_{\infty}(z)(\omega_{\Delta, r}(0, 0, v)(z)-\frac{dxdy}{2\pi y^2}) &\ff n=0, \mu=0.
\end{cases}\nonumber
\end{eqnarray}

By the definition of Kudla-Millson theta function $\Theta_{\Delta, r}^{KM}(\tau, z)$ in Section \ref{Sec2},
we obtain the result.
\end{proof}

\begin{proposition}\label{proinftyinter}
\begin{equation}
\langle \widehat{\theta}_{\Delta, r}(g), \widehat{\mathcal P}_{\infty}\rangle_{GS}=0.
\end{equation}
\end{proposition}
\begin{proof}
According to \cite[Theorem 5.1]{Alfes}, for any cusp form $g(\tau) \in S_{\frac{3}{2}, \tilde{\rho}_L}$,
\begin{equation}\label{petzero}
\langle\Theta_{\Delta, r}^{KM}(\tau, z), g(\tau)\rangle_{Pet}=0.
\end{equation}
Following  Lemma \ref{lemdeltaone},
 we have  \begin{equation}
\langle \widehat{\theta}_{\Delta, r}(g), \widehat{\mathcal P}_{\infty}\rangle_{GS}=\frac{1}{2}\int_{X_0(N)}g_{\infty}(z)\langle\Theta_{\Delta, r}^{KM}(\tau, z),~g\rangle_{Pet}=0.
\end{equation}
This concludes the proof.
\end{proof}

\subsection{Vertical component}
The vertical component  of $\widehat{\phi}_{\Delta, r}(\tau)$ is given as
\begin{equation}\label{equver}
\widehat{\phi}_{\text{Vert}}(\tau)=\Sigma_{p \mid N}\phi_p(\tau)\mathcal{Y}_p^{\vee}+ 2\phi_1(\tau) \mathds{1}.
\end{equation}
We will prove the following result in this subsection
\begin{proposition}\label{prover}
\begin{eqnarray}
\langle \widehat{\phi}_{\text{Vert}}, g(\tau)\rangle_{Pet}=0.
\end{eqnarray}
\end{proposition}
\begin{proof}
  It follows from Lemma \ref{lemphione} and equation (\ref{equver}).
\end{proof}


We define the normalized Eisenstein series  as
\begin{equation}\label{vectormodularform}
\mathcal{E}_{L}(\tau,s)=- \frac{s}{4} \pi^{-s-1}\Gamma(s)\zeta^{(N)}(2s)N^{\frac{1}{2}+\frac{3}{2}s}\sum \limits_{ \gamma ^{\prime} \in \Gamma_{\infty} ^{\prime}\diagdown
\Gamma^{\prime}} \big(v^{\frac{s-1}{2}}e_{0} \big)\mid_{3/2, \tilde{\rho}_{L}} \gamma ^{\prime},
\end{equation}
and
$$
\zeta^{(N)}(s) = \zeta(s) \prod_{p|N} (1-p^{-s}).
$$
The following result  has been proved in \cite{DY1}\cite{DY2}.
\begin{proposition}\label{lemvertic}
\begin{equation}
\deg(\widehat{\phi}_{\Delta, r}(\tau))=2\langle \widehat{\phi}_{\Delta, r}(\tau), \mathds{1} \rangle_{GS}=\left\{
                                           \begin{array}{ll}
                                             \frac{2}{\varphi(N)}\mathcal E_L(\tau, 1), & \hbox{$\Delta=1$;} \\
                                             0, & \hbox{$\Delta \neq 1$,}
                                           \end{array}\nonumber
                                         \right. 
\end{equation}
and
\begin{equation}
\langle \widehat{\phi}_{\Delta, r}(\tau), \mathcal X_p^0 \rangle_{GS}=\langle \widehat{\phi}_{\Delta, r}(\tau), \mathcal X_p^\infty\rangle_{GS}
=\left\{
   \begin{array}{ll}
     \frac{1}{\varphi(N)}\mathcal{E}_{L}(\tau, 1)\log p, & \hbox{$\Delta=1$;} \\
     0, & \hbox{$\Delta\neq 1$.}
   \end{array}
 \right.\nonumber
\end{equation}
\end{proposition}


We have the following result.
\begin{lemma}\label{lemphione}
\begin{eqnarray}
\langle \phi_{1}(\tau), g(\tau)\rangle_{Pet}=\langle \phi_{p}(\tau), g(\tau)\rangle_{Pet}=0.
\end{eqnarray}
\end{lemma}
\begin{proof}
By Proposition \ref{lemvertic}, $\deg{\widehat{\phi}_{\Delta, r}(\tau)}$
 is a constant multiple of $\mathcal E_L(\tau, 1)$.
Then we have
\begin{equation}\label{equdrezero}
\langle\deg{\widehat{\phi}_{\Delta, r}(\tau)}, g\rangle_{Pet}=0.
\end{equation}
According to Proposition  \ref{proinftyinter},  we have
\begin{eqnarray}\label{equonezero}
\langle \langle \widehat{\phi}_{\Delta, r}(\tau), \widehat{\mathcal P}_{\infty}\rangle_{GS}, g(\tau)\rangle_{Pet}=0.
\end{eqnarray}

%
Combining it with equation (\ref{equdrezero}), we obtain
\begin{eqnarray}\label{equzeropartone}
&&\langle \phi_{1}(\tau), g(\tau)\rangle_{Pet}\\
&=&\langle \langle \widehat{\phi}_{\Delta, r}(\tau), \widehat{\mathcal P}_{\infty}\rangle_{GS}, g(\tau)\rangle_{Pet}-\langle\deg{\widehat{\phi}_{\Delta, r}(\tau)}, g\rangle_{Pet} \langle \widehat{\mathcal P}_{\infty}, \widehat{\mathcal P}_{\infty}\rangle_{GS}=0.\nonumber
\end{eqnarray}
According to Proposition \ref{lemvertic}, $\phi_p(\tau)$ is a constant multiple of $\mathcal E_L(\tau, 1)$.
It follows that \begin{equation}\langle \phi_{p}(\tau), g(\tau)\rangle_{Pet}=0.
\end{equation}
Thus, we finish the proof.
\end{proof}

\subsection{Smooth component and spectral decomposition}
In this subsection, we will prove the following result
\begin{equation}
\langle \phi_{SM}(\tau, z) , g(\tau)\rangle_{Pet}=0.
\end{equation}

Let $\Delta_z$ be the Laplacian operator with respect to $\nu$ such that
\begin{equation}
d_z  d_z^c f =\frac{1}{2}\Delta_z(f)  \nu.
\end{equation}
 Then  the space $A^0(X)$  has an orthogonal normal basis $\{ f_{\lambda_j}\}$ with
$$
\Delta_z f_{\lambda_j} + \lambda_j f_{\lambda_j}=0,  \quad  \langle f_{\lambda_j}, f_{\lambda_j} \rangle = \delta_{ij}, \quad \hbox{ and } \lambda_0=0 < \lambda_1 < \lambda_2 < \cdots,
$$
where  the inner product is given by
$$
\langle f, g \rangle = \int_{X_0(N)} f \bar g  \nu.
$$
Then for any $f \in  A^0(X)$, one has
 $$f = \sum  \langle f, f_\lambda \rangle   f_\lambda.$$
It follows that
 \begin{equation}\label{spectraldecom}
 \phi_{SM}(\tau, z) =\sum_{\lambda>0} \langle \phi_{SM} , f_\lambda \rangle f_\lambda.
 \end{equation}
By \cite[Theorem 8.4]{DY1}, we have
 \begin{align*}
 \langle \phi_{SM} , f_\lambda \rangle=
-\frac{2}{\lambda}\Theta_{\Delta, r}^{KM}(\tau,  f_\lambda),\nonumber
\end{align*}
where  \begin{equation}
\Theta_{\Delta, r}^{KM}(\tau,  f)=\int_{X_0(N)} \Theta_{\Delta, r}^{KM}(\tau, z) f,
\end{equation}
and $$\Theta_{\Delta, r}^{KM}(\tau,  f_0)=\int_{X_0(N)}\Theta_{\Delta, r}^{KM}(\tau,  z),~ for ~f_0=1.$$
Thus we obtain the following spectral decomposition.
\begin{lemma}
\begin{equation}\label{specone}
 \phi_{SM}(\tau, z) =-2\sum_{\lambda>0}\lambda^{-1} \Theta_{\Delta, r}^{KM}(\tau,  f_\lambda) f_\lambda,
\end{equation}
and
\begin{eqnarray}\label{spectwo}
\Theta_{\Delta, r}^{KM}(\tau, z)=\sum_{\lambda}\Theta_{\Delta, r}^{KM}(\tau,  f_\lambda) f_\lambda.
\end{eqnarray}

\end{lemma}
Let $g(w, z)$ be the real function on $X_{0}(N)\times X_{0}(N)$ such that
$$d_zd_z^c[g(w, z)]+\delta_w=[\nu(z)],$$ which is a Green function associated to $\nu$.


\begin{proposition}\label{prosmooth}
Let the notation be as above. Then
\begin{equation}\label{prophism}
\phi_{SM}(\tau, z)=-\int_{X_0(N)}g(w, z)\Theta_{\Delta, r}^{KM}(\tau, w).
\end{equation}
\end{proposition}
\begin{proof}

For $f(z) \in  A^0(X)$,
by the current function $$ d_z d_z^c [g(w, z)]+\delta_w=[\nu(z)],$$
we have \begin{eqnarray}
\int_{X_0(N)}g(w, z)\Delta_z f(z)\nu(z)&=&2\int_{X_0(N)}g(w, z) d_z d_z^c  f(z)\nonumber\\
&=&2\bigg(\int_{X_0(N)}f(z)\nu(z)-f(w)\bigg)\nonumber\\
&=&-2f(w).\nonumber
\end{eqnarray}
Thus
\begin{equation}\label{equfirst}
f(w)=-\frac{1}{2}\int_{X_0(N)}g(w, z)\Delta_z f(z)\nu(z).
\end{equation}

According to the proof of \cite[Theorem 8.5]{DY1}, we have
\begin{eqnarray}\label{equdiffsmooth}
 d_z d_z^c \phi_{SM}(
\tau, z)   &=&\sum d_z d_z^c\phi_{SM}(m, \mu,v;z) q^m e_\mu\nonumber\\
&=& \Theta_{\Delta, r}^{KM}(\tau, z)
  -\deg\widehat{\phi}_{\Delta,r}\nu(z).
\end{eqnarray}
Here $\phi_{SM}(m, \mu,v;z)$ is an element of $ A^0(X)$.
By equations (\ref{equfirst}) and (\ref{equdiffsmooth}),
\begin{equation}
 \phi_{SM}(\tau, z)=-\int_{X_0(N)}g(w, z)\Theta_{\Delta, r}^{KM}(\tau, w).
 \end{equation}

Thus we finish the proof.
\end{proof}

Now we prove the following  result.
\begin{theorem}\label{lemmaphizero}
If $g \in S_{\frac{3}{2}, \tilde{\rho}_L}$,  then
\begin{equation}
\langle \phi_{SM}(\tau, z) , g(\tau)\rangle_{Pet}=0.
\end{equation}

\end{theorem}
\begin{proof}
By Proposition \ref{prosmooth}, we have
\begin{eqnarray}\label{equpetsmooth}
&&\langle \phi_{SM}(\tau, z) , g(\tau)\rangle_{Pet}\\
&=&- \bigg\langle\int_{X_0(N)}g(w, z)\Theta_{\Delta, r}^{KM}(\tau, w),  g(\tau)\bigg\rangle_{Pet} \nonumber\\
&=&-\int_{X_0(N)}g(w, z) \langle\Theta_{\Delta, r}^{KM}(\tau, w), g(\tau)\rangle_{Pet}=0. \nonumber
\end{eqnarray}
 Thus we finish the proof.
\end{proof}

\subsection{Proof of Theorem \ref{Theolift}}
\

{\bf Proof.} 
According to Proposition \ref{prodecom},
we  have 
\begin{equation}\label{profirst}
 \widehat{\phi}_{\Delta, r}= \widehat{\phi}_{\MW}+\deg(\widehat{\phi}_{\Delta, r})\widehat{\mathcal P}_\infty    + \widehat{\phi}_{\text{Vert}} + a(\phi_{SM}).\end{equation}
According to Theorem \ref{lemmaphizero} and Proposition \ref{prover}, we have
\begin{eqnarray}\label{prosec}
\langle \phi_{SM}(\tau, z) , g(\tau)\rangle_{Pet}=\langle \widehat{\phi}_{\text{Vert}}, g(\tau)\rangle_{Pet}=0.
\end{eqnarray}
Recall that 
\begin{equation}\label{prothir}
 \langle \deg(\widehat{\phi}_{\Delta, r}) , g\rangle_{Pet}=0
\end{equation}
By this equation and equation (\ref{prosec}), we obtain that
\begin{equation}
  \widehat{ \theta}_{\Delta, r}(g)=\langle \widehat{\phi}_{MW}(\tau), g(\tau) \rangle_{Pet}.
\end{equation}
It implies that 
\begin{equation}
\widehat{ \theta}_{\Delta, r}(g) \in \widetilde{\MW},
\end{equation} and 
 $$\langle \widehat{ \theta}_{\Delta, r}(g), \widehat{\mathcal{Z}} \rangle_{GS}=0,~ for~ any ~\widehat{\mathcal{Z}} \in \widetilde{\MW}^{\perp}.$$
Thus we finish the proof.

\part{ Green functions and CM values}

\section{Automorphic Green functions }\label{Sec4}
In this section, we mainly study the twisted theta lift of non-holomorphic Hejhal-Poincar\'e series.

\subsection{Automorphic Green functions}

For $s>0$, $t>1$,
the Legendre function of the second kind is defined as follows
\begin{equation}
Q_{s-1}(t)=\int_0^\infty(t+\sqrt{t^2-1}\cosh u)^{-s}du.
\end{equation}
It can  be expressed as
\begin{equation}
Q_{s-1}(t)=\frac{\Gamma(s)^2}{2\Gamma(2s)}\bigg(\frac{2}{1+t}\bigg)^sF(s, s, 2s; \frac{2}{1+t}),
\end{equation}
where $F(a, b, c; z)$ is the hypergeometric function. For any points $z, z'\in \H$ and $z\neq z'$, we denote
\begin{equation}
g_s(z, z')=-2Q_{s-1}\bigg(1+\frac{|z-z'|^2}{2\Im z \Im z'}\bigg).
\end{equation}
When  $z\rightarrow z'$, $g_s(z, z')=2\log |z-z'|+O(1)$.
The sum
\begin{equation}\label{defGN}
G_{N, s}(z, z')=\sum_{\gamma \in \Gamma_0(N)}g_s(z, \gamma z'), ~z' \notin \Gamma_0(N)z
\end{equation}
absolutely converges for $\Re(s)>1$.
It is known as the resolvent kernel or automorphic Green function, and
it is also a $\Gamma_0(N)$-invariant function,
$$G_{N, s}(z, z')=G_{N, s}(\gamma z, \gamma z').$$
It has a simple pole of residue $\kappa_N$ at $s=1$, $$\kappa_N=-12N^{-1}\prod_{p| N}(1+p^{-1})^{-1}.$$
 Gross and Zagier constructed the revised Green functions  for Heegner point in \cite{GZ} as,
\begin{equation}
G(z, z')=\lim_{s\rightarrow 1}\big[ G_{N, s}(z, z')+4\pi E_N(w_Nz,s) +4\pi E_N(z', s)+\frac{\kappa_N}{s-1}\big]
+C,\nonumber
\end{equation}
where $C=2\kappa_N-\lambda_N$,
$$\lambda_N=\kappa_N\bigg[ \log N+2\log 2-2\gamma+2\frac{\zeta'}{\zeta}(2)-2\sum_{p| N}\frac{p\log p}{p^2-1}\bigg],$$
and $\gamma$ is the Euler constant.

When $\kappa >1$,  $G_{N, \kappa}(z, z')$ is known as the higher Green functions.


\subsection{Regularized theta integral}

We write  $\widetilde{\SL}_2(\A)$ the twofold metaplectic cover of group   $\SL_2(\A)$. The group  $\widetilde{\SL}_2(\A)$ and $O(V)(\A)$ act on the Schwartz-Bruhat space $S(V(\A))$ via the Weil representation $\omega$ with the standard character $\psi$ of $\A/\Q$.
For any $\varphi \in S(V(\A))$, the theta function is defined by $$\Theta(g, h, \varphi)=\sum_{w \in V(\Q)}\omega(g)\varphi(h^{-1}w),$$
where $g \in\widetilde{ \SL}_2(\A)$ and $h \in O(V)(\A)$.

Taking special Schwartz-Bruhat function $\varphi$, one can obtain different theta functions, and one can see in Section \ref{sectheta}.


 For $\tau=u+iv \in \H$, we set $$g_{\tau}=\kzxz{1}{u}{0}{1}\kzxz{v^{1/2}}{0}{0}{v^{-1/2}}$$
and  $g'_{\tau}=(g_\tau, 1) \in \widetilde{\SL}_2(\R)$.
For any $z \in \D$, the associated Gaussian is defined by $$\varphi_{\infty}(w,z)=e^{-\pi(w, w)_z},$$ which belongs to $S(V(\R))$.
The theta function
\begin{eqnarray}\Theta(\tau, z, h_f, \varphi_f)
&:=&v^{-n/4+1/2}\Theta(g'_\tau, h_f, \varphi_f)\nonumber\\
&=&v^{-n/4+1/2}\sum_{w \in V(\Q)}\omega(g'_\tau)\varphi_{\infty}(w,z)\otimes\varphi_f(h_f^{-1}w)\nonumber\\
&=&v\sum_{w \in V(\Q)}e(Q(w_{z^\perp})\tau+Q(w_z)\bar\tau)\varphi_f(h_f^{-1}w).
\end{eqnarray}

For any $\mu \in L^\sharp/L$, denote $$\phi_{\mu}=char(\widehat{L}+\mu) \in S(V(\A_f)),$$ where $\widehat{L}=L\otimes \widehat{\Z}$.
Bruinier and Yang defined the theta function
\begin{equation}
\Theta(\tau, z, h_f)=\sum_{\mu \in L'/L}\Theta(\tau, z, h_f, \phi_{\mu})\phi_{\mu}.
\end{equation}
We identify $\phi_{\mu}$ with $e_{\mu}$.
Then $\Theta(\tau, z, h_f)$ is a modular form of weight $\frac{n}{2}-1$ like
$$\Theta(\gamma\tau, z, h_f)=\phi(\tau)^{n-2}\rho_L(\gamma')\Theta(\tau, z, h_f).$$

Then the representation $\rho_L$ can be identified with the
restriction to $\Gamma'$ of the complex conjugate of the  Weil representation
$\omega$ on $S(V(\A_f))$.

For any $f \in H_{1-n/2, \bar\rho_L}$,
the regularized theta integral is defined by
\begin{equation}
\Phi(z, h, f)
=\int^{reg}_{\mathcal F}\langle f(\tau), \Theta(\tau, z, h) \rangle d\mu(\tau).
\end{equation}
This regularized integral is defined as the constant term of Laurent expansion. 
When $n=1$, it is given by 
$$CT_{s=0}\lim_{T \rightarrow \infty}\int_{\mathcal F_T}F(\tau)\frac{1}{v^s}d\mu(\tau),$$
where $\mathcal F_T=\{\tau \in \H\mid |u|\leq \frac{1}{2}, v<T ~and~ |\tau|>1\}$ denotes the truncated fundamental domain. 

The twisted regularized theta integral \cite[Section 5]{BO} is defined  as
 \begin{equation}
\Phi_{\Delta,r}(z, h, f)
=\int^{reg}_{\mathcal F}\langle f(\tau), \Theta_{\Delta, r}(\tau, z, h) \rangle  d\mu(\tau).
\end{equation}
Here $\Theta_{\Delta, r}(\tau, z, h)$ is the twisted theta function. The reader can check more details in Section \ref{sectheta}.

\begin{theorem}\cite[Proposition 5.2]{BO}
For any $f \in H_{1/2, \bar{\tilde{\rho}}_{L}}$,
the function $\Phi_{\Delta,r}(z, h, f)$ is smooth on $Y_0(N) \setminus Z_{\Delta, r}(f)$ with  logarithmic singularities along the divisor $-2Z_{\Delta, r}(f)$. If $\Delta_z$ denotes the invariant Laplace operator on $\H$, one has
$$\Delta_z \Phi_{\Delta, r}(z, h, f)=\big( \frac{\Delta}{0}\big)c^+(0, 0),$$
where
$$\big( \frac{\Delta}{0}\big)=
\begin{cases}
1 &\ff \Delta=1,\\
0 &\ff \Delta\neq 1.
\end{cases}$$
\end{theorem}
The function $ \Phi_{\Delta, r}(z, h, f)$ is a Green function for the divisor $$Z_{\Delta, r}(f)+ C_{\Delta, r},$$ where $C_{\Delta, r}$ is a divisor supported at the cusps.
Notice that, when $\Delta \neq 1$, $C_{\Delta, r}=0$.

Now we write the arithmetic divisor
\begin{equation}\widehat{\mathcal Z}_{\Delta, r}(f)=(\mathcal Z_{\Delta, r}(f), ~\Phi_{\Delta, r}(z, f)) \in \widehat{\CH}^1(\mathcal X_0(N))_{\R}.
\end{equation}

\subsection{Theta lift of Non-holomorphic Poincar\'e series}

 Let $k \in \frac{1}{2}\Z$ and $M_{\mu,\nu}(v)$ denote the usual Whittaker functions. We put  
\begin{equation}\mathcal M_{s, k}(v):=v^{-\frac{k}{2}}M_{-\frac{k}{2}, s-\frac{1}{2}}(v).
\end{equation}
For  any pair $(n, \mu)$, $n>0$ and $n\equiv \sgn(\Delta)Q(\mu)(\mod \Z)$,
the non-holomorphic
Hejhal-Poincar\'e series of index $(n,  \mu)$ and weight $k$ is studied in \cite[Chapter 1]{Br},  and it is also generalized in \cite{JKK} and \cite{Alfes},


\begin{equation}
F_{n,\mu}(\tau, s, k)=\frac{1}{2\Gamma(2s)}\sum_{\gamma\in \widetilde{\Gamma_{\infty}}\setminus Mp_2(\Z)}
 \big[\mathcal M_{s, k}(4\pi nv)e(-nu)e_{\mu}\big]\mid_{k, \overline{\tilde\rho}_L}\gamma.
\end{equation}



We denote $s_0=1-\frac{k}{2}$ and define
\begin{equation} F_{n,\mu}(\tau)=F_{n,\mu}(\tau, s_0, k).\end{equation}

It is known that $\mathcal M_{s, k}(4\pi nv) e(-nu)$ is  an eigenfunction of the weight k hyperbolic Laplacian
\begin{equation}\Delta_{k}=-v^2(\frac{\partial }{\partial u^2}+\frac{\partial }{\partial v^2})+
ikv(\frac{\partial }{\partial u}+i \frac{\partial }{\partial v}),\end{equation}
and has eigenvalue $(s-\frac{k}{2})(1-\frac{k}{2}-s)$.
This implies that $F_{n,\mu}(\tau)$ is harmonic and has a principal part
\begin{equation}
e(-n\tau)(e_{\mu}+(-1)^{k+\frac{b^--b^+}{2}}e_{-\mu})+C,
\end{equation}
for some constant $C \in \C[L^\sharp/L]$.
\begin{lemma}
When $k=1/2$, the principal part of $ F_{n,\mu}$ is given by
$e(-n\tau)e_\mu+e(-n\tau)e_{-\mu}+C$, if $\Delta>0$, resp.
$e(-n\tau)e_\mu-e(-n\tau)e_{-\mu}+C$, if $\Delta< 0$.
\end{lemma}

It is known that
\begin{equation}\label{equra}
\frac{1}{(4\pi n)^j}R_{k-2j}^jF_{n, \mu}(\tau, s_0+j, k-2j)=j!F_{n, \mu}(\tau, s_0+j, k).
\end{equation}
For simplicity,  here we write  \begin{equation}R_{k-2j}^j= R_{k-2}\circ\cdots   \circ R_{k-2j}\circ R_{k-2j}.\end{equation}

For any $z \in \H$, let
\begin{equation}
  \lambda(z)=\frac{\sqrt{|\Delta|}}{\sqrt{2N}y}\kzxz{-x}{x^2+y^2}{-1}{x}
\end{equation} be the
normalized vector.
For any $w \in L^\sharp$, we define
$$p_z(w)=-\frac{1}{\sqrt{2}}(w, \lambda(z))_{\Delta}.$$
The Millson theta function is defined in Section \ref{sectheta} as
\begin{equation}
\Theta^{\mathcal M}(\tau, z, h)=v\sum_{\delta \in L^\sharp/L^\Delta}\sum_{w \in h L^\Delta_\delta}p_z(w)e(Q_\Delta(w_{z^\perp})\tau+Q_\Delta(w_z)\bar\tau)e_\delta,
\end{equation}
which is a modular form of weight $1/2$ in $\tau$ and transforms with the representation $\rho_{L^\Delta}$.
 Kudla and Millson  studied it in \cite{KM2}.

 The twisted Millson theta function
is given  by
\begin{eqnarray}
&&\Theta^{\mathcal M}_{\Delta, r}(\tau, z, h)=\psi(\Theta^{\mathcal M}(\tau, z, h))\nonumber\\
&=&v\sum_{\substack{\mu \in L^\sharp/L \\ \pi(\delta)=r\mu\\ Q_\Delta(\delta)\equiv\sgn(\Delta)Q(\mu)( \Z)}} \chi_\Delta(\delta)
\sum_{w \in h L^\Delta_\delta}p_z(w)e(Q_\Delta(w_{z^\perp})\tau+Q_\Delta(w_z)\bar\tau)e_\mu,\nonumber
\end{eqnarray}
which transforms of weight $1/2$ with representation $\tilde{\rho}_L.$

Similarly as in \cite{BEY}, we can define the twisted  regularized theta integral by
\begin{equation}
\Phi^{\mathcal M}_{\Delta,r}(z, h, f)
=\int^{reg}_{\mathcal F}\langle f(\tau), \Theta^{\mathcal M}_{\Delta, r}(\tau, z, h) \rangle  d\mu(\tau).
\end{equation}
Then we have the following result.

\begin{theorem}\label{thetwogreen}
Let the notation be as above. Then we have
\begin{equation}\label{equlift}
\Phi_{\Delta, r}(z,  F_{n, \mu}(~, s, 1/2))=-\frac{2}{\Gamma(s+\frac{1}{4})}G_{N, 2s-\frac{1}{2}}(z, Z_{\Delta, r}(n, \mu));
\end{equation}
and
\begin{equation}
\Phi^{\mathcal M}_{\Delta, r}(z, F_{n, \mu}(~, s, -1/2))=\frac{2\sqrt{n}}{\Gamma(s-\frac{1}{4})}G_{N, 2s-\frac{1}{2}}(z, Z_{\Delta, r}(n, \mu)).
\end{equation}
Especially, when $\Delta \neq 1,$
\begin{equation}\label{cortwogreen}
\Phi_{\Delta, r}(z,  F_{n, \mu})=-2G_{N, 1}(z, Z_{\Delta, r}(n, \mu)).
\end{equation}
\end{theorem}
\begin{proof}

According to \cite[Theorem 2.14]{Br}, utilizing the unfolding methods yields the following equation
\begin{equation}
\Phi_{\Delta, r}(z, F_{n, \mu}(~, s))
=\frac{2}{\Gamma(2s)}\int_0^\infty\int_0^1 \mathcal M_{s, \frac{1}{2}}(4\pi |n|v)e(-nu)\Theta_{\Delta, r, \mu}(\tau, z)v^{-2}dudv,\nonumber
\end{equation}
where \begin{eqnarray}
&&\Theta_{\Delta, r, \mu}(\tau, z)\nonumber\\
&=&v\sum_{\substack{w \in L_{r\mu}\\ Q(w)\equiv \Delta Q(\mu)(\Delta)}}\chi_{\Delta}(w) \exp(-2\pi Q_\Delta(w_{z^\perp})v+2\pi Q_\Delta(w_z)v) e(Q_\Delta(w)u).\nonumber
\end{eqnarray}

From this equation, we derive that
\begin{eqnarray}
&&\Phi_{\Delta, r}(z, F_{n, \mu}(~, s))\\
&=&\frac{2(4\pi |n|)^{-\frac{k}{2}}}{\Gamma(2s)}
\sum_{w \in L_{r\mu}[n|\Delta|]}\chi_{\Delta}(w)\int_0^\infty M_{-sgn(n)\frac{k}{2}, s-\frac{1}{2}}(4\pi |n| v) \nonumber\\
&&\times\exp(-4\pi Q_\Delta(w_{z^\perp})v+2\pi nv)v^{-1-\frac{k}{2}}dv.\nonumber
\end{eqnarray}
By the Laplace transforms, we have
\begin{eqnarray}\label{equtheta}
\Phi_{\Delta, r}(z, F_{n, \mu}(~, s))
&=&\frac{2\Gamma(s-\frac{1}{4})}{\Gamma(2s)}\sum_{w \in L_{r\mu}[|\Delta|n]}\chi_{\Delta}(w)\bigg(\frac{n}{Q_\Delta(w_{z^\perp})}\bigg)^{s-\frac{1}{4}}\nonumber\\
&\times&F\bigg(s-\frac{1}{4}, s+\frac{1}{4}, 2s; \frac{n}{Q_\Delta(w_{z^\perp})}\bigg).
\end{eqnarray}
According to \cite[Proposition 6.2]{BEY}, we have
\begin{eqnarray}
&&\bigg(\frac{n}{Q_\Delta(w_{z^\perp})}\bigg)^{s-\frac{1}{4}}F\bigg(s-\frac{1}{4}, s+\frac{1}{4}, 2s; \frac{n}{Q_\Delta(w_{z^\perp})}\bigg)\\
&=&\frac{2^{\frac{3}{2}-2s}\Gamma(4s-1)}{\Gamma(2s-\frac{1}{2})^2}Q_{2s-\frac{3}{2}}\bigg(1+\frac{|z- z(w)|^2}{2y\Im(z(w))}\bigg).\nonumber
\end{eqnarray}
Combining it with equation (\ref{equtheta}), we obtain
\begin{eqnarray}
&&\Phi_{\Delta, r}(z,  F_{n, \mu}(~, s))\nonumber\\
&=&\frac{2^{\frac{5}{2}-2s}\Gamma(s-\frac{1}{4})\Gamma(4s-1)}{\Gamma(2s)\Gamma(2s-\frac{1}{2})^2}\sum_{\substack{w \in L_{r\mu}[n|\Delta|]\\ Q(w)\equiv \Delta Q(\mu)(\Delta)}}\chi_{\Delta}(w)Q_{2s-\frac{3}{2}}\bigg(1+\frac{|z- z(w)|^2}{2y\Im(z(w))}\bigg).\nonumber\\
&=&-\frac{2}{\Gamma(s+\frac{1}{4})}\sum_{w \in \Gamma_0(N) \setminus L_{r\mu}[n|\Delta|]}\sum_{\gamma\in \Gamma_0(N)}\chi_{\Delta}(\gamma w)g_{2s-\frac{1}{2}}(z, z(\gamma w)).\nonumber
\end{eqnarray}
For any $\gamma \in \Gamma_0(N)$,  $$\gamma z(w)= z(\gamma w),~~\chi_{\Delta}(\gamma w)=\chi_{\Delta}( w).$$
Then it follows that
\begin{eqnarray}
\Phi_{\Delta, r}(z,  F_{n, \mu}(~, s))
=-\frac{2}{\Gamma(s+\frac{1}{4})}G_{N, 2s-\frac{1}{2}}(z, Z_{\Delta, r}(n, \mu)).
\end{eqnarray}
Similarly, we have
\begin{eqnarray}
\Phi^{\mathcal M}_{\Delta, r}(z, F_{n, \mu}(~, s))
&=&\frac{2\Gamma(s+\frac{1}{4})}{\Gamma(2s)}\sum_{w \in L_{r\mu}[|\Delta|n]}\chi_{\Delta}(w)p_z(w)\bigg(\frac{n}{Q_\Delta(w_{z^\perp})}\bigg)^{s+\frac{1}{4}}\nonumber\\
&&\times F\bigg(s+\frac{1}{4}, s-\frac{1}{4}, 2s; \frac{n}{Q_\Delta(w_{z^\perp})}\bigg).
\end{eqnarray}
It is known  in equation (\ref{equdist}) that   $$\sqrt{Q_\Delta(w_{z^\perp})}=|p_z(w)|.$$ Then we have
\begin{eqnarray}
&&\Phi^{\mathcal M}_{\Delta, r}(z, F_{n, \mu}(~, s))
=\frac{2\sqrt{n}}{\Gamma(s-\frac{1}{4})}G_{N, 2s-\frac{1}{2}}(z, Z_{\Delta, r}(n, \mu)).
\end{eqnarray}
Thus we obtain the result.
\end{proof}

%
For
$f \in H_{3/2-\kappa,\overline{\tilde \rho}}$, we define the twisted higher regularized theta lift   by
\begin{equation}\label{equphij}
\Phi_{\Delta, r}^j(z, h, f)=\frac{1}{(4\pi)^j}\times\left\{
                              \begin{array}{ll}
\Phi_{\Delta, r}(z, h, R^j_{\frac{1}{2}-2j}f), & \hbox{$\kappa=2j+1$;} \\
 \Phi^{\mathcal M}_{\Delta, r}(z, h, R_{-\frac{1}{2}-2j}^j f), & \hbox{$\kappa=2j+2$.}
                              \end{array}
                            \right.
\end{equation}
It  is  higher Green function for the divisor
\begin{equation}
Z_{\Delta, r}^j(f)=\sum_{n>0, \mu \in L^\sharp/L}c^+(-n, \mu)n^jZ_{\Delta, r}(n, \mu).
\end{equation}

We denote
\begin{equation}
\Phi^j_{\Delta, r}(z, F_{n,\mu})=\Phi^j_{\Delta, r}(z, F_{n,\mu}(~, 1/4+\kappa/2, 3/2-\kappa)).
\end{equation}

Then we have the following result:

\begin{proposition}\label{proarich}
\begin{equation}
\Phi^j_{\Delta, r}(z, F_{n,\mu})=(-1)^{\kappa}2n^{\frac{\kappa-1}{2}} G_{N, \kappa}(z, Z_{\Delta, r}(n, \mu)).
\end{equation}
\end{proposition}

\begin{proof}
When $\kappa=2j+1$,  we have
\begin{eqnarray}
\Phi^j_{\Delta, r}(z,  F_{n, \mu})&=&\frac{1}{(4\pi)^j}\Phi_{\Delta, r}(z,  R_{\frac{1}{2}-2j}^jF_{n, \mu}( ,3/4+j, 1/2-2j )).\nonumber
\end{eqnarray}
Combining it with equations (\ref{equra}) and (\ref{equlift}), we obtain
\begin{eqnarray}
\Phi^j_{\Delta, r}(z,  F_{n, \mu})&=&j!n^j\Phi_{\Delta, r}(z,  F_{n, \mu}(~, 3/4+j, 1/2))\nonumber\\
&=&-2n^j G_{N, 2j+1}(z, Z_{\Delta, r}(n, \mu)).
\end{eqnarray}
Similarly for $\kappa=2j+2$, we have
\begin{eqnarray}
\Phi^j_{\Delta, r}(z, F_{n, \mu})&=&j!n^j\Phi^{\mathcal M}_{\Delta, r}(z,  F_{n, \mu}(~, 5/4+j, -1/2))\nonumber\\
&=&2n^{j+\frac{1}{2}}G_{N, \kappa}(z, Z_{\Delta, r}(n, \mu)).
\end{eqnarray}

Thus we complete the proof.
\end{proof}

\section{Twisted Heegner divisors}\label{Sec5}

In this section, we will study the translated twisted Heegner divisor. Bruinier, Ehlen and Yang studied the case when level $N=1$ in \cite{BEY}. We will generalize it  to general level $N$.

Let $V$ to be the space $\{x\in M_2(\Q)\mid \tr(x)=0\}$ with quadratic form $Q=N\det$, and $L$ to be the lattice  in $V$ as follows.
 \begin{equation}\label{equlattice}
L=\big\{w =\kzxz {b}{\frac{-a}{N}}{c}{-b}
   \in M_{2}(\Z) |\,   a, b, c \in \Z \big\}.
\end{equation}
We identify  $\Gspin(V)=\GL_2$.
\subsection{Generalized genus character}

Now we recall the definition of the generalized genus character in \cite{GKZ} and \cite{BO}.

For any $w=\kzxz{\frac{b}{2N}}{-\frac{a}{N}}{c}{-\frac{b}{2N}} \in L^\sharp$, define
\begin{equation}\label{genchardef}
\chi_{\Delta}(w) =\begin{cases}
  (\frac{\Delta}{n}),  &\ff \Delta \mid b^2-4Nac~and~ \frac{b^2-4Nac}{\Delta}~is~ a\\
  &  ~square ~modulo~4N~and~(a, b, c, \Delta)=1,
  \\
  0,  &otherwise.
 \end{cases}
\end{equation}
Here $n$ is any integer prime to $\Delta$ represented by one of the quadratic forms $[aN_1, b, cN_2]=aN_1x^2 + b xy + cN_2 y^2$, for any decomposition $N_1N_2=N$. 
Especially, when the level $N=1$
and $D= b^2-4ac$ is a fundamental discriminant,  $\chi_\Delta$  is exactly the genus character that corresponds to the decomposition $D=\Delta \times \frac{D}{\Delta}$.

\begin{lemma}\label{lemcharone}
When $(\Delta, N)=1$,
\begin{equation}
\chi_\Delta(w)=\chi^{N=1}_\Delta(Nw),
\end{equation}
where $\chi^{N=1}_\Delta$ is the generalized character for level $N=1$.
\end{lemma}
\begin{proof}
Assume that $w=\kzxz{\frac{b}{2N}}{-\frac{a}{N}}{c}{-\frac{b}{2N}}$, and the associated quadratic form is  given by $[a, b, Nc]$.

Since $(\Delta, N)=1$, we have
$$(a, b, Nc, \Delta)=1\Leftrightarrow (aN_1, b, cN_2, \Delta)=1,$$
for any $N_1N_2=N$.
Therefore, it suffices to consider the form $[a, b, Nc]$.

We assume that   $(a, b, Nc, \Delta)=1$.
Then the form  $[a, b, Nc]$ represents an integer $n$ which  is prime to $\Delta$.
So we have \begin{equation}\label{equcha}
                                \chi_\Delta(w)=\bigg(\frac{\Delta}{n}\bigg).
                                    \end{equation}

Moreover, we have $Nw=\kzxz{\frac{b}{2}}{-a}{Nc}{-\frac{b}{2}}$ and 
\begin{equation}\label{equchatwo}
                                \chi^{N=1}_\Delta(Nw)=\bigg(\frac{\Delta}{n}\bigg)=\chi_\Delta(w).
                                    \end{equation}
Thus we obtain the result.
\end{proof}

The generalized character $\chi^{N=1}_\Delta$ can be defined locally \cite{BEY}, one can see as follows:

when $p \nmid \Delta$, let $\chi^{N=1}_{\Delta, p}$ be the characteristic function;

when $p \mid \Delta$,
let
\begin{equation}
\chi^{N=1}_{\Delta, p}(w) =\begin{cases}
  (\frac{p^*}{n}),  &\ff (a, b, c, \Delta)=1,
  \\
  0,  &otherwise.
 \end{cases}
\end{equation}
Here $p^*=(-1)^{\frac{p-1}{2}}p$,  and $n$ is any integer  that is prime  to $\Delta$ and  is represented by the quadratic form  $[a, b, c]$.

Thus, we can write it locally as
$$\chi_\Delta(w)=\prod_p\chi_{\Delta, p}(w)=\prod_p\chi^{N=1}_{\Delta, p}(Nw).$$
Now we define  $K_p$ the compact open subgroup of $\GL_2(\Q_p)$ as
\begin{equation}K_p=\{\kzxz{a}{b}{c}{d} \in \GL_2(\Z_p)\mid c \in N\Z_p\}.\end{equation}
\begin{lemma}\label{lemgeneralized}
Assume that $(\Delta, N)=1$. For any $h \in  K_p$, $w\in L_p^\sharp$, then
\begin{equation}\label{equaction}
\chi_{\Delta, p}(h\cdot w)=(\det(h), \Delta)_p\chi_{\Delta, p}(w),
\end{equation}
where $(~, ~)_p$ is the local Hilbert symbol.
\end{lemma}
\begin{proof}
 When $h \in K_p$, it is proved in \cite[Lemma 7.3]{BEY} that
\begin{equation}\chi_{\Delta, p}^{N=1}(h\cdot N w)=(\det(h), \Delta)_p\chi_{\Delta, p}^{N=1}(Nw).
\end{equation}
Combining it with Lemma \ref{lemcharone}, we have
 \begin{eqnarray}
 \chi_{\Delta, p}(h\cdot w)&=&\chi_{\Delta, p}^{N=1}(h\cdot N w)=(\det(h), \Delta)_p\chi_{\Delta, p}^{N=1}(Nw)\nonumber\\
 &=&(\det(h), \Delta)_p\chi_{\Delta, p}(w).\nonumber
 \end{eqnarray}
%
Thus, we obtain the result.
\end{proof}

Let $U$ denote a negative definite 2-dimensional subspace of $V$.
By Clifford algebra, we assume that $U\cong k=\Q(\sqrt{D})$, where $D$ is a fundamental discriminant.
We identify $\Gspin(U)=k^\times$.

%
 We let $\Cl_k$ to be the ideal class group.
It is known that each genus character corresponds to a decomposition of fundamental discriminants \cite{Siegel}.
Let
 $$ \chi: \Cl_k/ \Cl_k^2 \to \{\pm 1\}$$ denote the genus character that is given by the  discriminant decomposition $D=\Delta D_0$. We write $ k^\times \setminus \A_{k, f}$ for the  finite idele class group. Then the 
 map $$ k^\times \setminus \A_{k, f}\rightarrow \Cl_k, ~h=(h_{\mathfrak{p}})\rightarrow [h]=\prod_{\mathfrak{p} \nmid \infty} \mathfrak{p}^{v_{\mathfrak{p}}(h_\mathfrak{p})}$$
  induces an isomorphism
  \begin{equation}
  k^\times \setminus \A_{k, f}/\hat{ \mathcal O}_k^\times \cong  \Cl_k,
  \end{equation}
where $\hat{ \mathcal O}_k^\times = \prod_{\mathfrak{p}\nmid\infty}  \mathcal O_{\mathfrak{p}}^\times$.

We denote the finite Hilbert symbol as $(~, ~)_{\A_f}=\prod_{p <\infty}(~, ~)_p$.
Then we have the following result:

\begin{proposition}\label{prochar}
For any $h \in \Gspin(U)=\A^\times_{k, f} $,
\begin{equation}(\Delta, \det(h))_{\A_f}=\chi([h]). 
\end{equation}

\end{proposition}
\begin{proof}
The Hilbert symbol is  multiplicatively  bilinear. Thus it suffices to prove it locally as follows
\begin{equation}(\Delta, \det(h))_q=\chi([h]), ~for ~h \in k_\mathfrak{q}^\times.
\end{equation}
We will divide the proof into two distinct  cases:  $h=\pi_{\mathfrak{q}}$  and $h \in \mathcal O^\times_{\mathfrak{q}}$.

$(1) ~h=\pi_{\mathfrak{q}}$.


When $\mathfrak{q}$ is  inert, $\pi_{\mathfrak{q}}=q$ is  prime to $\Delta$, we have 
\begin{equation}
(\Delta, \det(\pi_{\mathfrak{q}}))_q=(\Delta, q^2)_q=1=\bigg(\frac{\Delta}{N(\mathfrak{q})}\bigg)=\chi([\mathfrak{q}]).
\end{equation}

 When $\mathfrak{q}$ is split, we have
 \begin{eqnarray}
(\Delta, \det(\pi_{\mathfrak{q}}))_q=(\Delta, q)_q=\bigg(\frac{\Delta}{N(\mathfrak{q})}\bigg)=\chi([\mathfrak{q}]).\nonumber
\end{eqnarray}

When $\mathfrak{q}$ is ramified,  either $q \mid \Delta$ or $q \mid D_0$. Let  $\pi_{\mathfrak{q}}=\sqrt{D} \in k_{\mathfrak{q}}^\times$, and denote the associated maximal ideal by $\mathfrak{q}=(\pi_{\mathfrak{q}})$.

If $q \mid \Delta$,  we find that
\begin{eqnarray}
( \Delta, \det(\pi_{\mathfrak{q}}))_q=( \Delta, -D)_q
=( \Delta, D_0)_q=\bigg(\frac{D_0}{\mathfrak{q}}\bigg)=\chi([\mathfrak{q}]);
\end{eqnarray}
if $q \mid D_0$, we similarly obtain
\begin{eqnarray}
( \Delta, \det(\pi_{\mathfrak{q}}))_q=( \Delta, -D)_q
=( \Delta, D_0)_q=\bigg(\frac{\Delta}{\mathfrak{q}}\bigg)=\chi([\mathfrak{q}]).
\end{eqnarray}
In summary,   for any $\pi_{\mathfrak{q}}$, we have
\begin{equation}\label{equchar}
( \Delta, \det(\pi_{\mathfrak{q}}))_q=\chi([\mathfrak{q}]).
\end{equation}

%

$(2)~h \in \mathcal O^\times_{\mathfrak{q}}$.

When $q \mid \Delta$,
it is easy to know that $N_{k_\mathfrak{q}/ \mathcal \Q_q}( \mathcal O^\times_{\mathfrak{q}}) \subseteq \Z_q^{\times2}$.
Consequently, for any $h \in \mathcal O^\times_{\mathfrak{q}}$,  we find that
$$( \Delta, \det(h))_q=1.$$
Moreover, when $q \nmid \Delta$, $$( \Delta, \det(h))_q=1.$$

Therefore, for any $h \in \mathcal O^\times_{\mathfrak{q}}$, we have
\begin{equation}(\Delta, \det(h))_{\mathfrak{q}}=1=\chi([h]).
\end{equation}
Combining it with equation (\ref{equchar}),   we obtain the result.
\end{proof}

\subsection{Twisted Heegner divisors}

  For  any compact open subgroup $K$ of $\GL_2(\A_f)$, the modular curve
 is defined as follows
 \begin{equation}
X_K=\GL_2(\Q)\backslash \H^\pm \times \GL_2(\A_f)/K.
 \end{equation}
According to the following decomposition
 $$\GL_2(\A_f)=\coprod\GL_2(\Q)_+hK,$$
we have
 \begin{equation}
\coprod\Gamma_h\setminus \H =X_K,
 \end{equation}
where the mapping is given by $z\rightarrow (z,  h)$, 
and $\Gamma_h=hKh^{-1}\bigcap \GL_2(\Q)_+$.


Now we define the compact open subgroup as
$$K=\prod_pK_p,~~K_p=\{\kzxz{a}{b}{c}{d} \in \GL_2(\Z_p)\mid c \in N\Z_p\}.$$

According to the strong approximation,  $\GL_2(\A_f)=\GL_2(\Q)_+ K$,
we assume that $h=\gamma k$, where   $\gamma \in  \GL_2(\Q)_+$ and $k \in K$. Up to a  factor in $\Gamma_0(N)$,  the decomposition of $h$ is unique.

The modular curve $X_K$ has only one connected component and the map
\begin{equation}\label{map}
\Gamma_h\setminus \H =X_{K^h}\cong X_{K}=\Gamma_0(N)\setminus \H ,
\end{equation}
is given by \begin{equation}\label{equmap}
z \mapsto (z, 1)\mapsto (z, h) \mapsto \gamma^{-1}z.
\end{equation}
We define the subgroup of $K$ as follows
$$K_\Delta=\{h \in K \mid (\Delta, \det(h))_{\A_f}=1\},$$
which has index $2$ in $K$.
Then we can write $K=K_\Delta\bigsqcup \xi K_\Delta$. Then $X_{K_\Delta}$ has two components which are given as follows
$$\Gamma_0(N)\setminus \H \bigsqcup \Gamma_\xi\setminus \H  \rightarrow X_{K_\Delta}.$$

It is easy to find that $K_\Delta \bigcap \GL_2(\Q)_+=\Gamma_0(N)$.
For any  $\xi \in K\setminus K_\Delta$,  $\xi K_\Delta \xi^{-1}=K_\Delta$  and $\Gamma_\xi=\xi K_\Delta \xi^{-1}\bigcap \GL_2(\Q)_+=\Gamma_0(N).$


\begin{definition}
For any $h \in  \GL_2(\A_f)$, the translated twisted Heegner divisor   on $X_{K_\Delta}$ is defined as follows
\begin{equation}
Z_{\Delta, r}(n, \mu, h)= \sum_{x \in \Gamma_h \diagdown hL_{r\mu}[|\Delta|n]}\chi_{\Delta}(h^{-1}x)Z(x, h).
\end{equation}
\end{definition}
By virtue of Witt's theorem, we assume the following conditions hold $$\{x\mid x\in V(\Q), Q(x)=|\Delta|n\}= \GL_2(\Q)x_0$$
and
$$hL_{r\mu}[|\Delta|n]=\coprod_iK^hx_i,$$
where  $x_i=h_i^{-1}x_0 \in \Gamma_h \diagdown hL_{r\mu}[|\Delta|n]$ and $h_i \in  \GL_2(\Q)$.
Consequently, we can express
\begin{eqnarray}
Z_{\Delta, r}(n, \mu, h)=\sum \chi_\Delta(h^{-1}x_i) Z(x_i, h).
\end{eqnarray}

We write $h=\gamma h_0 k$, $\gamma \in \GL_2(\Q)_+$, $k \in K_\Delta$, $h_0=1$ or $\xi$.
Then we have
$$L_{r\mu}[|\Delta|n]=\coprod_iK(h_ih)^{-1}x_0=\coprod_iK\gamma^{-1}x_i.$$

It follows that
\begin{eqnarray}
Z_{\Delta, r}(n, \mu)=\sum \chi_\Delta(\gamma^{-1}x_i) Z(\gamma^{-1}x_i).
\end{eqnarray}
According to lemma \ref{lemgeneralized},  we have
\begin{eqnarray}
Z_{\Delta, r}(n, \mu, h)
&=&(\Delta, \det(h))_{\A_f}\sum \chi_\Delta(\gamma^{-1}x_i) Z(x_i, h)\\
&=&
(\Delta, \det(h))_{\A_f}Z_{\Delta, r}(n, \mu).\nonumber
\end{eqnarray}
We denote the function  by
\begin{equation}
G_{N, s}(z, Z_{\Delta, r}(n, \mu, h))
=(\Delta, \det(h))_{\A_f}G_{N, s}( z, Z_{\Delta, r}(n, \mu)).
\end{equation}

We define
\begin{equation}\label{equtrangrz}
G_{N, s}((z, h), Z_{\Delta, r}(n, \mu, h))
=(\Delta, \det(h))_{\A_f}G_{N, s}(\gamma^{-1} z, Z_{\Delta, r}(n, \mu)).
\end{equation}
It  can be viewed a function on $\Gamma_h\setminus \H$ under the map (\ref{equmap}).

\begin{proposition}
For any 
$h \in \GL_2(\A_f)$, we have
\begin{equation}\label{lemtwiphi}
\Phi_{\Delta,r}(z, h, F_{n, \mu})=(\Delta, \det(h))_{\A_f}\Phi_{\Delta,r}(\gamma^{-1}z, F_{n, \mu}).
\end{equation}
Moreover, if $h \in \A^\times_{k, f}$,
\begin{eqnarray}\label{lemtwi}
\Phi_{\Delta,r}(z, h, F_{n, \mu})
=\chi([h])\Phi_{\Delta,r}(\gamma^{-1}z, F_{n, \mu}).
\end{eqnarray}
\end{proposition}
\begin{proof}

By Theorem \ref{thetwogreen}, we have
\begin{eqnarray}
&&\Phi_{\Delta, r}(z, h,  F_{n, \mu})=(\Delta, \det(h))_{\A_f}\Phi_{\Delta,r}(\gamma^{-1}z, F_{n, \mu}).
\end{eqnarray}


By Proposition (\ref{prochar}),  we obtain the second equation.
\end{proof}

For any  negative definite 2-dimensional subspace $U\simeq k \subset V$,  we identify the group $\Gspin(U)=k^\times$, which can be viewed as a subgroup of $\Gspin(V)=\GL_2$. It is known that
 $K\bigcap \A_{k, f}^\times=\hat{ \mathcal O}_k^\times$. 
The CM cycle is defined as follows
\begin{equation}
Z(U)=k^\times\setminus \{z_U^{\pm}\} \times \A_{k, f}^\times/\hat{ \mathcal O}_k^\times\cong\{z_U^{\pm}\} \times \Cl_k \rightarrow X_K,
\end{equation}
where each point  is counted with multiplicity $\frac{2}{w_k}$ and  $w_k=|\mathcal O_k^\times|$.
 Here two points $\{z_U^{\pm}\}$ are  given by $U(\R)$ with the
two possible choices of orientation.


It is also known that
 \begin{equation}\deg(Z(U))=\frac{4}{\vol(\hat{ \mathcal O}_k^\times)}=\frac{4h_k}{w_k}.
 \end{equation}
 
We define the twisted CM values as follows.
\begin{definition}
For
$f \in H_{3/2-\kappa,\overline{\tilde \rho}}$, we define
\begin{equation}\label{defcm}
\Phi_{\Delta, r}^j(Z(U), f)=\frac{2}{w_k}\sum_{(z,h)\in supp(Z(U))}\Phi_{\Delta, r}^j(z, h, f).
\end{equation}
\end{definition}

\section{CM values of automorphic Green functions}\label{Sec6}
According to the work \cite{BY} and \cite{BEY}, when $\Delta =1$, the CM value $\Phi_{\Delta, r}^j(Z(U), f)$ is equal to the sum of  derivative of L-function and a constant term(CT). In this section, we will prove that,
 when $\Delta \neq 1$, the constant term vanishes  under some conditions.
\subsection{Eisenstein series}
Let $V$ be a quadratic space with dimension $m$. We denote the Gram determinant of $V$  by $\det(V)$.
The character is defined by the Hilbert symbol as $$\chi_V(x)=(x, (-1)^{\frac{m(m-1)}{2}}\det(V))_{\A}.$$

 For any standard section $\Phi( ~ ,s) $ 
in the 
induced representation $I(s, \chi_V)$,
the Eisenstein series is defined by 
\begin{equation}
E(g, s, \Phi)=\sum_{\gamma \in P(\Q)\setminus\widetilde{\SL}_2(\Q)}\Phi(\gamma g, ~s),
\end{equation}
where $P(\Q)$ is the parabolic subgroup.
It has a  Fourier
expansion
$$E(g, s, \Phi)=\sum_{n \in \Q}E_n(g, s, \Phi),$$
where 
\begin{equation}
E_n(g, s, \Phi)=\int_{\Q\setminus \A}E(n(b)g,  s, \Phi)\psi(-nb)db, ~~n(b)=\kzxz{1}{b}{}{1}.
\end{equation}
When $\Phi=\otimes \Phi_p$ is  factorizable, it can be computed by 
\begin{equation}
E_n(g, s, \Phi)=\prod_{p\leq\infty} W_{n, p}(g_p, s, \Phi_p)
\end{equation}
where 
\begin{equation}
W_{n, p}(g_p, s, \Phi_p)=\int_{\Q_p}\Phi_p(wn(b)g_p,  s)\psi(-nb)db, ~w=\kzxz{}{1}{-1}{},
\end{equation}
is the local Whittaker function.

For any  Schwartz-Bruhat function $\varphi \in S(V(\A))$, there exists an unique standard section $\lambda(\varphi)=\Phi( ~ ,s) \in I(s, \chi_V)$, such that 
 $$\Phi(g, s_0)=\omega(g)\varphi(0),~~s_0=\frac{m}{2}-1.$$
Here $\omega$ is the Weil representation.

Now we also write $$E(g, s, \varphi)=E(g, s, \Phi)$$ and $$W_{n, p}(g_p, s, \varphi_p)=W_{n, p}(g_p, s, \Phi_p),$$ for easier to read.

For $\mu \in L^\sharp/ L$, we denote the function as $\varphi_\mu=\cha(L_\mu) \in S(V(\A_f))$. 
Then we define the weight $l$ Eisenstein series by
\begin{equation}E_{L}(\tau, s; l)=v^{-\frac{l}{2}}\sum_{\mu}E(g_\tau, s, \Phi^l_\infty\otimes \lambda(\varphi_\mu))\varphi_\mu,\end{equation}
where $\Phi^l_\infty$ is the unique archimedean standard section such that
$$\Phi^l_\infty(\kzxz{cos\theta}{sin\theta}{-sin \theta}{cos\theta}, s)=e^{il\theta},~ \theta \in [0, 2\pi].$$
We identify $\varphi_\mu$ with $e_\mu$. Then this Eisenstein series is a $\C[ L^\sharp/ L]$-valued Eisenstein series.
According to \cite[Section 2]{BY}, we can write it as follows
 \begin{equation}
E_{ L}(\tau, s; l) = \sum \limits_{ \gamma ^{\prime} \in \Gamma_{\infty} ^{\prime}\diagdown
\Gamma^{\prime}} \big(v^{\frac{s+1-l}{2}}e_{0} \big)\mid_{l, \rho_{L}} \gamma ^{\prime}.
\end{equation}

When $V$ is a negative definite two-dimensional space,
$E_{ L}(\tau, s; 1) $
 is an incoherent Eisenstein series and vanishes at $s_0=0$. Furthermore, $ E_{  L}(\tau,s ; -1)$ is holomorphic at $s_0$. The following relationship holds
\begin{equation}
L_1E_{L}(\tau,s ;1)=\frac{s}{2}E_{ L}(\tau,s ;-1),
\end{equation}
where $L_l=-2iv^2\frac{\partial}{\partial \bar\tau}$ is the Maass lowering  operators in weight $l$.

If the lattice $ L$ is replaced by  $ L^\Delta$, we have the same result.

\subsection{Derivatives of Eisenstein series}

Let the lattice $L\subset V$ be defined by equation (\ref{equlattice}).

We assume that $D=-4N |\Delta |m_0=\Delta D_0$ is a fundamental discriminant such that $D \equiv R^2  (\mod 4N)$. 
We choose 
\begin{equation}
w=\kzxz{\frac{R}{2N}}{\frac{1}{N}}{\frac{D-R^2}{4N}}{-\frac{R}{2N}}  \in L_{r\mu_0}[|\Delta |m_0],
\end{equation}
and define  two sublattices of  $L$ as follows
\begin{equation}\label{equn}
\mathcal N=\Z e_1\oplus \Z e_2,~~
 \mathcal P=\Z \frac{2N}{t}w,
 \end{equation}
where $e_1=\kzxz{1}{0}{-R}{-1}$, $e_2=\kzxz{0}{\frac{1}{N}}{\frac{R^2-D}{4N}}{0}$ and $t=( R, 2N)$.

It is  known  in \cite{BY} that
\begin{equation}
\mathcal P^{\sharp}=\Z \frac{t}{D}w,~and ~\mathcal P^{\sharp}\cap L^{\sharp}=\Z w.
\end{equation}

We can write the incoherent Eisenstein series as follows
\begin{eqnarray}
E_{\mathcal N^\Delta}(\tau,s ;1)&=&v^{-\frac{1}{2}}\sum_{\delta\in \mathcal N^\sharp/\mathcal N^\Delta}E(g_\tau, s, \Phi^1_\infty\otimes \lambda(\varphi_\delta))e_\delta\\
&=&\sum_{\delta \in \mathcal N^\sharp/\mathcal N^\Delta}  \sum_{n}A_{\delta}(s, n, v)q^ne_\delta,\nonumber
\end{eqnarray}
 where $\varphi_\delta=\Char(N^\Delta_\delta)$ is the characteristic function of $\mathcal N^\Delta_\delta$.
It has the following Taylor expansion
\begin{equation}
A_{\delta}(s, n, v)=b_\delta(n, v)s+O(s^2).\nonumber
\end{equation}
Consequently, we have
\begin{equation}
E'_{\mathcal N^\Delta}(\tau, s_0 ;1)=\sum_{\delta \in \mathcal N^\sharp/\mathcal N^\Delta} \sum_{n}b_\delta(n, v)q^ne_\delta,
\end{equation}
which is a harmonic weak Maass form of weight $1$. Here $s_0=0$.
The following terms are studied in  \cite{KuIntegral} and \cite{Scho},
\begin{equation}
k(n, \delta)=
\begin{cases}
\lim_{v\rightarrow\infty}b_\delta(n, v), & if~ n\neq 0 ~or~ \delta \neq 0,\\
\lim_{v\rightarrow\infty}b_0(0, v)-\log(v), & if~ n=0 ~and~ \delta=0.
\end{cases}
\end{equation}
For $n>0$, $b_\delta(n, v)$ is independent of $v$, while for $n<0$, $k(n, \delta)=0$.
When $n=0$, $\delta \neq 0$, $k(n, \delta)=0$.

Then we denote
\begin{equation}\label{equFE}
\mathcal E_{\mathcal N^\Delta}(\tau)=\sum_\delta\sum_{n\geq 0} k(n, \delta)q^ne_\delta,
\end{equation}
which is the holomorphic part of $E'_{\mathcal N^\Delta}(\tau, s_0 ;1)$. We will compute $k(n, \delta)$ in this section.

The $(n, \delta)$-th Fourier coefficients of  $E_{\mathcal N^\Delta}(\tau,s ;1)$ are given by the product of local Whittaker functions as
\begin{equation}\label{equfouprod}
E_n(\tau, s, \varphi_\delta)=W_{n, \infty}(\tau, s, \Phi^1_\infty)\prod_{p<\infty} W_{n, p}(s, \varphi_{\delta, p}),\end{equation}
where  $$W_{n, \infty}(\tau, s, \Phi^1_\infty)=v^{-\frac{1}{2}}W_{n, \infty}(g_\tau, s, \Phi^1_\infty)$$
and  $$W_{n, p}(s, \varphi_{\delta, p})=W_{n, p}(1, s, \lambda(\varphi_{\delta, p})).$$

The  archimedean  Whittaker function  $W_{n, \infty}(\tau, s, \Phi^1_\infty)$ is studied in \cite[Section 15]{KRYComp}. We will take  Yang's method \cite{YaDensity} to compute non-archimedean  Whittaker functions $W_{n, p}(s, \varphi_{\delta, p})$.

Yang studied the following  local density,
\begin{equation}
W_p(s_0, n,\delta_p)=\int_{\Q_p} \int_{\mathcal N^\Delta_\delta}\psi(bQ(x))dx\psi(-nb)db,
\end{equation}
where $dx$ is the standard Haar measure on $\Z_p^2$ and $\psi :\A/\Q\rightarrow \C^\times$ is the standard
additive character with $\psi_\infty(x)=e(x)$.

It is known that
\begin{equation}
W_{n, p}(s_0, \varphi_{\delta, p})=\gamma_p|S|_p^{\frac{1}{2}}W_p(s_0, n,\delta_p),
\end{equation}
where  $\gamma_p$ denotes the local splitting  index and $S$ is the Gram matrix. 

By the product formula (\ref{equfouprod}), we obtain the coefficient $E_n(\tau, s, \varphi_\delta)$,
and so $k(n, \delta)$.

We define the subset of primes by
 \begin{equation}
 \Diff(n)=\{p \mid\chi_p(-nN|\Delta|)=-1\},
 \end{equation}
where $$\chi_p(-nN|\Delta|)=(D, -nN|\Delta|)_p.$$
By the product formula, we have
$$\prod_{p\leq \infty}(D, -nN|\Delta|)_p=1.$$
Since $(D, -nN|\Delta|)_\infty=-1$,
 the cardinality of $ \Diff(n)$ is odd.
 
We denote
\begin{equation}
\rho(n)=\sharp\{\textbf{a}\subseteq \mathcal{O}_k| N(\textbf{a})=n\},
\end{equation}
which can be expressed locally as
  $$\rho(n)=\prod_{p}\rho(p^{\ord_p(n)}).$$

We identify  $\mathcal N^\Delta\otimes\Z_p\cong \Z^2_p$, and
the  Gram matrix is given by
\begin{equation}S=\kzxz{-2|\Delta|N}{-|\Delta|R}{-|\Delta|R}{-\frac{R^2-D}{2N}|\Delta|},
\end{equation}
such that
$$\kzxz{1}{}{-\frac{r}{2N}}{1}S\kzxz{1}{-\frac{r}{2N}}{}{1}=\kzxz{-2|\Delta|N}{}{}{\frac{D|\Delta|}{2N}}.$$


By the local density formula given in \cite{YaDensity} and  \cite{KYEisen}, we have the following result.

\begin{lemma}\label{lemlocalwhi}
Assume that $\delta \in \mathcal N$ and $n \in \N$.   Then  
\begin{equation}
\chi_p(-nN|\Delta|)=-1~\Rightarrow ~W_p(s_0, n,\delta_p)=0.
\end{equation}
$(1)$ If $W_p(s_0, n,\delta_p)\neq 0,$ then 
$$
  W_p(s_0, n,\delta_p)= \left\{
                           \begin{array}{ll}
                              \bigg(1-p^{-1}\bigg(\frac{D}{p}\bigg)\bigg)\rho_p(nN), & \hbox{$p\nmid D$;} \\
                               2, & \hbox{$p\mid D_0$.}\\
                            2p, & \hbox{$p\mid \Delta, ~\delta_p=0$;} \\
                               p, & \hbox{$p\mid \Delta, ~\delta_p\neq 0$;} \\
                           
                           \end{array}
                         \right.
$$
$(2)$  If $W_p(s_0, n,\delta_p)= 0,$  then
$$
  W'_p(s_0, n,\delta_p)= \left\{
                           \begin{array}{ll}
                              (1+p^{-1})\frac{\ord_p(nN)+1}{2}\ln p, & \hbox{$p\nmid D$;} \\
                                (\ord_p(n)+1)\ln p, & \hbox{$p\mid D_0$.}\\
                             (p\ord_p(n)+1)\ln p, & \hbox{$p\mid \Delta, ~\delta_p=0$;} \\
                                (p+1)\ln p, & \hbox{$p\mid \Delta, ~\delta_p\neq 0$;} \\
                           
                           \end{array}
                         \right.
$$
\end{lemma}
\begin{proof}

The proof is divided into three distinct cases.

{\bf Case 1 $p \nmid D$.}

According to the density formula,
we have 
\begin{equation}
W_p(s+s_0, n,\delta_p)=\bigg(1-p^{-1}\bigg(\frac{D}{p}\bigg)X\bigg)\sum_{0\leqslant k\leqslant \ord_p(nN)}\bigg(\frac{D}{p}\bigg)^kX^k,
\end{equation}
where $X=p^{-s}$.

We have $$\chi_p(-nN|\Delta|)=-1 \Leftrightarrow \bigg(\frac{D}{p}\bigg)=-1~and~ \ord_p(nN)~is~odd.$$
It follows that
\begin{equation}
\chi_p(-nN|\Delta|)=-1 \Leftrightarrow W_p(s_0, n,\delta_p)=0.
\end{equation}
Thus we have
\begin{equation}W'_p(s_0, n,\delta_p)=(1+p^{-1})\frac{\ord_p(nN)+1}{2}\ln p.\end{equation}
Otherwise, we have
\begin{equation}W_p(s_0, n,\delta_p)=\bigg(1-p^{-1}\bigg(\frac{D}{p}\bigg)\bigg)\rho_p(nN).\end{equation}

{\bf Case 2 When $p|D_0$.}

In this case, 
 we have
\begin{eqnarray}
W_p(s+s_0, n,\delta_p)
=1+\chi_p(-nN |\Delta|)X^{a+1}.
\end{eqnarray}
Then  we know that
\begin{equation}
\chi_p(-nN|\Delta|)=-1 ~\Leftrightarrow~ W_p(s_0, n,\delta_p)=0.
\end{equation}
Moreover,  \begin{equation}
W'_p(s_0, n,\delta_p)=(\ord_p(n)+1)\ln p.
\end{equation}
Otherwise, we have $$W_p(s_0, n,\delta_p)=2.$$

 {\bf Case 3 When $p| \Delta$. }

Firstly,  we consider the case where  $\delta_p\equiv 0 $. 

It follows that $\chi_{\Delta}(\delta)=0$ in this case.

If $\ord_p(n)=0$, we have
$$
W_p(s, n,\delta_p)=1-p^{-s},
$$
\begin{equation}
W_p(s_0, n,\delta_p)=0,~ ~W_p'(s_0, n,\delta_p)=\ln p.
\end{equation}
Conversely, if  $\ord_p(n) \geq 1$, then
$$
W_p(s, n,\delta_p)=1+(p-1)X+
p\chi_p(-nN \mid \Delta\mid) X^{\ord_p(n)+1}
$$
and
\begin{equation}
W_p(s_0, n,\delta_p)
=p(1+\chi_p(-nN \mid \Delta\mid)).
\end{equation}
If $\chi_p(-nN \mid \Delta\mid)=-1$, then
\begin{equation}
W_p'(s_0, n,\delta_p)=(p\ord_p(n)+1)\ln p.
\end{equation}
Then we have $$W_p(s_0, n,\delta_p)=0\Leftrightarrow \ord_p(n)=0, or~\ord_p(n)\geq 1, ~\chi_p(-nN \mid \Delta\mid)=-1. $$

Secondly, we consider the case where $\delta_p\neq 0  $.

We denote $$n_\delta =n-Q_\Delta(\delta_p) \in \Z_p~and~a=\ord_p(n_\delta) \geq 0.$$
We assume that
 $\delta_p=(0, \frac{i}{p})$  for some $i$,  $1 \leq i \leq p-1$.

It is easy to see that when $\ord_p(n)=0$,
\begin{eqnarray} 
\text{there exists }n_\delta=n-\frac{D|\Delta|}{4Np^2}i^2 \in p Z_p &\Leftrightarrow& \exists i, ~ \frac{4Np^2}{D|\Delta|}n\equiv i^2 (\mod p) \nonumber\\
&\Leftrightarrow& \chi_p(-nN|\Delta|)=1.\nonumber
\end{eqnarray}

If $a\geq 1$,  then   we have $$\ord_p(n)=0~and ~\chi_p(-nN|\Delta|)=1.$$
By the local density formula, we  have  \begin{equation}W_p(s_0, n,\delta_p)=p.\end{equation}

If $\chi_p(-nN|\Delta|)=-1$, then $a=0$.
Thus we have
\begin{equation}\label{equazero}
W_p(s_0, n,\delta_p)=0,~~ W_p'(s_0, n,\delta_p)=\ln p.
\end{equation}

Then we have $$W_p(s_0, n,\delta_p)\neq 0 \Leftrightarrow \ord_p(n)=0,   n_\delta \in p\Z_p,$$
which implies that  $\chi_p(-nN \mid \Delta\mid)=1. $


%

In summary, based on the three cases discussed as above, we have
\begin{equation}
 \chi_p(-nN|\Delta|)=-1\Rightarrow  W_p(s_0, n,\delta_p)=0.
\end{equation}
Thus, we obtain the result.
\end{proof}

We define two subsets of  $ \mathcal N /\mathcal N^\Delta$ as follows
\begin{equation}
S=\big\{\delta\in \mathcal N /\mathcal N^\Delta \mid n-Q_\Delta(\delta_p) \in p\Z_p, ~\text{for any }p \mid \Delta\big\},                                      
\end{equation}
and 
\begin{equation}
S_q=\big\{\delta\in \mathcal N /\mathcal N^\Delta \mid n-Q_\Delta(\delta_p) \in p\Z_p, ~\text{for any }p \mid \frac{\Delta}{q}\big\}.                                       
\end{equation}
Each  $\delta \in \mathcal N /\mathcal N^\Delta$ with $Q_\Delta(\delta) \in \Z$ can be written as 
 $$\delta_i=Rie_1+2Nie_2=i\kzxz{R}{2}{\frac{-R^2-D}{ 2}}{-R} \in \mathcal N, ~0\leq i< |\Delta|.$$
At the place $p|\Delta$, it gives $\delta_p=(0, \frac{2Ni}{|\Delta|})$. 


It follows  that 
 \begin{equation}\label{equitwo}
   \delta_i \in S \Leftrightarrow \frac{4Np^2}{D|\Delta|}n\equiv \bigg(\frac{2Npi}{|\Delta|}\bigg)^2(\text{mod p}), ~\text{for~ any}~ p |\Delta.
 \end{equation}
 We write the two solutions of the following equation as $\pm \mathfrak{i}$,
\begin{equation}\label{equslou}
\frac{4Np^2}{D|\Delta|}n\equiv x^2 (\mod p).
\end{equation} 
 Then we know that $\mathfrak{i}=\frac{2Npi}{|\Delta|}(\mod p)$. The number of $S$ is equal to $|S|=2^{o(\Delta)}$.
%
%


Then we have the following result.
\begin{lemma}\label{lemsum}
$(1)$ We have 
\begin{eqnarray}
\sum_{\delta \in S_q }\chi_{\Delta}(\delta)=0.
\end{eqnarray}
$(2)$
 {\bf Assumption A} The discriminant $\Delta$ has a prime factor $p $ such that $p\equiv 3(\mod 4)$.
 
  Then
 we have
\begin{eqnarray}
\sum_{\delta \in S }\chi_{\Delta}(\delta)=0.
\end{eqnarray}
\end{lemma}
\begin{proof}
$(1)$  There is a bijective map between the set $$\{\delta \in\mathcal N /\mathcal N^\Delta| Q_\Delta(\delta)\in \Z\} $$ and the set $\{\delta_i, 0 \leq i <|\Delta|\}$, where 
 $\delta_i=i\kzxz{R}{2}{\frac{-R^2-D}{ 2}}{-R}$.
Then we have 
\begin{eqnarray}
&&\sum_{\substack{\delta \in \mathcal N /\mathcal N^\Delta\\ Q_\Delta(\delta)\in \Z} }\chi_{\Delta}(\delta_i)
=\sum_{ i=0}^{|\Delta|-1}\chi_{\Delta}(\delta_i)=\sum_{ i=0}^{|\Delta|-1}\bigg(\frac{\Delta}{2Ni}\bigg)=0.
\end{eqnarray}

 According to the Chinese Remainder Theorem, 
there is a bijective map between  $S_q$ and the set of vectors  $\{(\delta_p)_{p| \Delta}\}$, where $\delta_p=\pm (0,\frac{\mathfrak{i}}{p})$ if $p |\frac{\Delta}{q}$, and  $\delta_q=(0, \frac{j}{q})$, with $0\leq j < q$.
Then  we have 
\begin{eqnarray}
&&\sum_{\delta_i \in S_q }\chi_{\Delta}(\delta_i)
=\bigg(\frac{\Delta}{2N}\bigg)\sum_{\delta_i \in S_q  }\bigg(\frac{\Delta}{i}\bigg)\\
&=&\bigg(\frac{\Delta}{2N}\bigg)\prod_{p\mid \frac{\Delta}{q}}\bigg(\bigg(\frac{p^*}{i}\bigg)+\bigg(\frac{p^*}{-i}\bigg)\bigg)\times\bigg(\sum_{j=0}^{q-1}\big(\frac{q^*}{j})\bigg)=0.\nonumber
\end{eqnarray}
Locally, here $2Ni\equiv \frac{|\Delta|}{p}\mathfrak{i}(\mod p)$.

$(2)$
If $p \mid (\Delta, n)$, then we have $S=\{0\}$.

Now we assume that $(\Delta, n)=1$.

There is a bijective map between $S$ and the set of vectors  $\{(\delta_p)_{p| \Delta}\}$,
where  $\delta_p=\pm (0,\frac{\mathfrak{i}}{p})$.

Thus, we have 
\begin{eqnarray}
&&\sum_{\delta_i \in S }\chi_{\Delta}(\delta_i)
=\sum_{\delta_i \in S  }\bigg(\frac{\Delta}{2Ni}\bigg)
=\bigg(\frac{\Delta}{2N}\bigg)\sum_{\delta_i \in S  }\bigg(\frac{\Delta}{i}\bigg)\\
&=&\bigg(\frac{\Delta}{2N}\bigg)\sum_{\delta_i \in S  }\prod_{p\mid \Delta}\bigg(\frac{p^*}{i}\bigg)=\bigg(\frac{\Delta}{2N}\bigg)\prod_{p\mid \Delta}\bigg(\bigg(\frac{p^*}{i}\bigg)+\bigg(\frac{p^*}{-i}\bigg)\bigg).\nonumber
\end{eqnarray}

If one of the factor $p \equiv 3(\mod 4)$, then $$\big(\frac{p^*}{i}\big)+\big(\frac{p^*}{-i}\big)=0.$$
Thus, we obtain the result.
\end{proof}
We let
$$ \Lambda(\chi_D, s)=|D|^{\frac{s}{2}}\pi^{-\frac{s+1}{2}}\Gamma(\frac{s+1}{2})L(\chi_D, s)$$
to be the completed L-function and $\chi_D$ to be  the quadratic Dirichlet character. Notice that 
$\Lambda(\chi_D, 1)=\frac{\sqrt{|D|}}{\pi}L(\chi_D, 1)$.
According to the preceding Lemma \ref{lemlocalwhi},
we can derive the following result.
\begin{proposition}\label{theindepend}
Let the notation be as above. Then
 $k(n, \delta)=0$ unless $|\Diff(n)|=1$. If $k(n, \delta)\neq 0$, then it depends only on $n$ and we denote it as $k(n)$. Assuming  that $\Diff(n)=\{p\}$, then we have the following equations.
\begin{enumerate}
  \item If $p | \Delta$, then
\begin{eqnarray}\Lambda(\chi_D, 1)k(n)
=-2^{o(D_0)+1}\rho(nN)\frac{1}{p}\ln p,\nonumber
\end{eqnarray}
where  $o(D_0)$ denotes the number of prime factors of $D_0$.

  \item If $p | D_0$, then
$$\Lambda(\chi_D, 1)k(n)=- 2^{o(D_0)}\rho(nN)(\ord_p(n)+1)\ln p.$$
  \item If  $p$ is inert in $k$, then
$$\Lambda(\chi_D, 1)k(n)=-2^{o(D_0)}\rho\big(\frac{nN}{p}\big)(\ord_p(nN)+1)\ln p .$$
\end{enumerate}
\end{proposition}
\begin{proof}
According to equations (\ref{equFE}) and (\ref{equfouprod}), we have
\begin{eqnarray}
k(n, \delta)&=&E'_{\mathcal N^\Delta, n}(\tau,0 ;1)q^{-n}\\
&=&q^{-n}W_{n, \infty}(\tau, s_0, \Phi^1_\infty)\prod_{p< \infty}W_{n, p}(\tau, s_0, \varphi_{\delta, p}).\nonumber
\end{eqnarray}
The following archimedean  Whittaker function is studied in \cite[Section 2]{KYEisen}, 
\begin{equation}
 q^{-n}W_{n, \infty}(\tau, s_0, \Phi^1_\infty)=-2\pi i.
\end{equation}
Then we have
\begin{eqnarray}
k(n, \delta)&=&-2\pi i \prod_{p< \infty}W_{n, p}(s_0, \varphi_{\delta, p})=-2\pi i \prod_{p< \infty} \gamma_p |S|^{\frac{1}{2}}_pW_{ p}(s_0, n,  \delta_p)\nonumber\\
&=&-2\pi i \prod_{p< \infty} \gamma_p \prod_{p< \infty} |S|^{\frac{1}{2}}_p\prod_{p< \infty}
W_{ p}(s_0, n, \delta_p)\nonumber\\
&=& - \frac{2\pi}{\sqrt{|D\Delta^2|}}\prod_{p< \infty}
W_{ p}(s_0, n,  \delta_p),
\end{eqnarray}
since $$\prod_{p\leq \infty} \gamma_p=1~and ~\gamma_\infty=i.$$
The local densities $W_{ p}(s_0, n,  \delta_p)$ are given in Lemma \ref{lemlocalwhi}. 

Thus, we obtain the result.
\end{proof}

\subsection{On the constant term}

The Siegel theta function for the lattice $ \mathcal P^\Delta$ is defined by
\begin{equation}
\theta_{\mathcal P^\Delta}(\tau)=\sum_{h \in \mathcal P^{\sharp}/\mathcal P^\Delta}\sum_{\lambda \in  \mathcal P^\Delta_h}e(Q_{\Delta}(\lambda)\tau)e_h \in M_{1/2, \rho_{\mathcal P^\Delta}}.
\end{equation}

 Schofer  investigated the CM values of Borcherds lifts  of weakly modular forms in \cite{Scho}.
It has been extended to harmonic weak Maass forms in \cite{BY}  and \cite{BEY}.  Our focus will be on the twisted cases.

For any  $f \in H_{1/2, \tilde{\rho}_{L}}$, by Proposition \ref{proadjoint}, we have
\begin{eqnarray}\label{equjiaocha}
\langle f(\tau), \Theta_{\Delta, r}(\tau, z_{U}^\pm, h) \rangle
&=&\langle f(\tau), \psi(\Theta_{L^\Delta}(\tau, z_{U}^\pm, h) )\rangle \\
&=&\langle \phi(f(\tau)), \Theta_{L^\Delta}(\tau, z_{U}^\pm, h) \rangle.\nonumber
\end{eqnarray}

According to Lemma \ref{lemresext},  it follows that
\begin{equation}
\langle \phi(f(\tau)), \Theta_{L^\Delta}(\tau, z_{U}^\pm, h) \rangle
=\langle \phi(f(\tau))_{\mathcal P^\Delta\bigoplus\mathcal N^\Delta}, \theta_{\mathcal P^\Delta\bigoplus\mathcal N^\Delta}(\tau, z_{U}^\pm, h) \rangle.
\end{equation}
Now we can assume that $L^\Delta=\mathcal P^\Delta\bigoplus\mathcal N^\Delta$.
This lead us to  the splitting equation
\begin{equation}
\Theta_{L^\Delta}(\tau, z_{U}^\pm, h)= \theta_{\mathcal P^\Delta}(\tau)\otimes\theta_{\mathcal N^\Delta}(\tau, z_{U}^\pm, h).
\end{equation}
Combining it with equation (\ref{equjiaocha}), we obtain
\begin{equation}\langle f(\tau), \Theta_{\Delta, r}(\tau, z_{U}^\pm, h) \rangle=\langle \phi(f(\tau)), \theta_{\mathcal P^\Delta}(\tau)\otimes\theta_{\mathcal N^\Delta}(\tau, z_{U}^\pm, h) \rangle.\end{equation}

By the same argument as given in \cite{BY}  and \cite{BEY}, we have the following result.
\begin{lemma}\label{lemBEY}

If $f \in H_{1/2-2j, \bar{\tilde{\rho}}_{L}}$, then
\begin{eqnarray}
&&\Phi^j_{\Delta,r}(z_{U}^\pm, h, f)\nonumber\\
&=&\lim_{T\rightarrow \infty}\bigg [\frac{1}{(4\pi)^j}\int_{\mathcal F_T}\langle \phi(R^j_{k-2j}f(\tau)),~\theta_{\mathcal P^\Delta}(\tau)\otimes
\theta_{\mathcal N^\Delta}(\tau, z_{U}^\pm, h)  \rangle d\mu(\tau)-A_0\log(T)\bigg],\nonumber
\end{eqnarray}
where
\begin{equation}
A_0=(-1)^jCT\bigg(\big\langle (\phi(f)^+)^{(j)}, \theta_{\mathcal P^\Delta}\otimes e_{0+\mathcal N^\Delta} \big\rangle\bigg),
\end{equation}
and  $g^{(j)}=\frac{1}{(2\pi i )^j}\frac{\partial^j}{\partial \tau^j}g$.
\end{lemma}

According to the above Lemma and  the Siegel-weil formula,  we have the following result.
\begin{proposition}\label{procmv}
\begin{eqnarray}
&&\Phi_{\Delta, r}^j(Z(U), f)
=\frac{\deg(Z(U)}{2}\nonumber\\
&&\times\lim_{T\rightarrow \infty}\bigg [\frac{1}{(4\pi)^j}\int_{\mathcal F_T}\big\langle \phi(R^j_{\frac{1}{2}-2j}f(\tau)), \theta_{\mathcal P^\Delta}(\tau)\otimes
E_{\mathcal N^\Delta}(\tau, 0; -1)  \big\rangle d\mu(\tau)- 2A_0\log(T)\bigg].\nonumber
\end{eqnarray}
\end{proposition}
\begin{proof}
When $\Delta=1$, it has been  proved in \cite{BY} and \cite{BEY}.
When $\Delta \neq 1$, we  can prove it by the same method.
\end{proof}



Moreover, we have the following result.
\begin{lemma}\label{lemzero}
When $n \in \Z \geq 0$, we have
\begin{equation}\sum_{\delta \in \mathcal N / \mathcal N^\Delta  }
\chi_{\Delta}(\delta)k(n, \delta)=0.
\end{equation}
\end{lemma}
\begin{proof}
Because   $k(0, \delta)=0$ for  any $\delta \neq 0$, and $\chi_{\Delta}(\delta)=0$ if $\delta=0$, so $$\chi_{\Delta}(\delta)k(0, \delta)=0$$ for any  $\delta \in \mathcal N / \mathcal N^\Delta$.

Now we assume that $n \in \N$ and $\Diff(n)=\{p\}$.

Then   we have 
\begin{equation}
\sum_{\delta \in \mathcal N / \mathcal N^\Delta  }
\chi_{\Delta}(\delta)k(n, \delta)=
\left\{
                                    \begin{array}{ll}
                                      \sum_{\delta \in S  }\chi_{\Delta}(\delta)k(n, \delta), & \hbox{$p \nmid \Delta$;} \\
                                      \sum_{\delta \in S_p  }\chi_{\Delta}(\delta)k(n, \delta), & \hbox{$p \mid \Delta$.}
                                    
\end{array}\right.
\end{equation}
According to Proposition \ref{theindepend}, we know that $k(n, \delta)$ is a constant $k(n)$ for $\delta \in S$ or $S_p$.

By Lemma \ref{lemsum}, we have
\begin{eqnarray}
\sum_{\delta \in \mathcal N / \mathcal N^\Delta  }
\chi_{\Delta}(\delta)k(n, \delta)=0.
\end{eqnarray}
Thus we obtain the result.
\end{proof}
We write
\begin{equation}
\theta_{\mathcal P^\Delta}(\tau)=\sum_{h \in \mathcal P^{\sharp}/\mathcal P^\Delta}r(m, h)q^me_h.
\end{equation}
For any $f \in H_{1/2-2j, \bar{\tilde{\rho}}_{L}}$,
we denote
\begin{equation}
M= \max\{n ~| ~ c^+(-n, \mu)\neq 0, n>0\},
\end{equation}
where $c^+(n, \mu)$ are Fourier coefficients of holomorphic part   $f^+$.
Then we have the following result.
\begin{proposition}\label{lemconzero}
We assume that $\Delta$ satisfies the  Assumption A.
For any $f \in H_{1/2-2j, \bar{\tilde{\rho}}_{L}}$, if $m_0 >M$, then
\begin{equation}
CT\big(\langle \phi(f)^+, [\theta_{\mathcal P^\Delta}, \mathcal{E}_{\mathcal N^\Delta}]_j\rangle\big)=0.
\end{equation}
Moreover, when $f=F_{n, \mu}$,
we don't need the Assumption A.
\end{proposition}

\begin{proof}

To simplify the proof, we assume that $j=0$ and omit this index. The argument can be generalized to the general cases.

The constant term is equal to
\begin{eqnarray}\label{equcon}
&&CT\big(\langle \phi(f)^+, \theta_{\mathcal P^\Delta}(\tau)\otimes
\mathcal E_{\mathcal N^\Delta} \rangle\big)\\
&=&2\sum_{\substack{h+\delta \in L^{\sharp}/L^\Delta  \\ h+\delta\equiv r\mu(L)\\ n>0 }}\chi_{\Delta}(h+\delta)c^+(-n, \mu)\sum_{ n\geq m\geq 0}r(n-m, h)k(m, \delta).\nonumber
\end{eqnarray}

Notice that $\delta \in (\mathcal N^\Delta)^\sharp\bigcap L^\sharp=\mathcal N$ and
$h \in \mathcal P^{\sharp}\cap L^{\sharp}=\Z w$.

Taking $m_0$  sufficiently large, specifically
\begin{equation}
Q_\Delta(w)=m_0>M,
\end{equation}
then  the non-vanishing term  $r(n-m, h)$ simplifies to $r(0, 0)$.
It implies  that $n=m$, $h=0$. Since $n\in \frac{1}{4N}\Z$ and $m \in \frac{1}{\Delta}\Z$,  we have $n=m \in \Z$.

%

According to Lemma \ref{lemzero}, we have
  \begin{eqnarray}
&&CT\big(\langle \phi(f)^+, \theta_{\mathcal P^\Delta}(\tau)\otimes
\mathcal E_{\mathcal N^\Delta} \rangle\big)\\
&=&2\sum_{n>0} c^+(-n, 0) \sum_{\delta \in \mathcal N / \mathcal N^\Delta  }
\chi_{\Delta}(\delta)k(n, \delta)=0.\nonumber
\end{eqnarray}

Especially, we fix   $f=F_{n, \mu}$. 

If $\ord_p(n)>0$ for some $p \mid \Delta$ and $\Diff(n)\neq \{p\}$, then all $k(n, \delta)=0$ for $\delta \neq 0$;
if $\Diff(n)= \{p\}$,
according to Lemma \ref{lemzero},
we have 
\begin{eqnarray}
&&CT\big(\langle \phi(f)^+, \theta_{\mathcal P^\Delta}(\tau)\otimes
\mathcal E_{\mathcal N^\Delta} \rangle\big)\\
&=&2c^+(-n, 0) k(n)\sum_{\delta \in S_p  }
\chi_{\Delta}(\delta)=0.\nonumber
\end{eqnarray}

If $\ord_p(n)=0$ for some $p \mid \Delta$, according to Lemma \ref{lemexist},
 we can choose infinitely many fundamental discriminants $D_0$ such that 
$$ \chi_p(-nN|\Delta|)=(\Delta D_0, -nN|\Delta|)_p=-1.$$
Then we obtain the result by the same reason.
Thus we complete the proof without the Assumption A.
\end{proof}

\begin{remark}
  \item $(1)$ If $\Delta <0$, there exists a factor $p\equiv 3(\mod 4)$; if $\Delta>0$ and $\kappa$ is even, then
  $c^+(-n, 0)=0$. Thus we  need  the Assumption A only when $\Delta>0$ and $\kappa$ is odd.
  \item  $(2)$  For any $f \in H_{-1/2-2j, \bar{\tilde{\rho}}_{L}}$, the equality holds if $\theta_{\mathcal P^\Delta}(\tau)$ is replaced by the Millsion theta function as follows
\begin{equation}
\tilde{\theta}_{\mathcal P^\Delta}(\tau)=\sum_{\delta \in \mathcal P^\sharp/\mathcal P^\Delta}\sum_{\lambda \in \mathcal P^\Delta_\delta}
p_{z_u^+}(\lambda)e(Q_\Delta(\lambda))e_\delta \in M_{3/2, \rho_{\mathcal P^\Delta}}.
\end{equation}

\end{remark}

\part{On the derivative of L-function}

\section{Twisted L-function and Shimura correspondence}\label{Sec7}

\subsection{ Twisted L-function}

For  any cusp form $g(\tau) \in S_{\frac{3}{2}+2j, \rho_L}$,
 the L-function is defined in \cite{BY} and \cite{BEY}.
For any cusp form $g \in S_{\frac{1}{2}+\kappa, \tilde{\rho}_L}$, we define the twisted L-function by, 
\begin{equation}
L(g, U, \chi_\Delta, s)=\left\{
                          \begin{array}{ll}
                            \big\langle [\theta_{\mathcal P^\Delta}, E_{ \mathcal N^\Delta}(\tau, s; 1)]_j, \phi(g) \big\rangle_{Pet},
                            & \hbox{$\kappa=2j+1$;} \\
                            \big\langle [\tilde{\theta}_{\mathcal P^\Delta}, E_{\mathcal N^\Delta}(\tau, s; 1)]_j, \phi(g) \big\rangle_{Pet}, & \hbox{$\kappa=2j+2$.}
                          \end{array}
                        \right.
\end{equation}
We assume that $U=\Q(\sqrt{D})$ and $D=-4N|\Delta| m_0$ is a fundamental discriminant.
\begin{lemma}\label{lemdirichlet}
For any  $g=\sum_{m>0, \mu} b(m, \mu)q^me_\mu \in S_{\frac{1}{2}+\kappa, \tilde{\rho}_L}$, we have
\begin{eqnarray}
&&L(g, U, \chi_\Delta, s)\nonumber\\
&=&\frac{(-1)^{\kappa-1}\Gamma(\kappa-\frac{1}{2})\Gamma(s+\kappa)}{2^{2s+2\kappa-2}\pi^{\frac{s-1}{2}+\kappa} m_0^{\frac{s+\kappa}{2}}\Gamma(\frac{s}{2}+1)(\kappa-1)!}\sum_{k \geq 1 }\bigg(\frac{\Delta}{k}\bigg)\overline{b(k^2 m_0, k\mu_0)} k^{-s-\kappa}.\nonumber
\end{eqnarray}
Here $\Gamma(\kappa-\frac{1}{2})=(\kappa-\frac{3}{2})!\sqrt{\pi}$.

\end{lemma}
\begin{proof}
When $\kappa=2j+1$,
by the same argument as in \cite[Lemma 5.3]{BEY}, we have
\begin{equation}
L(g, U, \chi_\Delta, s)=\tbinom{2j-\frac{1}{2}}{j}
\frac{\Gamma(\frac{s}{2}+1+j)}{(4\pi)^{j}\Gamma(\frac{s}{2}+1)}\big\langle \theta_{ P^\Delta} \otimes E_{ \mathcal N^\Delta}(\tau, s; 1+2j), \phi(g)\big\rangle_{Pet}.
\end{equation}
By the unfolding method, we obtain the result.

Similarly, for $\kappa=2j+2$, we find
\begin{equation}\label{equlfun}
L(g, U, \chi_\Delta, s)=\tbinom{2j+\frac{1}{2}}{j}
\frac{\Gamma(\frac{s}{2}+1+j)}{(4\pi)^{j}\Gamma(\frac{s}{2}+1)}\big\langle \tilde{\theta}_{ P^\Delta} \otimes E_{ \mathcal N^\Delta}(\tau, s; 1+2j), \phi(g)\big\rangle_{Pet}.
\end{equation}
Recall that
$$\lambda(z_U^+)=\frac{1}{\sqrt{2m_0}}w,$$
and
$$p_{z_U^+}(w)=-\frac{1}{\sqrt{2}}(w, \lambda(z_U^+))_\Delta=-\sqrt{m_0}.$$
Using the unfolding method, we get
\begin{eqnarray}
&&\big\langle \tilde{\theta}_{\mathcal P^\Delta} \otimes E_{\mathcal N^ \Delta }(\tau, s; 1+2j), \phi(g)\big\rangle_{Pet}\\
&=&-\frac{\Gamma(\frac{s+3}{2}+j)\sqrt{m_0}}{(\pi m_0)^{\frac{s+3}{2}+j}2^{s+2j+2}}\sum_{k \geq 1 }\bigg(\frac{\Delta}{k}\bigg)\overline{b(k^2 m_0, k\mu_0)} k^{-s-2j-2}.\nonumber
\end{eqnarray}
Combining it with equation (\ref{equlfun}), we obtain that
\begin{eqnarray}
&&L(g, U, \chi_\Delta, s)\nonumber\\
&=&-
\frac{\tbinom{2j+\frac{1}{2}}{j}\Gamma(\frac{s}{2}+1+j)\Gamma(\frac{s+3}{2}+j)}{2^{s+4j+2}\pi^{\frac{s+3}{2}+2j} m_0^{\frac{s}{2}+j+1}\Gamma(\frac{s}{2}+1)}\sum_{k \geq 1 }\bigg(\frac{\Delta}{k}\bigg)\overline{b(k^2 m_0, k\mu_0)} k^{-s-2j-2}\nonumber\\
&=&-
\frac{\Gamma(\kappa-\frac{1}{2})\Gamma(s+\kappa)}{2^{2s+2\kappa-2}\pi^{\frac{s-1}{2}+\kappa} m_0^{\frac{s+\kappa}{2}}\Gamma(\frac{s}{2}+1)(\kappa-1)!}\sum_{k \geq 1 }\bigg(\frac{\Delta}{k}\bigg)\overline{b(k^2 m_0, k\mu_0)} k^{-s-\kappa}.\nonumber
\end{eqnarray}
Thus, we obtain the result.
\end{proof}
The CM value $\Phi^j_{\Delta,r}(Z(U), f)$  provides the derivative of L-functions as follows.

\begin{theorem}\label{theotwistedvalue}
For any  $f \in H_{\frac{3}{2}-\kappa, \bar{\tilde{\rho}}_L}$,  if $m_0>M$ and $(\Delta, 2N)=1$, then
\begin{equation}
\Phi^j_{\Delta,r}(Z(U), f)
=-\deg(Z(U))L'(\xi_{\frac{3}{2}-\kappa}f, U, \chi_\Delta, 0).
\end{equation}
\end{theorem}
\begin{proof}
Now we assume that $\kappa=2j+1$. When $\kappa=2j+2$, it can be proved by the same method.

It is known that the rasing operator is a nearly self-adjoint operator.

For any $f \in H_{1/2-2j, \bar{\tilde{\rho}}_{L}}$, by Proposition \ref{procmv}, we have
\begin{eqnarray}
&&\Phi_{\Delta, r}^j(Z(U), f)\nonumber\\
&=&\frac{\deg(Z(U))}{2}\lim_{T\rightarrow \infty}\bigg[\frac{1}{(4\pi)^j}\int_{\mathcal F_T}\langle \phi(R^j_{\frac{1}{2}-2j}f(\tau)), \theta_{\mathcal P^\Delta}(\tau)\otimes
E_{\mathcal N^\Delta}(\tau, 0; -1)  \rangle d\mu(\tau)\nonumber\\
&&-2A_0\log(T)\bigg ]\nonumber\\
&=&\frac{\deg(Z(U))}{2}\lim_{T\rightarrow \infty}\big[ I_T(f)-2A_0\log(T)\big ],\nonumber
\end{eqnarray}
where the integral
\begin{equation}
I_T(f)=\frac{1}{(-4\pi)^j}\int_{\mathcal F_T}\langle \phi(f(\tau)), R^j_{-\frac{1}{2}}(\theta_{\mathcal P^\Delta}(\tau)\otimes E_{\mathcal N^\Delta}(\tau, 0; -1))  \rangle d\mu(\tau).\nonumber
\end{equation}
By \cite[Theorem 5.4]{BEY}, we have
\begin{equation}
R^j_{-\frac{1}{2}}(\theta_{\mathcal P^\Delta}(\tau)\otimes
E_{\mathcal N^\Delta}(\tau, 0; -1)) =2(-4\pi)^jL[\theta_{\mathcal P^\Delta}, E'_{\mathcal N^\Delta}(\tau, 0; 1)]_j.\nonumber
\end{equation}
Then we obtain
\begin{eqnarray}\label{equsum}
I_T(f)
&=&\int_{\mathcal F_T}\langle \phi(f(\tau)), L[\theta_{\mathcal P^\Delta}, E'_{\mathcal N^\Delta}(\tau, 0; 1)]_j\rangle d\mu(\tau)\nonumber\\
&=&2\int_{\mathcal F_T}L\langle \phi(f(\tau)), [\theta_{\mathcal P^\Delta}, E'_{\mathcal N^\Delta}(\tau, 0; 1)]_j\rangle d\mu(\tau)\nonumber\\
&-&2\int_{\mathcal F_T}\langle L\phi(f(\tau)), [\theta_{\mathcal P^\Delta}, E'_{\mathcal N^\Delta}(\tau, 0; 1)]_j\rangle d\mu(\tau).
\end{eqnarray}
By the  Stokes’ theorem
and  Proposition \ref{lemconzero}, we obtain
\begin{eqnarray}
&&\lim_{T\rightarrow \infty}\bigg[ \int_{\mathcal F_T}L\langle \phi(f(\tau)), [\theta_{\mathcal P^\Delta}, E'_{\mathcal N^\Delta}(\tau, 0; 1)]_j\rangle d\mu(\tau)-A_0\log(T)\bigg ]\nonumber\\
&=&\lim_{T\rightarrow \infty}\int_{iT}^{iT+1}\langle \phi(f), [\theta_{\mathcal P^\Delta}, \mathcal E_{\mathcal N^\Delta}(~, 0; 1)]_j\rangle d\tau\nonumber\\
&=&CT(\langle \phi(f), [\theta_{\mathcal P^\Delta}, \mathcal{E}_{\mathcal N^\Delta}\rangle]_j)=0.\nonumber
\end{eqnarray}
The  second summand  in (\ref{equsum}) can be written as the Petersson product. Then we have
\begin{eqnarray}
\Phi_{\Delta, r}^j(Z(U), f)&=&-\deg(Z(U))
\big\langle[\theta_{\mathcal P^\Delta}, E'_{\mathcal N^\Delta}(~, 0; 1)]_j, ~\xi_{\frac{1}{2}-2j}\phi(f)\big\rangle_{Pet}\nonumber\\
&=&-\deg(Z(U))L'(\xi_{\frac{1}{2}-2j}f, U, \chi_\Delta, 0).\nonumber
\end{eqnarray}
Thus, we complete the proof.
\end{proof}
We have
\begin{eqnarray}
\Phi_{\Delta,r}^j(Z(U), f)&=&\frac{2}{w_k}\sum_{(z,h)\in supp(Z(U))}\Phi_{\Delta, r}^j(z, h, f)\nonumber\\
&=&\frac{2}{w_k}\sum_{h \in k^\times\setminus \A^\times_{k, f}/ \hat {\mathcal O}_k^\times}\chi([h])\Phi_{\Delta,r}^j(\gamma^{-1}z_U^\pm, f),
\end{eqnarray}

\subsection{Shimura correspondence }

Let $m_0 \in \Q$ and $\mu_0 \in L^\sharp/L$ such that $m_0 \in Q(\mu_0)+\Z$. Assume that
$D_0=-4N\sgn(\Delta)m_0$ is a fundamental discriminant.

The Shimura lift
$Sh_{m_0,\mu_{0}} :S_{\frac{1}{2}+\kappa, \tilde{\rho}_L}\rightarrow S_{2\kappa}(N) $(see \cite{GKZ}) is given by
\begin{equation}
g=\sum_{\mu }\sum_{m>0} b(m, \mu)q^me_\mu \mapsto Sh_{m_0,\mu_{0}}(g)
=\sum_{n=1}^\infty\sum_{d \mid n}d^{\kappa-1}\big(\frac{D_0}{d}\big)b(\frac{m_0n^2}{d^2}, \frac{n}{d}\mu_0)q^n,\nonumber
\end{equation}
where $\kappa\in \Z\geq 1$.

The Fourier coefficients of $Sh_{m_0,\mu_{0}}(g)$ is denoted by $B(n)$.
The twisted L-functions is defined by
\begin{equation}
L(Sh_{m_0,\mu_{0}}(g), \chi_{\Delta}, s )=\sum_{n>0}\bigg(\frac{\Delta}{n}\bigg)B(n)n^{-s}.
\end{equation}
Then we have the following result.
\begin{lemma}\label{lemtder}
Assume that $D=D_0\Delta$ is a fundamental discriminant.  For any $g\in S_{\frac{1}{2}+\kappa, \tilde{\rho}_L}$ with real coefficients,
\begin{eqnarray}
L(g, U, \chi_\Delta, s)= (-1)^{\kappa-1}\frac{\Gamma(s+\kappa)\Gamma(\kappa-\frac{1}{2})L(Sh_{m_0,\mu_{0}}(g), \chi_{\Delta}, s+\kappa)}{2^{2s+2\kappa-2}\pi^{\frac{s-1}{2}+\kappa}m_0^{\frac{s+\kappa}{2}}\Gamma(\kappa)\Gamma(\frac{s}{2}+1)L(\chi_D, s+1)}.\nonumber
\end{eqnarray}
Moreover,
\begin{eqnarray}
L'(g, U, \chi_\Delta, 0)=(-1)^{\kappa-1}\frac{2^{4-2\kappa}\Gamma(\kappa-\frac{1}{2})\sqrt{N |\Delta|}}
{ m_0^{\frac{\kappa-1}{2}}\deg(Z(U))\pi^{\frac{1}{2}+\kappa}}
L'(Sh_{m_0,\mu_{0}}(g), \chi_{\Delta}, \kappa).\nonumber
\end{eqnarray}
\end{lemma}
\begin{proof}
It is easy to see that
\begin{eqnarray}
L(Sh_{m_0,\mu_{0}}(g), \chi_{\Delta},s )&=&\sum_{n>0}\sum_{d \mid n}d^{\kappa -1}\big(\frac{\Delta}{n}\big)\big(\frac{D_0}{d}\big)b(\frac{m_0n^2}{d^2}, \frac{n}{d}\mu_0)n^{-s}\nonumber\\
&=&L(\chi_{D},s-\kappa+1)\sum_{n>0}\big(\frac{\Delta}{n}\big)b(m_0n^2, n\mu_0)n^{-s}.
\end{eqnarray}
By Lemma \ref{lemdirichlet},
 we obtain the result.
\end{proof}

\subsection{ L-function of newforms}
Let $S^\varepsilon_{2\kappa}(N)$ denote the
space of cusp forms $G$ of weight $2\kappa$ for $\Gamma_0(N)$, with $$G |W_N=\varepsilon (-1)^{\kappa} G$$  under the Fricke involution $W_N$, where $\varepsilon=\pm 1$.

Let $\mathcal{A} \in \Cl_K$ be a given ideal class and $r_{\mathcal{A}}(n)$ be the number of integral ideals of norm $n$
in this class. The Dirichlet series is defined in \cite[Section IV]{GZ},
\begin{equation}
L_{\mathcal{A} }(G, s)=\sum_{(n, N)=1}\big(\frac{D}{n}\big)n^{-2s+2\kappa-1}\sum_{n=1}^{\infty}a(n)r_{\mathcal{A}}(n)n^{-s}.
\end{equation}
Let $\chi$ be a  character of $\Cl_K$, and the associated $L$-function is defined by
\begin{equation}
L_K(G, \chi, s)=\sum_{\mathcal{A} \in \Cl_K}\chi(\mathcal{A} )L_{\mathcal{A} }(G, s),
\end{equation}
which  has an analytic continuation and satisfies the function equation
\begin{equation}\label{equminus}
L^*_K(G, \chi, s)=\chi_\Delta(-N)\varepsilon L^*_K(G, \chi, 2\kappa-s),
\end{equation}
where $L^*_K(G, \chi, s)=(2\pi)^{-2s}N^s |D|^s\Gamma(s)^2L_K(G, \chi, s)$.

When $\Delta>0$, the Shimura lift $G$ belongs to the space  $ S_{2\kappa}^{new,-}(N)$.  The sign is given by $$\chi_\Delta(-N)\varepsilon=-\sgn(\Delta)=-1.$$
When $\Delta <0$, the Shimura lift $G$ belongs to the space
 $ S_{2\kappa}^{new,+}(N)$. The sign is given by  $$\chi_\Delta(-N)\varepsilon=\sgn(\Delta)=-1.$$
It implies  that  $L_K(G, \chi, s)$ vanishes at  $s=\kappa$.

Assume that $\chi$ is a genus character associated to the decomposition  $D=D_0\Delta$, then
\begin{equation}\label{equ5}
L_K(G, \chi,s)=L(G, \chi_{D_0}, s)L(G, \chi_{\Delta}, s),
\end{equation}
where $\chi_{\Delta}=\big(\frac{\Delta}{}\big)$ and $\chi_{D_0}=\big(\frac{D_0}{}\big)$.

According to \cite{SZ}, \cite{GKZ}, and \cite{Sk2}, the space  $S^{new}_{1/2+\kappa, \rho_L}$
is isomorphic to the $S_{2\kappa}^{new, -}(N)$ as a module over the Hecke algebra, which  is given by the Shimura correspondence.  Similarly, $S^{new}_{1/2+\kappa, \bar\rho_L}$ is isomorphic to the space $S_{2\kappa}^{new, +}(N)$.

For any fundamental discriminant $D_0$, the sign of $L(G, \chi_{D_0}, s)$  is given by $\chi_\Delta(-N)\varepsilon =1$.
Then we have the following result
\begin{lemma}\label{lemderiv}
Let $G \in S_{2\kappa}^{new}(N)$ be the Shimura lift of  $g \in S^{new}_{1/2+\kappa, \tilde{\rho}_L}$. Then there exists a
fundamental discriminant $D_0=-4Nm\sgn(\Delta)$, such that
\begin{equation}\label{equ6}
L'_K(G, \chi,\kappa)=L(G, \chi_{D_0}, \kappa)L'(G, \chi_{\Delta}, \kappa), ~L(G, \chi_{D_0}, \kappa)\neq 0.
\end{equation}
\end{lemma}
%
%

For a normalized newform  $G \in  S_{2\kappa}^{new}(N)$, we denote by $F_G$ the total real number field generated by the eigenvalues of $G$. There exists a newform  $g \in S^{new}_{\frac{1}{2}+\kappa, \tilde{\rho}_L}$ that corresponds to $G$ under the Shimura
correspondence. We normalize $g$ such that all its coefficients are contained in $F_G$.
Therefore,
\begin{equation}\label{equshimura}
L(Sh_{m_0,\mu_{0}}(g), \chi_{\Delta},s )=b(m_0, \mu_0)L(G, \chi_{\Delta},s).
\end{equation}

The classical Shimura lift from the weight $\frac{3}2$ cusp forms  to weight $2$ cusp forms is important in several areas, including Tunnell’s congruent number problem. 
Recently, Qin provided a new method to study this problem in \cite{Qin}.


\begin{lemma}\cite[Lemma 7.3]{BO}
There is a
$f \in H_{\frac{3}{2}-\kappa, \bar{\tilde{\rho}}_L}(F_G)$such that
$$\xi_{\frac{3}{2}-\kappa}(f)=\parallel g\parallel^{-2} g .$$
Here the coefficients of the principal part $P_f$ belong to $F_G$ and $\parallel g\parallel$ denotes the Petersson norm.
\end{lemma}

%

Thus we have the following result.
\begin{proposition}\label{proderivative}
Let the notation be as above. Then
\begin{eqnarray}
\Phi^j_{\Delta,r}(Z(U), f)=(-1)^{\kappa}
\frac{2^{4-2\kappa}\Gamma(\kappa-\frac{1}{2})\sqrt{N |\Delta|}}
{ m_0^\frac{\kappa-1}{2}\pi^{\frac{1}{2}+\kappa}\parallel g \parallel^2}b(m_0, \mu_0)L'(G, \chi_{\Delta},\kappa).\nonumber
\end{eqnarray}
\end{proposition}
\begin{proof}
By Theorem \ref{theotwistedvalue} and Lemma \ref{lemtder},
\begin{eqnarray}
&&\Phi^j_{\Delta,r}(Z(U), f)=-\deg(Z(U))L'(\xi_{\frac{3}{2}-\kappa}f, U, \chi_\Delta, 0)\nonumber\\
&=&(-1)^{\kappa}
\frac{2^{4-2\kappa}\Gamma(\kappa-\frac{1}{2})\sqrt{N |\Delta|}}
{ m_0^\frac{\kappa-1}{2}\pi^{\frac{1}{2}+\kappa}}L'(Sh_{m_0,\mu_{0}}(\xi_{\frac{3}{2}-\kappa}f), \chi_{\Delta}, \kappa).\nonumber
\end{eqnarray}
Combining it with equation (\ref{equshimura}), we obtain the result.
\end{proof}

\section{Intersection on the Heegner cycles}\label{Sec8}

In this section, we will prove Theorem \ref{maintheo}.
\subsection{Heegner cycles}

We recall certain details about Kuda-Sato varieties and their CM cycles as discussed in \cite{Zhang}. Let $\kappa$ be a positive integer, and
let $D$ be a discriminant.  Consider $E$ as an elliptic curve with complex multiplication by $\sqrt{D}$, and let $Z(E)$ be the divisor class on $E\times E$ defined by $\Gamma-(E\times\{0\})+D(\{0\times E\})$. Here $\Gamma$ refers to  the graph of multiplication by $\sqrt{D}$. Then $Z(E)^{\kappa-1}$ generates  a cycle in $E^{2\kappa-1}$ with codimension $\kappa-1$.
Now we denote the following cycle  $$c\sum_{\sigma \in P_{2\kappa-1}}sgn(\sigma)\sigma^*(Z(E)^{\kappa-1})$$by $S_\kappa(E)$, where $ P_{2\kappa-1}$ is the symmetric group of $2\kappa-2$ letters which acts on
$E^{2\kappa-1}$ by permuting the factors, and $c$ is a real number such that
the self-intersection of $S_\kappa(E)$ on each fiber is $(-1)^{\kappa-1}$.

For $N$ a product of two relatively prime integers $\geq 3$.  The Kuga-Sato variety $\mathcal{Y}=\mathcal{Y}_\kappa(N)$ is defined to be a canonical
resolution of the $2\kappa-2$-tuple fiber product of $\mathcal{E}(N)$ over $\mathcal X(N)$, where $\mathcal{E}(N)$ is a regular semistable elliptic curve. If $y$ is a CM point in $\mathcal X(N)$,  the CM cycle $S_\kappa(y)$ over $y$ is defined  to
be $S_\kappa(\mathcal{E}_y)$ in $\mathcal{Y}$.

Let $\pi:\mathcal{X}(N) \rightarrow \mathcal{X}_0(N)$ be the projection map. If  $x$ is a CM point in $\mathcal{X}_0(N)$ and $\pi^*(x)=\frac{w(x)}{2}\sum x_i$, where $w(x)=|Aut(x)|$. Then the cycle over $x$ is defined to be
\begin{equation}
S_\kappa(x)=\frac{1}{\deg(\pi)}\sum S_\kappa(x_i).
\end{equation}
Let $x$ ba a CM point in $X_0(N)$, and $\bar{x}$ be the Zariski closure in $\mathcal{X}_0(N)$. Then
$S_\kappa(\bar x)$
 has zero intersection with any cycle of dimension $\kappa$ in $\mathcal{Y}$ which is supported in the
special fibers \cite{Zhang}. The class of $S_\kappa(\bar x)$ in $H^{2\kappa}(\mathcal{Y}(\C), \C)$ vanishes. There is a Green current $g_\kappa(x)$ on $\mathcal{Y}(\C)$, 
$$\frac{\partial \bar{\partial}}{\pi i}g_\kappa(x)= \delta_{S_\kappa(x)}.$$ In the sense of Gillet and Soul\'e \cite{GS}, the codimension $\kappa$ arithmetic cycle on $\mathcal{Y}$ is defined as $$\hat{S}_\kappa(x)=(S_\kappa(\bar x), g_\kappa(x)).$$
The intersection number is defined by
\begin{equation}
\langle S_\kappa(x), S_\kappa(y)\rangle=(-1)^\kappa \langle \hat{S}_\kappa(x), \hat{S}_\kappa(x)\rangle_{GS}.
\end{equation}
According to \cite{Zhang}, it can decompose into local height pairings,
\begin{equation}\label{definter}
\langle S_\kappa(x), S_\kappa(y)\rangle=\sum_{p\leq \infty}\langle S_\kappa(x), S_\kappa(y)\rangle_{p}.
\end{equation}
More precisely,
\begin{eqnarray}\label{equarichi}
&&\langle S_\kappa(x), S_\kappa(y)\rangle_{\infty}=\frac{1}{2}G_{N, \kappa}(x, y),\\
&&\langle S_\kappa(x),  S_\kappa(y)\rangle_p=(-1)^\kappa( S_\kappa(\bar{x})\cdot S_\kappa(\bar{y}))_p,
\end{eqnarray}
where $G_{N, \kappa}(x, y)$ is the higher Green function.

We generalize the higher Heegner cycles  to twisted case by
\begin{eqnarray}
&&Z_{\Delta, r, \kappa}(m, \mu)=m^{\frac{\kappa-1}{2}}\sum_{x \in Z_{\Delta, r}(m, \mu)}\chi_\Delta(x)S_\kappa(x),\\
&& Z_{\Delta, r, \kappa}(U)=m_0^{\frac{\kappa-1}{2}}\sum_{x \in Z_{\Delta, r}(U)}\chi_\Delta(x)S_\kappa(x),\\
&& Z_{\Delta, r, \kappa}(f)=\sum_{m>0} c^+(-m, \mu)Z_{\Delta, r, \kappa}(m, \mu).
\end{eqnarray}

\subsection{Moduli stack}

Let $D_0=-4Nm_0 $ be a discriminant and the order $\OO_{D_0}=\Z[\frac{D_0+\sqrt D_0}2]$ of discriminant $D_0$. Assume that $D_0\equiv r_0^2 \mod 4N$. Then  $\mathfrak {n}_0=[N, \frac{r_0+\sqrt D_0}2]$ is an ideal of $\OO_{D_0}$ with norm $N$.

  The  moduli stack  $\mathcal Z(m_0, \mu_0)$  over $\Z$ is defined in \cite[Section 7]{BY}, which is a horizontal divisor on $\mathcal X_0(N)$. We denote $\mu_0=\kzxz{\frac{r_0}{2N}}{}{}{-\frac{r_0}{2N}}$.

Now we let $D_0$ to be a fundamental discriminant and $k=\Q(\sqrt{D_0})$.

Let $\mathcal{C}$ be the moduli stack  over $\Z$ representation the moduli problem which assigns to every scheme $S$ over $\Z$ the set  $\mathcal{C}(S)$ of pairs $(\emph{E}, \iota)$, where 
\begin{itemize}
  \item $\emph{E}$ is an CM elliptic curves  over $S$; 
  \item the map $$\iota:\mathcal O_k \hookrightarrow \mathcal O_E:=\End_S(\emph{E})$$
  is  an $\mathcal O_k$ action on $\emph{E}$ such that the main involution on $\mathcal O_E$ gives the complex conjugation on $k$.
\end{itemize}

 For any pair $(\emph{E}, \iota) \in \mathcal C(S)$, we let
\begin{equation}
V(\emph{E}, \iota )=\{x\in \mathcal O_\emph{E}| \iota(\alpha)x=x \iota(\bar{\alpha}),
\alpha \in \mathcal O_k ~and~ \tr x=0\}
\end{equation}
be the space of special endomorphisms with the definite quadratic form $N(x):=\deg(x)$. It is a $\mathcal O_k$
 module.
 
Let $\partial$ denote  the different of $k$ and let $\textbf{a}$ be an ideal.
For any $m>0$ and $\delta \in \partial^{-1}\textbf{a}/\textbf{a}$,
we let $\mathcal Z(m, \textbf{a}, \delta)$  be an algebraic stack, which can be viewed as a cycle in $\mathcal C$ via
$(\emph{E}, \iota, \beta) \mapsto (\emph{E}, \iota) $. Now we will recall more details in \cite{KYPull} and \cite{BY}.

We consider the moduli problem that assigns to each scheme $S$ the set of triples $(\emph{E}, \iota, \beta)$, where the following conditions hold
\begin{itemize}
  \item $(\emph{E}, \iota) \in \mathcal C(S)$;\\
  \item $\beta \in V(\emph{E}, \iota )\partial^{-1}\textbf{a}$ such that
  $$N(\beta)=mN\textbf{a},~~~\delta+\beta \in \mathcal O_\emph{E}\textbf{a}.$$
\end{itemize}
This moduli problem is represented by a algebraic stack $\mathcal Z(m, \textbf{a}, \delta)$ of dimension $0$. The forgetful map is a finite \'etale map from $\mathcal Z(m, \textbf{a}, \delta)$ into $\mathcal C$ \cite[Section 6]{BY}.

The arithmetic degree of a 0-dimensional DM-stack $\mathcal Z$ 
is defined by 
\begin{equation}
  \widehat{\deg}(\mathcal Z)=\sum_{p}\sum_{x\in \mathcal Z(\bar \F_p)}\frac{1}{\sharp Aut(x)}i_p(x)\log p,
\end{equation}
where $i_p(x)$ is the length of strictly Henselian local ring as follows
\begin{equation}
  i_p(x)=\text{Length}(\hat{\mathcal O}_{\mathcal Z, x}).
\end{equation}

For any $x=(\emph{E}, \iota, \beta) \in  \mathcal Z(m, \textbf{a}, \delta)(\bar \F_p)$,  according to  \cite[Section 4]{KYPull}, we know that  $i_p(x)$ depends only on $m$. Then we denote this length by $i_p(m)$.


  For
 any elliptic curve $(\emph{E}, \iota) \in \mathcal{C}(S)$, we let $\emph{E}_{\textbf{n}_0}=\emph{E}/\emph{E}[\textbf{n}_0]$ and
 let $\pi: \emph{E} \rightarrow \emph{E}_{\textbf{n}_0}$  be the natural map.
 We denote $\mathcal O_{\emph{E}, \textbf{n}_0}=\End_S(\pi)$.

Now we let $D_1$ be   another discriminant such that $D_0D_1$ is not a square.
According to \cite[Lemma 7.10]{BY},  the natural isomorphism of stacks is given by
 \begin{equation}
 j: \mathcal{C}\rightarrow \mathcal Z(m_0, \mu_0),~~j((\emph{E}, \iota))=(\pi: \emph{E} \rightarrow \emph{E}_{\textbf{n}_0}, \iota).
 \end{equation}
Then the intersection can be viewed as a fiber product:

$$\xymatrix{
j^*\mathcal Z(m_1, \mu_1)   \ar[r]_{~~\pi_1} \ar[d]^{\pi_2} & \mathcal{C}  \ar[d]           \\
   \mathcal Z(m_1, \mu_1) \ar[r] & \mathcal{X}_0(N)}
  $$
where $ j^*\mathcal Z(m_1, \mu_1)=\mathcal Z(m_1, \mu_1) \times_{\mathcal{X}_0(N)} \mathcal{C} $ consists of triples $(\emph{E}, \iota, \phi)$ and
$\phi:\mathcal O_{D_1}\hookrightarrow \mathcal O_{\emph{E}, \textbf{n}_0}$ such that $\phi(\textbf{n}_1)\emph{E}[\textbf{n}_0]=0$.
Here $\pi_1((\emph{E}, \iota, \phi))=(\emph{E}, \iota)$
and $\pi_2((\emph{E}, \iota, \phi))=(\emph{E}\rightarrow \emph{E}_{\textbf{n}_0}, \phi)$.

\begin{lemma}\cite[Lemma 7.10]{BY}
The isomorphism  
\begin{equation}\label{equinterdecom}
j^*\mathcal Z(m_1, \mu_1)(\bar{\mathbb{F}}_p)
\cong\bigsqcup_{\substack{n\equiv r_0r_1(\mod 2N),\\ n^2\leq D_0D_1}}
\mathcal Z\bigg(\frac{D_0D_1-n^2}{4N|D_0|}, \textbf{n}_0, \frac{n+r_1\sqrt{D_0}}{2\sqrt{D_0}}\bigg)(\bar{\mathbb{F}}_p),
\end{equation}
 is given by $(\emph{E}, \iota, \phi)\rightarrow (\emph{E}, \iota, \beta)$ via
$$\beta=\phi\big(\frac{r_1+\sqrt{D_1}}{2}\big)-\frac{n+r_1\sqrt{D_0}}{2\sqrt{D_0}}.$$

\end{lemma}

\subsection{Finite intersection}
It is known that every Elliptic curve $E$ with CM by $\mathcal O_k$ is isomorphic to $E_\textbf{a}=\C/\textbf{a} $, where 
$\textbf{a}$ is a fractional ideal in $k=\Q(\sqrt{D_0})$. 
There are 
 two $\mathcal O_k$-actions   on $E_\textbf{a}$ which are given as,
$$ \iota(\alpha)z=\alpha z,~~~  \bar{\iota}(\alpha)z=\bar\alpha z.$$
When $S=\Spec(\C)$,
 the bijective map   is given by \cite[Lemma 6.1]{BY},
\begin{equation}
  Z(U)=\{z_U^\pm\}\times \Cl_k\simeq\mathcal C(\C) , (z_U^\pm, [\textbf{a}]) \mapsto (E_{\textbf{a}}, \iota~or~\bar{\iota}).
\end{equation}

 For any pair $(\emph{E}, \iota) \in \mathcal C(\C)$,  we write $E(\C)\simeq \C/\Lambda$. Then  for any idele  $h\in \A^{\times}_{k, f}$, one has the action $hE(\C)\simeq \C/(h)\Lambda$, where $(h)$ denotes the ideal generated by $h$.  
 When $S=\Spec(\bar{\mathbb{F}}_p)$, the action of  $ \A^{\times}_{k, f}$ on $\mathcal C(\bar{\mathbb{F}}_p)$ is defined in \cite[Section 5]{KRYIMRN}.
 


 For any pair $(\emph{E}, \iota) \in \mathcal C(\bar \F_p)$,
when $p$ splits in $k$, $V(\emph{E}, \iota )=\{0\}$;
when $p$ is non-split, the endomorphism ring $\mathcal O_\emph{E}$ is a maximal order of the quaternion algebra $\mathbb{B}$, which is ramified  precisely
at $\infty$ and $p$.  Consequently, $V(\emph{E}, \iota )$ is a positive definite lattice of rank $2$ with $N(x)=-x^2$ representing the reduced norm.

Choosing a prime $p_0 \nmid 2pD_0$ such that
$$inv_q\B=\left\{
                 \begin{array}{ll}
                  (D_0, -pp_0)_q, & \hbox{$p~is~inert~in~k$;} \\
                   (D_0, -p_0)_q, & \hbox{$p~is~ramified$.}
                 \end{array}
               \right.
$$
Here $inv_q\B =-1$ exactly when $q=p$ or $\infty$. Thus implies that $p_0$ splits, 
so $p_0\mathcal O_k=\textbf{p}_0\bar{\textbf{p}}_0$.

Following \cite{KRYIMRN},
we  define 
\begin{equation}\kappa_p=\left\{
                 \begin{array}{ll}
                   pp_0, & \hbox{$p~is~inert$;} \\
                   p_0, & \hbox{$p~is~ramified$.}
                 \end{array}
               \right.
\end{equation}
Thus, we can express $$\B=k\oplus k \delta_{\B},$$ such that $\iota(\alpha)\delta_{\B}=\delta_{\B}\iota(\bar{\alpha})$ and $N(\delta_{\B})=\kappa_p$, for any $\alpha \in k$.
\begin{proposition}\cite[Proposition 5.13]{KRYIMRN}
For any $(\emph{E}, \iota) \in \mathcal C(\bar \F_p)$, there is  a fractional ideal $\textbf{b} $ in $k$ such that
\begin{equation}V(\emph{E}, \iota )\cong\textbf{b} \bar{\textbf{b}}^{-1} \mathbf{p}_0^{-1} \delta_{\B},
\end{equation}
where $\mathbf{p}_0$ is a fixed prime ideal lying above $p_0$.
Moreover, if $h \in \A_{k, f}^\times$,
\begin{equation}V(h(\emph{E}, \iota ))=(h)\overline{(h)}^{-1} V(\emph{E}, \iota ).
\end{equation}
\end{proposition}


For any ideal class $[\textbf{a}] \in \Cl_k$, we define
\begin{equation}
\rho(m, [\textbf{a}])=\sharp \{\textbf{c} \subset \mathcal O_k\mid N(\textbf{c})=m, \textbf{c}\in [\textbf{a}]\}.
\end{equation}
By the similar argument as presented in \cite[Section 5]{KRYIMRN}, we have the following result.
\begin{lemma}\label{lemnumber}
\begin{eqnarray}
&&\sharp\{\beta ~\mid \beta \in V(h(\emph{E}_0, \iota_0)  )\partial^{-1}\textbf{n}_0, N(\beta)=mN\}\\
&&=w_k\cdot
\rho(m|D_0|p_0/\kappa_p, [h^{-1}]^2[\mathbf{p}_0\partial\bar{\textbf{n}}_0^{-1}]),\nonumber
\end{eqnarray}
where $w_k=|\mathcal O_k^\times|$.
\end{lemma}
\begin{proof}

Fixing $(\emph{E}_0, \iota_0)$, such that
\begin{equation}V(\emph{E}_0, \iota_0 )\simeq \mathbf{p}_0^{-1} \delta_{\B}.
\end{equation}
For any $\beta=\alpha \delta_{\B} \in V(h(\emph{E}_0, \iota_0) )\partial^{-1}\textbf{n}_0$, it follows that
$$\alpha \in (h) (\overline{h})^{-1} \mathbf{p}_0^{-1}\partial^{-1}\bar{\textbf{n}}_0.$$ Consequently, the ideal
$\textbf{c}=\alpha(h)^{-1} \bar{(h)}\mathbf{p}_0\partial\bar{\textbf{n}}_0^{-1}$ is integral and is contained within the ideal class
$[h^{-1}]^2[\mathbf{p}_0\partial\bar{\textbf{n}}_0^{-1}] \in \Cl_k$ with the norm  given by $N(\textbf{c})=\frac{m|D_0|p_0}{\kappa_p}$.  

Therefore, the quantity of $\beta$ is determined by the integral ideal within an ideal class characterized by the norm $\frac{m|D_0|p_0}{\kappa_p}$. This leads us to the conclusion.
\end{proof}


According to  \cite[Lemma 7.10]{BY},  the isomorphism of stacks is given by 
\begin{equation}
 j: \mathcal{C}\rightarrow \mathcal Z(|\Delta|m_0, r\mu_0),~~j((\emph{E}, \iota))=(\pi: \emph{E} \rightarrow \emph{E}_{\textbf{n}_0}, \iota).
 \end{equation}
Then we view  $\mathcal Z(m, \textbf{n}_0, \delta)$ as a cycle in $\mathcal Z(|\Delta|m_0, r\mu_0)$  under the forgetful map. Then the   twisted cycle $\mathcal Z_{\Delta, r}(m_0, \mu_0)$ provides the twisted degree of $\mathcal Z(m, \textbf{n}_0, \delta)$.  We define it locally by
\begin{equation}
 \deg_{\Delta}(\mathcal Z(m, \textbf{n}_0, \delta)(\bar \F_p))=\sum_{x \in \mathcal Z(m, \textbf{n}_0, \delta)(\bar \F_p)}\frac{\chi_\Delta(x)i_p(x)}{\sharp Aut(x)}\log p.
  \end{equation}
%

Then we have the following result.

\begin{lemma}\label{lemtdeg}
When $p$ is non-split, we have 
\begin{eqnarray}\label{equinter}
 \deg_{\Delta}(\mathcal Z(m, \textbf{n}_0, \delta)(\bar \F_p))=0.
  \end{eqnarray}
\end{lemma}
\begin{proof}

For any $x=(\emph{E}, \iota, \beta) \in \mathcal Z(m, \textbf{n}_0, \delta)(\bar \F_p)$, 
the length  $i_p(x)$  of the local  ring   depends only on $m$ and we denote it  as $i_p(m)$.
We have 
   \begin{equation}\label{equinter}
 \deg_{\Delta}(\mathcal Z(m, \textbf{n}_0, \delta)(\bar \F_p))
=\frac{i_p(m)}{w_k}\sum_{x\in \mathcal Z(m, \textbf{n}_0, \delta)(\bar \F_p)}\chi_\Delta(x)\log p.
  \end{equation}
For each $(\emph{E}, \iota) \in \mathcal{C}(\bar\F_p)$,
we  write it as  $(\emph{E}, \iota)=h(\emph{E}_0, \iota_{0})  $ or $h(\emph{E}_0, \bar{\iota}_{0}) $.

Then one has
 \begin{eqnarray}\label{equintertwo}
 \sum_{x= (\emph{E}, \iota, \beta)\in \mathcal Z(m, \textbf{n}_0, \delta)(\bar \F_p)}\chi_\Delta(x)
 &=&\sum_{[h]\in \Cl_k} \sum_{ (h(\emph{E}_0, \iota_0), \beta) \in \mathcal Z(m, \textbf{n}_0, \delta)(\bar \F_p)} \chi_\Delta([h])\nonumber\\
  &+&\sum_{[h]\in \Cl_k}\sum_{ (h(\emph{E}_0, \bar{\iota}_0), \beta) \in \mathcal Z(m, \textbf{n}_0, \delta)(\bar \F_p)} \chi_\Delta([h]).
\end{eqnarray}

 By Lemma \ref{lemnumber},   the cardinality
 \begin{equation}\label{equnumberbeta}
\sharp\{\beta ~\mid (h(E_0, \iota_0), \beta)\in \mathcal Z(m, \textbf{n}_0, \delta)(\bar \F_p)\}
=w_k\rho([h^{-1}]^2).\nonumber
\end{equation}

Here we write $\rho([h]^2)=\rho(m|D_0|p_0/\kappa_p, [h]^2[\mathbf{p}_0\partial\bar{\textbf{n}}_0^{-1}])$ for easier.

Then we have
 \begin{equation}\label{equinterthree}
 \sum_{[h]\in \Cl_k} \sum_{ (h(\emph{E}_0, \iota_0), \beta) \in \mathcal Z(m, \textbf{n}_0, \delta)(\bar \F_p)} \chi_\Delta([h])=w_k\sum_{[h]\in \Cl_k}\chi_\Delta([h])\rho([h^{-1}]^2).
 \end{equation}

%

 For any class $[g] \in \Cl_k[2]$ with $[g] ^2=1$,  we know that
\begin{equation}
\rho([hg]^2)=\rho([h]^2).\nonumber
\end{equation}
Consequently, we have
  \begin{eqnarray}
 &&\sum_{[h]\in \Cl_k}\chi_\Delta([h])\rho([h^{-1}]^2)\\
 &=&\sum_{[h] \in \Cl_k/ \Cl_k[2]}\sum_{g \in \Cl_k[2]}\chi_\Delta([hg])\rho([h^{-1}g^{-1}]^2)\nonumber\\
 &=&\sum_{[h] \in \Cl_k/ \Cl_k[2]}\chi_\Delta([h])\rho([h^{-1}]^2)
 \sum_{g \in \Cl_k[2]}\chi_\Delta([g])=0,\nonumber
 \end{eqnarray}
 where the group $\Cl_k[2]$ is non-trivial.

Combining it with equation (\ref{equinterthree}), we obtain that
$$\sum_{[h]\in \Cl_k} \sum_{x=(h(\emph{E}_0, \iota_0), \beta) \in \mathcal Z(m, \textbf{n}_0, \delta)(\bar \F_p)}\frac{\chi_\Delta([h])i_p(x)}{\sharp Aut(x)}\log p=0.$$
We obtain the same result if replace $(\emph{E}_0, \iota_0)$ by $(\emph{E}_0, \bar{\iota}_0)$.
Combining it with equations (\ref{equinter})  and (\ref{equintertwo}), we  obtain the result.    
\end{proof}

In order to study the intersection number,
 it is essential to analyze the fiber product  $$j^*\mathcal  Z(|\Delta|m_1, r\mu_1)\cong \bar{Z}(|\Delta|m_1, r\mu_1) \times \bar{Z}(|\Delta|m_0, r\mu_0).$$
According to the decomposition  (\ref{equinterdecom}), it suffices to consider $\mathcal Z(m, \textbf{n}_0, \delta)$.
We define $D_1=-4N|\Delta|m_1$, $D_0=-4N|\Delta|m_0$. We assume that $D_0$ is a fundamental discriminant and $D_1D_0$ is not a  perfect square. 
\begin{theorem}\label{Lemnonsplit}
  Assume that   $\Delta \neq 1$. Then we have $$ \langle  \mathcal Z_{\Delta, r, \kappa}(m_1, \mu_1), \mathcal Z_{\Delta, r, \kappa}(m_0, \mu_0)\rangle_p=0.$$
\end{theorem}
\begin{proof}
Firstly, we consider the case where $\kappa=1$ and omit this index.

\textbf{Case 1}:When $p$ is split in $k=\Q(\sqrt{D_0})$.

For any pair $(E, \iota) \in \mathcal Z(|\Delta|m_0, r\mu_0)$, the elliptic curve
$E$ is ordinary, as  p splits in $k=\Q(\sqrt{D_0})$. Since $D_1D_0$ is not a square, there is no additional action of $\mathcal O_{D_1}$ on $E$, then
  \begin{equation}
 \langle \mathcal Z_{\Delta, r}(m_1, \mu_1), \mathcal Z_{\Delta, r}(m_0, \mu_0)\rangle_p=0.
 \end{equation}
\textbf{Case 2}: When $p$ is non-split in $k$.

According to the decomposition  (\ref{equinterdecom}), we assume that $$
\bigsqcup_{m,\delta}\mathcal Z(m, \textbf{n}_0, \delta)= j^* \mathcal Z(|\Delta|m_1, r\mu_1),$$
where $m=\frac{D_0D_1-n^2}{4N|D_0|}$ and $\delta =\frac{n+r_1\sqrt{D_0}}{2\sqrt{D_0}}$.


 According to Lemma \ref{lemtdeg}, we have
\begin{equation}\deg_{\Delta}(\mathcal Z(m, \textbf{n}_0, \delta)(\bar \F_p))=0.\end{equation}
 Thus, we have 
\begin{equation}\label{equfinite}
\langle \mathcal Z_{\Delta, r}(m_1, \mu_1), \mathcal Z_{\Delta, r}(m_0, \mu_0)\rangle_p=0.
\end{equation}

Assuming now that $\kappa >1$.

Let $x_i \in  Z_{\Delta, r}(m_i, \mu_i)$ and $\bar{x}_i$ denote the Zariski closure of $x_i$ in $\mathcal X_0(N)$.
According to the findings in \cite[Section 6]{BEY}, we have
\begin{equation}\label{equinterhigher}
\langle S_\kappa(x_1), S_\kappa(x_0)\rangle_p=-P_{\kappa-1}\bigg(\frac{n}{D_0D_1}\bigg)\langle \bar{x}_1, \bar{x}_0\rangle_p,
\end{equation}
 where
 \begin{equation}
 P_\kappa(x)=\frac{1}{2^\kappa \kappa!}\frac{d^\kappa}{dx^\kappa}(x^2-1)^\kappa
 \end{equation}
 is the $\kappa$-th Legendre polynomial.

Notice that $P_{\kappa-1}\big(\frac{n}{D_0D_1}\big)$ depends only on $\delta =\frac{n+r_1\sqrt{D_0}}{2\sqrt{D_0}}$. 

By Lemma \ref{lemtdeg} and equation (\ref{equinterhigher}),  we obtain
 \begin{equation}\label{equintefin}
\langle \mathcal Z_{\Delta, r, \kappa}(m_1, \mu_1), \mathcal Z_{\Delta, r, \kappa}(m_0, \mu_0)\rangle_p=0.
\end{equation}
Thus, we  complete the proof.
\end{proof}

According to the above theorem, we have the following result.
\begin{corollary}\label{lemfininter}
Let the notations be as above.
Then we have the following equation
\begin{equation}
\langle \mathcal Z_{\Delta, r, \kappa}(m_1, \mu_1), \mathcal Z_{\Delta, r, \kappa}(m_0, \mu_0)\rangle_{fin}=0.
\end{equation}
\end{corollary}

\subsection{Main theorem}

\begin{lemma}\label{lemexist}
 Let $S$ be a finite set of primes including all those dividing $N$, and let $S'$ be another disjoint finite set  of primes.
 For any cusp form
$g=\sum_{m, \mu}b(m, \mu)q^ne_\mu \in S^{new}_{1/2+\kappa, \tilde{\rho}_L}$,
there
 exist infinitely many fundamental discriminants $D$ satisfying the following conditions:
\begin{list}{}{}
\item[(1)] $\sgn(\Delta)D<0$,
\item[(2)] $p$ splits in $\Q(\sqrt{D})$ for all $p \in S$, and $p$ is inert for any $p \in S'$.
\item[(3)] $b(m_0, \mu_0)\neq 0$ for $m_0=-\frac{sgn(\Delta)D}{4N}$ and $\mu_0 \in L^\sharp/L$ such that $m_0 \equiv sgn(\Delta)Q(\mu_0)(\mod \Z)$.
\end{list}
\end{lemma}
\begin{proof}
 The proof for the case when $\Delta=1$ has already been established in \cite[Lemma 7.5]{BY}.

 By  the non-vanishing theorem for $L$-functions in \cite{BFH}, \cite{OS} and the  Waldspurger type formula (\ref{equwald}), we obtain the result.
\end{proof}

When $\Delta>0$ and $\kappa$ is odd, we assume that $\Delta$  satisfies Assumption A.  We can prove the following result.

\begin{theorem}\label{mainthe}
For any $f \in H_{3/2-\kappa, \bar{\tilde{\rho}}_L}$, the global height is given by
\begin{equation}
\langle  Z_{\Delta, r, \kappa}(f),  Z_{\Delta, r, \kappa}(U)\rangle
=\frac{2\sqrt{N|\Delta|}\Gamma(\kappa-\frac{1}{2})}{(4\pi)^{\kappa-1}\pi^{\frac{3}{2}}}L'(Sh_{m_0, \mu_0}(\xi_{3/2-\kappa}f), \chi_\Delta, \kappa).
\end{equation}
When $\Delta=1$ and $\kappa=1$, it should plus a constant term as $$L'(Sh_{m_0, \mu_0}(\xi_{3/2-\kappa}f), \chi_\Delta, \kappa)+c^+(0, 0)k(0, 0).$$
\end{theorem}
\begin{proof}
When $\Delta=1$, it has been proved in  \cite[Theorem 7.14]{BY} and \cite[Theorem 6.5]{BEY}.

Now we assume that $\Delta\neq 1$.

According to the definition given in  (\ref{definter}), we have
\begin{eqnarray}\label{equneronfirst}
&&\langle  Z_{\Delta, r, \kappa}(f),  Z_{\Delta, r, \kappa}(m_0, \mu_0)\rangle\\
&=&(-1)^\kappa\langle \widehat{\mathcal Z}_{\Delta, r, \kappa}(f), \widehat{\mathcal Z}_{\Delta, r, \kappa}(m_0, \mu_0)\rangle_{GS}\nonumber\\
&=&\langle Z_{\Delta, r, \kappa}(f),  Z_{\Delta, r, \kappa}(U)\rangle_{\infty}+\langle Z_{\Delta, r, \kappa}(f),  Z_{\Delta, r, \kappa}(U)\rangle_{fin}.\nonumber
\end{eqnarray}

We  let $M$ denote the least common multiple of the discriminant associated with  the Heegner divisors  present in the support of the divisor $Z_{\Delta, r, \kappa}(f)$.
According to Lemma \ref{lemexist}, we can select a pair $(m_0, \mu_0)$ such that $D=-\sgn(\Delta)4Nm_0$ is coprime to $2MN\Delta$, $b(m_0, \mu_0)\neq 0$, and all prime factors of $N$  split in $\Q(\sqrt{D})$.

Now the can split the proof into two cases:

\textbf{Case 1: finite intersection}

We assume that $D_0=-4Nm_0|\Delta|=\Delta D$ is a fundamental discriminant.
By Theorem \ref{lemfininter}, it follows that
\begin{equation}\label{equassume}
\langle Z_{\Delta, r, \kappa}(f),  Z_{\Delta, r, \kappa}(U)\rangle_{fin}=(-1)^\kappa\langle \mathcal Z_{\Delta, r, \kappa}(f), \mathcal Z_{\Delta, r, \kappa}(m_0, \mu_0)\rangle_{fin}=0.
\end{equation}

\textbf{Case 2: infinite intersection }

According to equation (\ref{equarichi}), the archimedean intersection number is given by
\begin{eqnarray}
&&\langle Z_{\Delta, r, \kappa}(f),  Z_{\Delta, r, \kappa}(U)\rangle_{\infty}=\frac{1}{2}G_{N, \kappa}(Z_{\Delta, r, \kappa}(f), Z_{\Delta, r, \kappa}(U)).
\end{eqnarray}
By Proposition \ref{proarich} and
 Theorem \ref{theotwistedvalue}, we have
 \begin{eqnarray}
&&\langle Z_{\Delta, r, \kappa}(f),  Z_{\Delta, r, \kappa}(U)\rangle_{\infty}\\
&=&(-1)^{\kappa}\frac{(m_0)^{\frac{\kappa-1}{2}}}{2}\Phi^j_{\Delta, r}(Z(U), f)\nonumber\\
&=&\frac{2\Gamma(\kappa- \frac{1}{2})\sqrt{N|\Delta|}L'(Sh_{m_0,\mu_{0}}(\xi_{\frac{3}{2}-\kappa}f), \chi_{\Delta}, \kappa)}{(4\pi)^{\kappa-1}\pi^{\frac{3}{2}}}.\nonumber
\end{eqnarray}
Combining it with equations (\ref{equneronfirst}) and (\ref{equassume}), we obtain that
\begin{eqnarray}
&&\langle Z_{\Delta, r, \kappa}(f),  Z_{\Delta, r, \kappa}(U)\rangle\\
&=&\frac{2\Gamma(\kappa- \frac{1}{2})\sqrt{N|\Delta|}L'(Sh_{m_0,\mu_{0}}(\xi_{\frac{3}{2}-\kappa}f), \chi_{\Delta}, \kappa)}{(4\pi)^{\kappa-1}\pi^{\frac{3}{2}}}.\nonumber
\end{eqnarray}
Thus we finish the proof.
\end{proof}
Now we have the following result without the Assumption A.
\begin{theorem}
When $\kappa$ is an odd integer,
there  are infinitely many cycles $Z_{\Delta, r, \kappa}(U)= Z_{\Delta, r, \kappa}(m_0, \mu_0)$, such that
  \begin{equation}
\langle  Z_{\Delta, r, \kappa}(m, \mu),  Z_{\Delta, r, \kappa}(U)\rangle
=\frac{\sqrt{N|\Delta|}\Gamma(\kappa-\frac{1}{2})}{(4\pi)^{\kappa-1}\pi^{\frac{3}{2}}}L'(Sh_{m_0, \mu_0}(\xi_{3/2-\kappa}F_{m, \mu}), \chi_\Delta, \kappa).
\end{equation}
\end{theorem}
\begin{proof}

By proposition \ref{lemconzero}, there  are infinitely many cycles $Z_{\Delta, r, \kappa}(U)$, such that the constant term  $CT\big(\langle \phi(F_{m, \mu})^+, [\theta_{\mathcal P^\Delta}, \mathcal{E}_{\mathcal N^\Delta}]_j\rangle\big)=0.$

We have $Z_{\Delta, r, \kappa}(F_{m, \mu})=2Z_{\Delta, r, \kappa}(m, \mu)$.
Following the proof of Theorem \ref{mainthe}, we obtain the result.
\end{proof}
\section{Arithmetic inner product formula and the Gross- Zagier-Zhang formula}\label{Sec9}

\subsection{Gross-Zagier formula}
We define
\begin{equation}
y_{\Delta, r}(m, \mu)=Z_{\Delta,r }(m, \mu)-\deg(Z_{\Delta, r}(m, \mu))P_{\infty} \in J_0(N)(\Q(\sqrt{\Delta}))
\end{equation}
and
\begin{equation}y_{\Delta, r}(f)=\sum_{\mu, m> 0}c^+(-m, \mu)y_{\Delta, r}(m, \mu).\end{equation}
When $\Delta \neq 1$, $Z_{\Delta, r}(m, \mu)=y_{\Delta, r}(m, \mu)$ and $Z_{\Delta, r}(f)=y_{\Delta, r}(f)$.

For any divisor $D \in Div(X_0(N))$, let $\mathcal D$ be the Zariski closure in $\mathcal{X}_0(N)$ and $g_D$ be the associated $\nu$-admissible Green function.
Let $$i_{\infty}: J_0(N)\rightarrow \widehat{\CH}^1_\R(\mathcal{X}_0(N))$$be given by $i_{\infty}(D)=(\mathcal{D}+\Phi(D), g_D)$, where $\Phi(D)$ is the vertical divisor such that $\mathcal D+\Phi(D)$ has degree zero on every irreducible component of vertical fiber of $\mathcal{X}_0(N)$.
The following result provides a method to compute the N\'eron-Tate height by Arakelov arithmetic intersection.
\begin{theorem}[Faltings-Hailjac]
\begin{equation}
-\langle D_1 , D_2\rangle_{NT}=\langle i_{\infty}(D_1), i_{\infty} (D_2)\rangle_{GS}.
\end{equation}
Here $\langle~, ~\rangle_{NT}$ is N\'eron-Tate height.
\end{theorem}

\begin{lemma}\label{lemmavertzero}
The divisor $y_{\Delta, r}(n, \mu)$ has degree zero on every irreducible component of vertical fiber of $\mathcal{X}_0(N)$.
\end{lemma}
Let $G \in S_{2}^{new}(N)$ be a normalized newforms and let $y_{\Delta, r}^{G}(m, \mu)$ denote the projection
of the $y_{\Delta, r}(m, \mu)$ onto its $G$-isotypical component.
Now we recall the Gross-Zagier formula as follows.
\begin{theorem}\cite[Theorem 6.3]{GZ}
\begin{equation}\label{equgz}
\langle y_{\Delta, r}^G(m, \mu), y_{\Delta, r}^G(m, \mu)\rangle_{NT}=
\frac{\sqrt{|D|} }{4\pi^{2}\parallel G\parallel^2}
L'(G, \chi, 1),
\end{equation}
where  $D=-4N|\Delta| m$ is a fundamental discriminant.
\end{theorem}

\subsection{Arithmetic inner product formula}
It has been proved in  \cite[Section 6]{BO} that the generating function
\begin{equation}
A_{\Delta, r}(\tau)=\sum_{m>0, \mu}y_{\Delta, r}(m, \mu)q^me_\mu,
\end{equation}
is a cusp form of weight $3/2$  with respect to $\tilde{\rho}_L$.
According to Proposition \ref{Prothedeco}.
this function is the generic component  of  $\widehat{\phi}_{MW}.$

We identify
 $$ \widehat{\phi}_{MW}(g)=\langle A_{\Delta, r}, g \rangle_{Pet}.$$
%
%

\begin{lemma}\label{Alift}
For any $f \in H_{\frac{1}{2}, \bar{\tilde{\rho}}_L}$, we have
\begin{equation}
\widehat{\phi}_{MW}(\xi_{\frac{1}{2}}(f) )=y_{\Delta, r}(f).
\end{equation}
\end{lemma}
\begin{proof}
For any  $f \in H_{\frac{1}{2}, \bar{\tilde{\rho}}_L}$, $\langle A_{\Delta, r}(\tau), f \rangle d\tau $ is invariant under the action of $\Gamma'$. Therefore, we have
\begin{eqnarray}\label{equcom}
&&d(\langle A_{\Delta, r}(\tau), f \rangle d\tau )=
\bar\partial(\langle A_{\Delta, r}(\tau), f \rangle d\tau )\\
&=&\langle \bar\partial A_{\Delta, r}(\tau), f \rangle d\bar\tau d\tau +\langle A_{\Delta, r}(\tau), \overline{\partial} f \rangle d\bar\tau d\tau \nonumber\\
&=&\langle A_{\Delta, r}(\tau), \overline{\partial} f \rangle d\bar\tau d\tau.\nonumber
\end{eqnarray}
It is known that $$\langle A_{\Delta, r}(\tau), \overline{\partial} f \rangle d\bar\tau d\tau
 =-\langle A_{\Delta, r}(\tau), L_{\frac{1}{2}} f \rangle \mu(\tau).$$
Combining it with equation (\ref{equcom}), we have
\begin{eqnarray}
&&\widehat{\phi}_{MW}(\xi_{\frac{1}{2}}(f))=\langle A_{\Delta, r},  \xi_{\frac{1}{2}}(f) \rangle_{Pet}\nonumber\\
&=&\lim_{t\to \infty}\int_{\mathcal{F}_t} \langle A_{\Delta, r}(\tau), \overline{\xi_{\frac{1}{2}}(f)} \rangle v^{\frac{3}{2}}\mu(\tau)
=-\lim_{t\to \infty}\int_{\mathcal{F}_t}d(\langle A_{\Delta, r}(\tau), f \rangle d\tau) \nonumber\\
&=&-\lim_{t\to \infty}\int_{\partial\mathcal{F}_t}\langle A_{\Delta, r}(\tau), f \rangle d\tau
=\lim_{t\to \infty}\int_{-\frac{1}{2}}^{\frac{1}{2}}\langle A_{\Delta, r}(u+it), f(u+it) \rangle du \nonumber\\
&=&\sum_{\mu, m> 0}c^+(-m, \mu)y_{\Delta, r}(m, \mu)=y_{\Delta, r}(f).\nonumber
\end{eqnarray}
Thus we obtain the result.
\end{proof}

We choose  a  newform $g \in S^{new}_{\frac{3}{2}, \tilde{\rho}_L}$ that corresponds to $G$ under the Shimura lift.
Let $f \in H_{\frac{1}{2}, \bar{\tilde{\rho}}_L}$ be given as $\xi_{\frac{1}{2}}(f)=\parallel g\parallel^{-2} g .$

It is known in \cite[Theorem 7.7]{BO} that
\begin{equation}\label{equproj}
A^G_{\Delta, r}(\tau)=g(\tau)\otimes y_{\Delta, r}(f),
\end{equation}
where $A^G_{\Delta, r}(\tau)=\sum_{m, \mu}y_{\Delta, r}^{G}(m, \mu)q^ne_\mu$.

\begin{theorem}[Arithmetic inner product formula]\label{theoariinner}
Let the notation be as above.  Then
\begin{eqnarray}
\langle \widehat{\phi}_{MW}(g), \widehat{\phi}_{MW}(g) \rangle_{NT}=\frac{2\parallel g\parallel^{2}
\sqrt{N\mid \Delta\mid}}{\pi}L'(G, \chi_\Delta, 1).\nonumber
\end{eqnarray}
When $N$ is square free,
\begin{eqnarray}\label{formulaarith}
\langle \widehat{\theta}_{\Delta, r}(g), \widehat{\theta}_{\Delta, r}(g) \rangle_{NT}=\frac{2\parallel g\parallel^{2}\sqrt{N\mid \Delta\mid}}{\pi}L'(G, \chi_\Delta, 1).\nonumber
\end{eqnarray}
\end{theorem}
\begin{proof}

Following Theorem \ref{Lemnonsplit} and Lemma \ref{lemexist}, we can find a pair $(m_0, \mu_0)$, such that,
\begin{enumerate}
  \item $b(m_0, \mu_0)\neq 0$, $D=-4N|\Delta|m_0$ is odd;
  \item  the finite intersection $\langle \widehat{\mathcal Z}_{\Delta, r}(f), \widehat{\mathcal Z}_{\Delta, r}(m_0, \mu_0)\rangle_{fin}=0$.
\end{enumerate}

From equation (\ref{equproj}), we have
\begin{equation}
b(m_0, \mu_0)  y_{\Delta, r}(f)=  y_{\Delta, r}^{G}(m_0, \mu_0).
\end{equation}
$\mathbf{ method~ 1}$
According to Lemma \ref{Alift}, we derive the following equation
\begin{equation}
\langle \widehat{\phi}_{MW}(\xi_{\frac{1}{2}}(f)), \widehat{\phi}_{MW}(\xi_{\frac{1}{2}}(f)) \rangle_{NT}=\langle y_{\Delta, r}(f), y_{\Delta, r}(f) \rangle_{NT}.
\end{equation}
Thus, we obtain 
\begin{eqnarray}
&&|b(m_0, \mu_0)|^2\langle \widehat{\phi}_{MW}(g), \widehat{\phi}_{MW}(g) \rangle_{NT}\\
&=&|b(m_0, \mu_0)|^2\parallel g\parallel^4\langle y_{\Delta, r}(f),  y_{\Delta, r}(f) \rangle_{NT}\nonumber\\
&=&\parallel g\parallel^4\langle  y_{\Delta, r}^{G}(m_0, \mu_0),   y_{\Delta, r}^{G}(m_0, \mu_0) \rangle_{NT}.\nonumber
\end{eqnarray}
According to the Gross-Zagier formula (\ref{equgz}) and the following Waldspurger type formula (\ref{equwald}),
\begin{equation}\label{equwald}
|b(m, \mu)|^2=\frac{\Vert g \Vert^2\sqrt{\mid D_0\mid}}{8\pi\sqrt{ N}\Vert G \Vert^2}L(G, \chi_{D_0}, 1),
\end{equation}
we have
\begin{eqnarray}
&&\langle \widehat{\phi}_{MW}(g), \widehat{\phi}_{MW}(g) \rangle_{NT}\\
&=&\frac{2\parallel g\parallel^{2}
\sqrt{N\mid \Delta\mid}}{\pi}L'(G, \chi_\Delta, 1).\nonumber
\end{eqnarray}

By Theorem \ref{Theolift},  we have
\begin{eqnarray}
\widehat{ \theta}_{\Delta, r}(g)=\widehat{\phi}_{MW}(g).
\end{eqnarray}
Therefore, we  obtain the second equation.

$\mathbf{method ~2}$
We assume that $\Delta\neq 1$, and a similar argument can be applied when $\Delta=1$.

For the pair $(m_0, \mu_0)$, we have
\begin{eqnarray}
b(m_0, \mu_0)\langle \widehat{\mathcal Z}_{\Delta, r}(f), \widehat{\mathcal Z}_{\Delta, r}(f) \rangle_{GS}
&=&\langle \widehat{\mathcal Z}_{\Delta, r}(f), \widehat{\mathcal Z}_{\Delta, r}^G(m_0, \mu_0)\rangle_{GS}
\nonumber\\
&=&\langle \widehat{\mathcal Z}_{\Delta, r}(f), \widehat{\mathcal Z}_{\Delta, r}(m_0, \mu_0)\rangle_{\infty}.
\nonumber
\end{eqnarray}
 

According to Theorem \ref{mainthe}, we have
$$b(m_0, \mu_0)\langle \widehat{\mathcal Z}_{\Delta, r}(f), \widehat{\mathcal Z}_{\Delta, r}(f) \rangle_{GS}=-2b(m_0, \mu_0)\frac{\sqrt{N\mid \Delta\mid}}{\pi \parallel g\parallel^2}L'(G, \chi_\Delta, 1).$$
Then we have
\begin{eqnarray}\label{equkeylemma}
\langle \widehat{\phi}_{MW}(g), \widehat{\phi}_{MW}(g) \rangle_{NT}&=&-\langle \widehat{\phi}_{MW}(g), \widehat{\phi}_{MW}(g) \rangle_{GS}\nonumber\\
&=&-\parallel g\parallel^4\langle \widehat{\mathcal Z}_{\Delta, r}(f), \widehat{\mathcal Z}_{\Delta, r}(f) \rangle_{GS}\nonumber\\
&=&\frac{2\parallel g\parallel^{2}\sqrt{N\mid \Delta\mid}}{\pi}L'(G, \chi_\Delta, 1).\nonumber
\end{eqnarray}
 Similarly, we can prove the second equation.
\end{proof}

From this proof, we find the following relation.
\begin{corollary}
$$
\text{Arithmetic inner product formula $\Leftrightarrow$ The Gross-Zagier formula.}
$$
\end{corollary}
\begin{remark}
The first proof is valid for all $\Delta$, whereas the second method relies on the assumption that $\Delta$ possesses a prime factor $p$ such that $p\equiv 3 (\mod)$ when $\Delta>0$.
We select a negative odd fundamental discriminant $D$  that contains such a prime factor. 
\end{remark}

\subsection{Gross-Zagier-Zhang formula}

Let $G$ be a newform  of weight $2\kappa$. There exists a newform  $g \in S^{new}_{\frac{1}{2}+\kappa, \tilde{\rho}_L}$ that corresponds to $G$ under the Shimura
correspondence and there exists a function
$f \in H_{_{3/2-\kappa}, \bar{\tilde{\rho}}_L}$ such that
$$\xi_{3/2-\kappa}(f)=\parallel g\parallel^{-2} g .$$
Taking the $G$-component,
by the multiplicity one, we have the following conjecture.
\begin{conjecture}[Modularity]
The following generating function is a cusp form,
\begin{equation}\label{conone}
\sum_{m>0, \mu}Z_{\Delta, r, \kappa}^G(m, \mu)q^me_\mu=g(\tau)\otimes Z_{\Delta, r, \kappa}(f).
\end{equation}
More precisely, \begin{equation}
Z_{\Delta, r, \kappa}^G(m, \mu)=b(m, \mu)Z_{\Delta, r, \kappa}(f),
\end{equation}
where $b(m, \mu)$ are Fourier coefficients of $g$.
\end{conjecture}
When $\kappa =1$, this has been  proved by Gross, Kohnen and Zagier \cite{GKZ}, Borcherds  \cite{BoDuke},  and Bruinier and Ono \cite{BO}.
%

The Waldspurger type formula for the $L$-function in \cite[Chapter 2]{GKZ}  is given by
\begin{equation}\label{equwald}
|b(m, \mu)|^2=\frac{(\kappa-1)!\Vert g \Vert^2\mid D_0\mid^{\kappa-1/2}}{2^{2\kappa+1}\pi^{\kappa} N^{\kappa-1/2}\Vert G \Vert^2}L(G, \chi_{D_0}, \kappa),
\end{equation}
where $m=-\frac{\sgn(\Delta){D_0}}{4N}$ and $m \equiv \sgn(\Delta)Q(\mu)(\mod \Z)$.

The following result reveals the relationship between Theorem \ref{mainthe} and the higher weight Gross-Zagier-Zhang formula \cite{Zhang}.
\begin{theorem}
If  Conjecture \ref{conone} is true, then the following formula holds
\begin{equation}
\langle Z_{\Delta, r, \kappa}^G(m, \mu), Z_{\Delta, r, \kappa}^G(m, \mu)\rangle=
\frac{(2\kappa-2)!\sqrt{|D|} }{2^{4\kappa-2}\pi^{2\kappa}\parallel G\parallel^2}m^{\kappa-1}
L'(G, \chi, \kappa),
\end{equation}
where  $D=-4N|\Delta| m$ is a fundamental discriminant.
\end{theorem}
\begin{proof}
The proof follows a similar approach of Theorem \ref{theoariinner}.
Assuming the validity of modularity  conjecture \ref{conone}, we have
\begin{equation}
\langle Z_{\Delta, r, \kappa}(f), Z_{\Delta, r, \kappa}(f)\rangle=
\frac{2\Gamma(\kappa- \frac{1}{2})\sqrt{N|\Delta|}}{(4\pi)^{\kappa-1}\pi^{\frac{3}{2}}\parallel g\parallel^2}
L'_K(G, \chi_\Delta, \kappa).
\end{equation}
If 
$
Z_{\Delta, r, \kappa}^G(m, \mu)=b(m, \mu)Z_{\Delta, r, \kappa}(f),
$
then we have 
\begin{equation}
\langle Z_{\Delta, r, \kappa}^G(m, \mu), Z_{\Delta, r, \kappa}^G(m, \mu)\rangle=|b(m, \mu)|^2\langle Z_{\Delta, r, \kappa}(f), Z_{\Delta, r, \kappa}(f)\rangle.
\end{equation}

By the Waldspurger  formula (\ref{equwald}), we obtain the result.
\end{proof}

\section*{Acknowledgements}We would like to thank T.H. Yang for  suggestion to do this research, and  for helping us to check some details of the proof.


\end{document}